\documentclass[12pt]{amsart}

\usepackage[a4paper,left=2.5cm,right=2.5cm,top=2.5cm,bottom=2.5cm]{geometry}
\usepackage{amsmath}
\newcommand\numberthis{\addtocounter{equation}{1}\tag{\theequation}}
\newcommand{\un}{\underline}
\usepackage{graphicx}
\usepackage{amssymb}
\usepackage{mathabx}
\usepackage{bm}
\usepackage{enumerate}
\usepackage{setspace}
\usepackage{color}
\usepackage[colorinlistoftodos,textwidth=2cm]{todonotes}
\usepackage{tikz}
\usepackage{cases}
\usetikzlibrary{patterns,patterns.meta}

\newtheorem{theorem}{Theorem}[section]
\newtheorem{lemma}[theorem]{Lemma}
\newtheorem{proposition}[theorem]{Proposition}
\newtheorem{corollary}[theorem]{Corollary}

\newtheorem{remark}[theorem]{Remark}
\newtheorem{example}[theorem]{Example}
\numberwithin{equation}{section}

\begin{document}

\title{Furstenberg sumset conjecture and Mandelbrot percolations}

\author{Catherine Bruce}
\address{Department of Mathematics, University of Manchester, Oxford Road, Manchester M13 9PL, United Kingdom}
\email{catherine.bruce@manchester.ac.uk}

\author{Xiong Jin}
\address{Department of Mathematics, University of Manchester, Oxford Road, Manchester M13 9PL, United Kingdom}
\email{xiong.jin@manchester.ac.uk}

\begin{abstract}
In this paper we extend Hochman and Shmerkin's projection theorem to product measures of Mandelbrot cascades acting on ergodic measures imaged through canonical mappings of one-dimensional iterated function systems without any separation conditions. Consequently we extend Furstenberg sumset theorem to images of subshifts on symbolic spaces, and to Mandelbrot percolations on invariant sets. We also obtain dimension results for convolutions of Bernoulli convolutions and that of Mandelbrot cascade measures.
\end{abstract}

\maketitle

\section{Introduction}
In their seminal paper \cite{HS12} Hochman and Shmerkin proved the Furstenberg's sumset conjecture, that for two closed subsets $X,Y$ of the unit circle $\mathbb{T}=[0,1] \mod 1$ that are $\times 2$ and $\times 3$ invariant respectively, their summation has the largest possible dimension. More precisely, they showed that for all $s\neq 0$,
\begin{equation}\label{FSC}
\dim (X+sY)=\min\{1,\dim X+\dim Y\},
\end{equation}
where $X+sY=\{x+sy:x\in X,y\in Y\}$, and for a set $E$ we denote by $\dim E$ the exact dimension of $E$ when the box-counting and Hausdorff dimension of $E$ are equal, i.e., $\dim_BE=\dim_HE=\dim E$.

This formula is derived from a more general result on the dimension of projections of product of invariant measures. Here for future reference we shall use symbolic dynamical systems to introduce these invariant measures: For an integer $a\ge 2$ let $\Sigma^+_a=\{1,\cdots,a\}^\mathbb{N}$ denote the one-sided symbolic space of $a$-letters. Let
\[
\bm{I}_a=\{f_i(x)=a^{-1}x+(i-1)a^{-1}\}_{i=1}^a
\]
be the iterated function system (IFS) that represents the $a$-tilling of the unit interval $[0,1]$. For an infinite word $\bm{i}^+=i_1i_2\cdots\in\Sigma_a^+$ define
\[
\Phi_{\bm{I}_a}(\bm{i}^+)=\lim_{n\to \infty} f_{i_1}\circ\cdots\circ f_{i_n}(0).
\]
Then $\Phi_{\bm{I}_a}$ forms a canonical mapping from $\Sigma_a^+$ to $[0,1]$. Let $\sigma: i_1i_2\cdots \to i_2i_3\cdots$ denote the left-shift operator on $\Sigma_a^+$. Let $\mu^+$ be a $\sigma$-invariant  probability measure on $\Sigma_a^+$. Then the push-forward measure
\[
\Phi_{\bm{I}_a}(\mu^+):=\mu^+\circ \Phi_{\bm{I}_a}^{-1}
\]
is a $\times a$-invariant measure on $\mathbb{T}$.

Consider another integer $b\ge 2$ and its corresponding symbolic space $\Sigma_b^+$, let $\Phi_{\bm{I}_b}:\Sigma_b^+\to[0,1]$ be its associated canonical mapping. Let $\nu^+$ be a $\sigma$-invariant probability measure on $\Sigma_b^+$.  Let $\Pi$ be the family of orthogonal projections from $\mathbb{R}^2$ to its $1$-dimensional subspaces. Hochman and Shmerkin \cite[Theorem 1.3]{HS12} proved that if $\log a/\log b$ is irrational, then
\begin{equation}\label{HST}
\underline{\dim}_H \pi(\Phi_{\bm{I}_a}(\mu^+)\times \Phi_{\bm{I}_b}(\nu^+))=\min\{1,\underline{\dim}_H \Phi_{\bm{I}_a}(\mu^+)\times \Phi_{\bm{I}_b}(\nu^+)\}
\end{equation}
for all $\pi\in\Pi\setminus\{\pi_1,\pi_2\}$, where $\pi_1$, $\pi_2$ denote the horizontal and vertical projections, and $\underline{\dim}_H \lambda=\inf \{\dim_H E: \lambda(E)>0\}$ denotes the lower Hausdorff dimension of a measure $\lambda$. By using the variational principle one can easily get \eqref{FSC} from \eqref{HST}.

The proof of \eqref{HST} used the CP-chains method introduced by Furstenberg \cite{Furst2}, and the local entropy averages and $\rho$-tree method introduced by Hochman and Shmerkin \cite{HS12}. Such a mechanism turns out to be very useful to derive lower bounds of Hausdorff dimension of dynamically driven fractal measures. In particular it led Shmerkin \cite{Shmer} and Wu \cite{Wu} to solve, independently, another conjecture of Furstenberg on intersections of invariant sets.

\medskip

Obviously the $a$-tilling $\bm{I}_a$ is a special case of contracting IFSs on $\mathbb{R}$ and it is natural to ask if one can extend \eqref{HST} to more general IFSs. In their paper Hochman and Shmerkin also proved \eqref{HST} for a class of more general IFSs satisfying the open set condition (OSC), but only for less general quasi-Bernoulli measures. Here we shall consider this problem for a class of general IFSs without any separation conditions, and for arbitrary invariant measures. In particular we shall consider IFSs $\bm{I}$ and $\bm{J}$ of the form
\[
\bm{I}=\{\delta x+t_i\}_{i=1}^a;\\ \bm{J}=\{\rho x+s_j\}_{j=1}^b
\]
for $\delta,\rho\in(0,1)$ and $t_i,s_j\in\mathbb{R}$. Let $\Phi_{\bm{I}}$ and $\Phi_{\bm{J}}$ be their canonical mappings respectively (see Section \ref{ifsed} for definition). Our first result is the following.

\begin{theorem}\label{thm1}
Let $\mu^+$ and $\nu^+$ be two invariant measures on $(\Sigma^+_a,\sigma)$ and $(\Sigma^+_b,\sigma)$ respectively. If $\log \delta/\log \rho$ is irrational, then
\begin{equation}\label{e1}
\underline{\dim}_H \pi(\Phi_{\bm{I}}(\mu^+)\times \Phi_{\bm{J}}(\nu^+))=\min\{1,\underline{\dim}_H \Phi_{\bm{I}}(\mu^+)\times \Phi_{\bm{J}}(\nu^+)\}
\end{equation}
for all $\pi\in\Pi\setminus\{\pi_1,\pi_2\}$.
\end{theorem}

\begin{remark}
Note that in this projection theorem the relevant parameters are the contraction ratios $\delta$ and $\rho$, not the mapping counting numbers $a$ and $b$. A more general setting would be IFSs with arbitrary contraction ratios. We believe that a similar result would hold in such a setting, but it would involve more complicated conditions rather than just $\log \delta/\log \rho$ being irrational. One can imagine an example where the invariant measures $\mu^+$ and $\nu^+$ whose dimensions are reached on some subshifts of $\Sigma_a^+$ and $\Sigma_b^+$ respectively, and the contraction ratios of the IFSs restricted on these subshifts are actually log-rationally dependent.  
\end{remark}

As Theorem \ref{thm1} applies to any invariant measures, we may use the variational principle (see \cite[Theorem 2.13]{FH09} for example) to obtain an analogous projection theorem for invariant sets. In particular we have the following extension of Furstenburg's sumset theorem: 
\begin{corollary}\label{co1}
Let $X$ and $Y$ be two closed $\sigma$-invariant sets (subshifts) on $\Sigma_a^+$ and $\Sigma_b^+$ respectively. If $\log \delta/\log \rho$ is irrational, then
\begin{equation*}\label{e2}
\begin{split}
\dim \pi(\Phi_{\bm{I}}(X)\times \Phi_{\bm{J}}(Y))=&\ \min\{1,\dim \Phi_{\bm{I}}(X)\times \Phi_{\bm{J}}(Y)\} \\
=&\ \min\{1,\dim \Phi_{\bm{I}}(X)+\dim \Phi_{\bm{J}}(Y)\}
\end{split}
\end{equation*}
for all $\pi\in\Pi\setminus\{\pi_1,\pi_2\}$.
\end{corollary}

\begin{remark}
This corollary can also be deduced from \cite{HS12} by constructing invariant measures carried by sub-subshifts on which the IFSs satisfy OSC that approach the box-counting dimension of these image sets from below. The slightly new thing here is that, as Theorem \ref{thm1} applies to all invariant measures, the above dimension is actually reached by two invariant measures carried by $X$ and $Y$ respectively, namely the variational principle holds in this situation.
\end{remark}

Another application of Theorem \ref{thm1} is to Bernoulli convolutions: For $\beta\in (0,1)$ consider the IFS consisting of two maps $\{\beta x -1,\beta x+1\}$. Let $\Phi_\beta$ denote its canonical mapping from $\Sigma_2^+$ to $\mathbb{R}$. For $p\in(0,1)$ let $\mu_p^+$ denote the $p$-Bernoulli measure on $\Sigma_2^+$. The push-forward measure $\mu^+_{\beta,p}=\Phi_\beta(\mu^+_p)$ is called a $(\beta,p)$-Bernoulli convolution.

A measure $\lambda$ on a metric space $(E,d)$ is said to be exact-dimensional with dimension $D\ge 0$ if for $\lambda$-a.e. $x\in E$,
\begin{equation}\label{eddef}
\lim_{r\to0} \frac{\log \lambda(B_d(x,r))}{\log r}=D,
\end{equation}
where $B_d(x,r)$ stands for the closed ball in $(E,d)$ centred at $x$ with radius $r$. In this case we write $\dim \lambda=D$.

As used in \cite{HS12} to reprove the Rudolph-Johnson theorem, the projection of a product measure $\mu\times \nu$ on $\mathbb{R}^2$ to the line $y=x$ is equal to the convolution $\mu*\nu$, up to an affine map. Hence we have the following corollary from Theorem \ref{thm1}.

\begin{corollary}\label{cobc}
If $\log\beta/\log \beta'$ is irrational, then 
\[
\dim \mu^+_{\beta,p}*\nu^+_{\beta',p'}=\min\{1,\dim \mu^+_{\beta,p}+\dim \nu^+_{\beta',p'}\}.
\]
\end{corollary}

\begin{remark}
The case of $\beta,\beta'<1/2$ is covered by \cite{HS12} as in this case the IFSs satisfy the OSC. It is worth noting that the above statement still holds when replacing the Bernoulli measures $\mu_p^+$ and $\nu_{p'}^+$ with arbitrary ergodic measures $\mu^+$ and $\nu^+$ on $(\Sigma_a^+,\sigma)$ and $(\Sigma_b^+,\sigma)$ respectively.
\end{remark}

\medskip

Similarly as in \cite{FJ14}, the method of firstly constructing CP-chains on symbolic spaces turns out to be robust enough to extend \eqref{e1} to more general fractal measures, in particular to Mandelbrot cascades acting on ergodic measures: For an ergodic measure $\mu^+$ on $\Sigma_a^+$ and $k\ge 1$ let
\[
\widetilde{\mu}^+_k(\mathrm{d}\bm{i}^+)=V_{i_1}V_{i_1i_2}\cdots V_{i_1\cdots i_k} \mu^+(\mathrm{d}\bm{i}^+),\ \bm{i}^+=i_1i_2\cdots\in\Sigma_a^+,
\]
where $\{V_u: u\in \Sigma_a^{+,*}\}$ is a sequence of independently and identically distributed (i.i.d.) random variables with common law $V$, indexed by the set $\Sigma_a^{+,*}$ of all finite words in $\Sigma_a^+$ (see Section \ref{sds} for definition). We assume that
\[
V\ge 0, \ \mathbb{E}(V)=1 \text{  and  } h_V:=\mathbb{E}(V\log V)< h_\mu,
\]
where $h_\mu$ denotes the the measure-theoretic entropy of $\mu^+$ (see \eqref{mtemu} for definition). Under such an assumption, by the recent work of Barral and Jin \cite[Theorem 2.3]{BJ21} we have that the measure-valued martingale $\{\widetilde{\mu}^+_k\}_{k\ge 1}$ almost surely converges to a non-trivial limit $\widetilde{\mu}^+$. We call the limiting measure $\widetilde{\mu}^+$ a Mandelbrot cascade measure with respect to $(V,\mu^+)$.

Mandelbrot cascade measures (w.r.t. uniform measures) were introduced by Mandelbrot \cite{Man74} to simulate the energy dissipation phenomenon in fully developed turbulences, and were developed by Kahane and Peyri\`ere \cite{KP76}. When the random weights $\{V_u: u\in \Sigma_a^{+,*}\}$ have log-normal distribution, these random measures are considered as a simplified model of the Gaussian multiplicative chaos built by Kahane \cite{K85} (motivated by a model introduce by Mandelbrot \cite{Man72}), which has been playing a very important role in statistical physics and quantum gravity, see for example \cite{DuSh,KRV,BeWeWo}.

The projection theorem of Mandelbrot cascade measures on self-similar sets obtained in \cite{FJ14} was used in \cite{FJ15} to study the dimension of fibres. In particular in \cite{FJ15} a weak dimension conservation for planar self-similar sets with dense rotations was obtained. This motives us to extend Hochman and Shmerkin's projection theorem \eqref{HST} to Mandelbrot cascades, as a step towards studying the dimension of intersection of invariant sets.

\medskip

As an extension of Theorem \ref{thm1}, and as the main result of this paper, we have the following projection theorem for images of Mandelbrot cascades acting on ergodic measures through the canonical mappings $\Phi_{\bm{I}}$ and $\Phi_{\bm{J}}$ (the exact-dimensionality of these measures, which is also a highly non-trivial extension of the previous results in \cite{FH09,FJ14}, is proved in Theorem \ref{ed22}).

\begin{theorem}\label{mthm}
Let $\mu^+$ and $\nu^+$ be two ergodic measures on $(\Sigma^+_a,\sigma)$ and $(\Sigma^+_b,\sigma)$ respectively. Let $\widetilde{\mu}^+$ and $\widetilde{\nu}^+$ be the Mandelbrot cascade measures with respect to $(V,\mu^+)$ and $(W,\nu^+)$ respectively.
Assume that $\log \delta/\log \rho$ is irrational, and that $V$ and $W$ are independent. Then almost surely conditioned on non-degeneracy, i.e., $\widetilde{\mu}^+(\Sigma_a^+)\cdot\widetilde{\nu}^+(\Sigma_b^+)>0$,
\begin{equation*}\label{mlb}
\begin{split}
\dim \pi(\Phi_{\bm{I}}(\widetilde{\mu}^+)\times \Phi_{\bm{J}}(\widetilde{\nu}^+))=&\ \min\{1,\dim \Phi_{\bm{I}}(\widetilde{\mu}^+)\times \Phi_{\bm{J}}(\widetilde{\nu}^+)\}\\
=&\ \min\{1,\dim \Phi_{\bm{I}}(\widetilde{\mu}^+)+\dim \Phi_{\bm{J}}(\widetilde{\nu}^+)\}
\end{split}
\end{equation*}
for all $\pi\in\Pi\setminus\{\pi_1,\pi_2\}$.
\end{theorem}

\begin{remark}
Note that Theorem \ref{mthm} includes the deterministic case for ergodic measures. One may simply take $V=W\equiv 1$.
\end{remark}

When we take the uniform (Lebesgue) measures $\ell_a$ and $\ell_b$ on $\Sigma_a^+$ and $\Sigma_b^+$, and the tilling IFSs $\bm{I}_a$ and $\bm{I}_b$, the random measures $\widetilde{M}_{V,a}:=\Phi_{\bm{I}_a}(\widetilde{\ell}_a)$ and $\widetilde{M}_{W,b}:=\Phi_{\bm{I}_b}(\widetilde{\ell}_b)$ become the classical Mandelbrot cascade measures \cite{KP76} defined on the unit interval $[0,1]$. As a corollary of Theorem \ref{mthm} we have the following result about the convolution of $\widetilde{M}_{V,a}$ and $\widetilde{M}_{W,b}$.

\begin{corollary}
Assume that $V$ and $W$ are independent. If $\log a/\log b$ is irrational, then almost surely conditioned on $\widetilde{M}_{V,a}([0,1])\cdot\widetilde{M}_{W,b}([0,1])>0$,
\begin{equation*}\label{mcc}
\begin{split}
\dim \widetilde{M}_{V,a}*\widetilde{M}_{W,b}=&\ \min\{1,\dim \widetilde{M}_{V,a}+\dim \widetilde{M}_{W,b}\}\\
=&\ \min\Big\{1,2-\frac{h_V}{\log a}-\frac{h_W}{\log b}\Big\}.
\end{split}
\end{equation*}
\end{corollary}

One particular interesting example of Mandelbrot cascades is the case when the random weights are Bernoulli, with our normalisation that is
\[
V=\left\{\begin{array}{ll} 1/p, & \text{with probability } p;\\ 0, & \text{with probability } 1-p\end{array}  \right.
\]
for some $p\in(0,1]$. This case corresponds to the $p$-Mandelbrot percolation on $\Sigma_a^+$:
\[
\Sigma_{a,p}^+=\{\bm{i}^+\in\Sigma_a^+: V_{i_1}V_{i_1i_2}\cdots V_{i_1\cdots i_n} >0 \text{ for all } n\ge 1\}.
\]
One can view $\Sigma_{a,p}^+$ as the boundary of a Galton-Watson branching tree with offspring number $N=\#\{i=1,\ldots,a: V_i>0\}$. This is a classical statistically self-similar fractal set, and its canonical image to Euclidean spaces shares many intersection properties with the sample path of Brownian motion \cite{Per96}. 

Here we shall consider Mandelbrot percolations on a subshift $X$ of $\Sigma_a^+$:
\[
X_p=X\cap \Sigma^+_{a,p}.
\]
We would like to apply Theorem \ref{mthm} to obtain a similar result to Corollary \ref{co1} for the images of these random sets through canonical mappings. But we reach a complication from the possible overlaps in the IFSs: if $\bm{I}$ has overlaps, then for a cylinder $[u]\subset \Sigma_a^+$ (see Section \ref{sds} for definition), even if $[u]\cap X=\emptyset$, this does not necessarily imply that $\Phi_{\bm{I}}(X_p)\cap \Phi_{\bm{I}}([u])=\emptyset$, which makes it difficult to obtain a sharp upper bound of the dimension of $\Phi_{\bm{I}}(X_p)$ (see Example \ref{example1}).

To overcome this difficulty we shall introduce the following quantity: for $X\subset \Sigma_a^+$ and $n\ge 1$ let
\[
t_{X,\bm{I},n}=\sup_{x\in \mathbb{R}}\#\{u \in X_n: \Phi_{\bm{I}}([u])\cap B(x,\delta^n)\neq \emptyset\},
\]
where we denote by $X_n=\{u\in \{1,\cdots,a\}^n:[u]\cap X\neq \emptyset\}$ the set of admissible words of length $n$ in $X$. Let
\[
\gamma_{X,\bm{I}}=\limsup_{n\to\infty} \frac{\log t_{X,\bm{I},n}}{-n\log \delta}.
\]
The condition $\gamma_{X,\bm{I}}=0$ is similar to the so-called asymptotically weak separation condition (AWSC) introduced by Feng \cite{Feng07} but it is slightly stronger as we are counting all different letters, whereas in AWSC one only counts all different composed maps; they are equivalent when the IFS has no exact overlaps, in particular when the IFS satisfies the exponential separation condition (ESC). By \cite[Theorem 1.3]{BF21} such an example with non-trivial overlaps can be found when $\delta\cdot a<1$.

If we assume that $\gamma_{X,\bm{I}}=0$, then it is not hard to show (see Corollary \ref{bdp}) that, almost surely conditioned on $X_p\neq \emptyset$,
\[
\dim \Phi_{\bm{I}}(X_p)= \dim \Phi_{\bm{I}}(X)-\frac{\log p}{\log \delta}.
\]
Then by using Theorem \ref{mthm} we obtain another extension of Furstenburg's sumset theorem.

\begin{corollary}\label{co2}
Let $X$ and $Y$ be two subshifts on $\Sigma_a^+$ and $\Sigma_b^+$ respectively. Suppose that $\gamma_{X,\bm{I}}=\gamma_{Y,\bm{J}}=0$. Let $X_p$ and $Y_{p'}$ be two independent Mandelbrot percolations on $X$ and $Y$ respectively. If $\log \delta/\log \rho$ is irrational, then almost surely conditioned on $X_p\neq \emptyset$ and $Y_{p'}\neq \emptyset$,
\begin{equation*}\label{e3}
\begin{split}
\dim \pi(\Phi_{\bm{I}}(X_p)\times \Phi_{\bm{J}}(Y_{p'}))=&\ \min\{1,\dim \Phi_{\bm{I}}(X_p)+\dim \Phi_{\bm{J}}(Y_{p'})\}\\
=&\ \min\left\{1,\dim \Phi_{\bm{I}}(X)+\dim \Phi_{\bm{J}}(Y)-\frac{\log p}{\log \delta}-\frac{\log p'}{\log \rho}\right\}
\end{split}
\end{equation*}
for all $\pi\in\Pi\setminus\{\pi_1,\pi_2\}$.
\end{corollary}

\begin{remark}
There have been extensive studies on the geometric properties of classical Mandelbrot percolations, see \cite{RS,SV,FJ15,SO,SS} for example. The projection result stated in Corollary \ref{co2} dealing with IFSs with possible non-trivial overlaps is new to our best knowledge.
\end{remark}

In particular, let $E_{a,p}=\Phi_{\bm{I}_a}(\Sigma^+_{a,p})$ and $E_{b,p'}=\Phi_{\bm{I}_b}(\Sigma^+_{b,p'})$ be two independent Mandelbrot percolations on the whole interval, i.e., the classical Mandelbrot percolations along the $a$-adic and $b$-adic grids of $[0,1]$. Since obviously in this setting $\gamma_{\Sigma^+_a,\bm{I}_a}=\gamma_{\Sigma^+_b,\bm{I}_b}=0$, by Corollary \ref{co2} we have

\begin{corollary}\label{co3}
If $\log a/\log b$ is irrational, then almost surely conditioned on $E_{a,p}\neq \emptyset$ and $E_{b,p'}\neq \emptyset$, for all $s\neq 0$,
\begin{equation*}
\dim (E_{a,p}+sE_{b,p'})=\min\{1,\dim E_{a,p}+\dim E_{b,p'}\}=\min\left\{1,2+\frac{\log p}{\log a}+\frac{\log p'}{\log b}\right\}.
\end{equation*}
\end{corollary}

\begin{remark}
This result can also be deduced from \cite{SS}, where Shmerkin and Suomala proved that Mandelbrot percolations on the whole interval are Salem sets.
\end{remark}

\medskip

\subsection*{Strategy of proofs}

To prove our main result Theorem \ref{mthm} we shall follow the framework of Hochman and Shmerkin \cite{HS12} to build up a relevant CP-chain and estimate its local entropy averages, then use Marstrand's projection theorem to find the right dimension. The main difficulty of extending this method to our setting comes from the possible overlaps in the IFSs and the random extension due to the Mandelbrot cascade action, which we shall explain here.

\medskip

Similar to \cite{FJ14}, to deal with possible overlaps in the IFSs, we shall first apply the CP-chain method at the symbolic space level, then we map the CP-chain to the Euclidean space through the canonical mappings $\Phi_{\bm{I}}$ and $\Phi_{\bm{J}}$. For IFSs with different contraction ratios $\delta$ and $\rho$ (an affine case), the underlying dynamical system is the product space
\[
X=\mathbb{T}\times\Sigma_a\times\Sigma_b
\]
where $\Sigma_a=\{1,\cdots,a\}^\mathbb{Z}$ and $\Sigma_b=\{1,\cdots,b\}^\mathbb{Z}$ are the two-sided symbolic spaces, and the invertible transformation $T$ on $X$ is the skew-product
\[
T(t,\bm{v})=(R(t),\sigma_t(\bm{v})),
\]
where $R(t)=t+\alpha \mod 1$ is the irrational rotation on $\mathbb{T}$ with $\alpha=\log \delta/\log \rho \in(0,1)$, and for $\bm{v}=(\bm{i},\bm{j})\in \Sigma_a\times\Sigma_b$ and $t\in \mathbb{T}$, the skewed map $\sigma_t$ is
\[
\sigma_t(\bm{v})=\left\{\begin{array}{ll}(\sigma(\bm{i}),\sigma(\bm{j})),& \text{if } 1-\alpha\le t<1;\\
(\bm{i},\sigma(\bm{j})),& \text{if } 0\le t<1-\alpha,
\end{array}\right.
\]
where $\sigma$ denotes the left-shift operator.

Let $\mu^+$ and $\nu^+$ be two $\sigma$-ergodic measures on $\Sigma_a^+$ and $\Sigma_b^+$ respectively. We first need to use Rokhlin's natural extension \cite{Roh60} to extend $\mu^+$ and $\nu^+$ on one-sided symbolic spaces to ergodic measures $\mu$ and $\nu$ on two-sided symbolic spaces $\Sigma_a$ and $\Sigma_b$. The invariant measure to consider on $X$ is then the product measure
\[
\Theta=\ell\times\mu\times \nu,
\]
where $\ell$ denotes the Lebesgue measure on $\mathbb{T}$. It is easy to verify that $\Theta$ is $T$-invariant, but it is not necessarily $T$-ergodic. Therefore we need to use ergodic decomposition to obtain an ergodic measure $\theta$ with respect to $\Theta$ (see Section \ref{dfedm} for precise definition).

Next we shall consider the canonical projection $\phi$ from $X$ to its subspace $\mathbb{T}\times\Sigma_a^-\times\Sigma_b^-$, where $\Sigma_a^-=\{1,\cdots,a\}^{\mathbb{Z}\setminus\mathbb{N}}$ and $\Sigma_b^-=\{1,\cdots,b\}^{\mathbb{Z}\setminus\mathbb{N}}$. The pre-images
\[
\eta=\{\phi^{-1}(t,\bm{i}^-,\bm{j}^-): (t,\bm{i}^-,\bm{j}^-)\in \mathbb{T}\times\Sigma_a^-\times\Sigma_b^-\}
\]
then forms a measurable partition of $X$. By Rokhlin's conditional measure theorem \cite{Roh} (see also Theorem \ref{Rohthm}) the conditional measures $\theta^\eta_x$ of $\theta$ with respect to $\eta$ for $\theta$-a.e. $x\in X$ defines a family of probability measures $\theta^+_{t,\bm{i}^-,\bm{j}^-}$ carried by $\Sigma_a^+\times\Sigma_b^+$, and these measures will be used to define an ergodic CP-chain along certain partition of $\Sigma_a^+\times\Sigma_b^+$ depending on $t$, then we shall map this CP-chain through the canonical mapping $\Phi=\Phi_{\bm{I}}\times\Phi_{\bm{J}}$ from $\Sigma_a^+\times\Sigma_b^+$ to $\mathbb{R}^2$ and estimate its associated local entropy averages.

In order to obtain a sharp lower bound of the Hausdorff dimension of the projection, when applying Marstrand's projection theorem, we need that both $\Phi(\theta^+_{t,\bm{i}^-,\bm{j}^-})$ and $\Phi(\mu^+\times\nu^+)$ are exact-dimensional with the same dimension. When the IFSs satisfy the OSC, this fact is relatively easier to prove (may still need some work), but when there are possible overlaps in the IFSs, this fact is highly non-trivial. We shall use the methods developed in \cite{FH09,Feng20} to prove such a result (Theorem \ref{ed2}).

Another difficulty created by the possible overlaps in the IFSs comes from the CP-chain construction of $\theta^+_{t,\bm{i}^-,\bm{j}^-}$ on $\Sigma_a^+\times\Sigma_b^+$ according to $\Phi$. To capture the dynamics of the zooming-in procedure for $\Phi(\theta^+_{t,\bm{i}^-,\bm{j}^-})$, one needs to use a family of rectangles $\mathcal{U}$ (see \eqref{Ug} for definition) instead of balls on the Euclidean space. This creates problems to the dominated convergence of its associated conditional information (\cite[Proposition 3.5]{FH09}), where the differentiation theory of measures, Besicovitch covering lemma and Hardy-Littlewood maximal inequality are all applied in terms of balls. Fortunately we have a complete analogue of these theorems in terms of the so-called $\gamma$-Morse sets, which covers our family of rectangles $\mathcal{U}$. See \cite[Section 1.2]{FL07} and Section \ref{mct}.

\medskip

The difficulty created by the random extension due to Mandelbrot cascade action is the following. To capture the dynamics of the Mandelbrot cascade version of $\mu^+$ and $\nu^+$ we need to consider the following skew-product space (a random extension)
\[
\widetilde{X}=X\times\Omega_{a,b}
\]
and a skew-product transformation
\[
\widetilde{T}(x,\bm{\omega})=(T(x),\kappa_x(\bm{\omega})).
\]
See Section \ref{esps} for more precise definition. The problem of this random extension is that the transformation $\widetilde{T}$ is no longer invertible, therefore results like Shannon-McMillan-Breiman theorem for conditional measures (Proposition \ref{SMB}) cannot be used here, hence the Mandelbrot cascade version of the fact that $\Phi(\widetilde{\theta}^+_{t,\bm{i}^-,\bm{j}^-})$ and $\Phi(\widetilde{\mu}^+\times\widetilde{\nu}^+)$ are exact-dimensional with the same dimension, or even the simpler statement for $\Phi_{\bm{I}}(\widetilde{\mu}^+_{\bm{i}^-})$ and $\Phi_{\bm{I}}(\widetilde{\mu}^+)$ as in Theorem \ref{ed22}, becomes much harder to prove. We will bypass this difficulty by combining the methods developed in \cite{LY85II,FH09,Feng20}.

\medskip

The rest of the paper is organised as follows.

In Section \ref{sec:pre} we introduce Rokhlin's conditional measure theorem, Shannon-McMillan-Breiman theorem for invariant measures, and we prove in Proposition \ref{SMB} a Shannon-McMillan-Breiman type theorem for conditional measures, which we shall use frequently later on. In Section \ref{mct} we introduce the $\gamma$-Morse covering and its associated dominated convergence result for conditional information, which will be essential for us to prove Theorem \ref{ed2}. At last we present a version of Maker's ergodic theorem that will be used in the proofs of Lemma \ref{lem1}, Theorem \ref{ed1} and Theorem \ref{ed2}.
 
In Section \ref{sec3} we introduce the natural extension of invariant measures from one-sided symbolic space to two-sided symbolic space, and their fibre measures with respect to the left-infinite words. We prove in Lemma \ref{lem1} a folklore result that these fibre measures have the same dimension as the original invariant measure on the one-sided symbolic space when it is ergodic or exact-dimensional. We shall use this dimension result to build the fibre Mandelbrot cascade measures and their associated Peyri\`ere measure on the random extension skew-product space. We show in Lemma \ref{lminv} and Lemma \ref{erg} that this Peyri\`ere measure is invariant and ergodic.

In Section \ref{ifsed} we prove the exact-dimensionality of push-forward original and fibre Mandelbrot cascade measures, and that their dimension are the same. We first show in Theorem \ref{ed1} the exact-dimensionality of the fibre cascade measures, following the method of Feng and Hu \cite{FH09}. Then we show in Theorem \ref{ed22} that the original cascade measure is exact-dimensional with the same dimension, following the arguments of Ledrappier and Young \cite{LY85II}.

In Section \ref{pir} we introduce the skew-product space $X$ and we prove in Theorem \ref{ed3} that the fibre ergodic decomposition measure $\theta^\eta_x$ have the same dimension as $\mu^+\times\nu^+$. Then we use this dimension result to build their corresponding Mandelbrot cascade measures and associated Peyri\`ere measure, we show that this Peyri\`ere measure is also invariant and ergodic.

In Section \ref{mcfm} we mainly show in Theorem \ref{ed2} that the push-forward fibre ergodic decomposition Mandelbrot cascade measure is exact-dimensional with the same dimension as $\Phi(\widetilde{\mu}^+\times\widetilde{\nu}^+)$, which is crucial for us to derive a sharp lower bound of the Hausdorff dimension of projection in Section \ref{DoP}. For this we need to make clear the relation between the ergodic decomposition and the random extension (see Lemma \ref{erg3}).

In Section \ref{DoP} we prove our main theorem following \cite{HS12}, with several modifications due to the fact that we are not allowed to use the ``cut-line" method as in \cite{FJ14} to enlarge the symbolic spaces by generations. Finally in Section \ref{sec:perc} we apply our main result to Mandelbrot percolations.

\medskip

\noindent {\bf Acknowledgement. } The authors would like to thank Julien Barral, De-Jun Feng and Kenneth Falconer for their many useful comments and suggestions for improving an early version of this paper.

\section{Preliminaries}\label{sec:pre}

In this section we present some results related to Rokhlin's conditional measures, Shannon-McMillan-Breiman theorem, conditional information and entropy. For details one can check the nice book \cite{VO} of Viana and Oliveira, see also \cite[Section 1]{Roh67}, \cite[Section 3]{FH09} and \cite[Section 2]{Feng20}.
	
\subsection{Rokhlin's conditional measures}\label{sec:rokhlin}

A Lebesgue space is a probability space that is isomorphic to the union of an interval $[0,s]$ for some $0\le s\le 1$ with Lebesgue measure and a countable number of atoms. For example, a Borel probability measure on a standard Borel space (a metric space whose Borel $\sigma$-algebra makes the space complete and separable) is isomorphic to a Lebesgue space, and a countable product of Lebesgue spaces is still a Lebesgue space. 

Let $(X, \mathcal{B}, \mu)$ be a Lebesgue space. A measurable partition $\eta$ of $(X, \mathcal{B}, \mu)$ is a partition of some full $\mu$-measure subset $X'$ of $X$ such that there is a sequence of countable partitions $\eta_1 \prec \eta_2 \prec\cdots$ of $X'$ consisting of elements of $\mathcal{B}$ such that
\[
\bigvee_{i} \eta_i=\eta,
\]
where $\eta_i\prec \eta_j$ means that every element of $\eta_j$ is contained in some element of $\eta_i$, and $\bigvee_{i} \eta_i$ denotes the partition obtained by taking intersections of all elements of $\eta_i$ for all $i$. For $x\in X$ denote by $\eta(x)$ the unique element in $\eta$ containing $x$. Denote by $\widehat{\eta}$ the sub $\sigma$-algebra of $\mathcal{B}$ consisting of unions of elements of $\eta$. We have the following Rokhlin's theorem on existence and uniqueness of conditional measures.

\begin{theorem}[Rokhlin \cite{Roh}]\label{Rohthm}
There is a family of probability measures $\mu^{\eta}_{x}$ defined on $\eta(x)$ for $\mu$-a.e. $x\in X$. These measures are uniquely determined (up to sets of $\mu$-measure zero) by the following: if $A \subset X$ is a measurable set, then $x\mapsto \mu^{\eta}_{x}(A)$ is $\widehat{\eta}$-measurable and $\mu(A) =\int_X \mu^{\eta}_{x}(A) \, \mathrm{d}\mu(x)$. These properties implies that for any $f\in L^1(X,\mathcal{B},\mu)$, $E_\mu(f|\widehat{\eta})(x)=\int_X f \,\mathrm{d} \mu^{\eta}_{x}$ for $\mu$-a.e. $x\in X$, and $\int_X f \,\mathrm{d}\mu=\int_X E_\mu(f|\widehat{\eta}) \, \mathrm{d}\mu$.
\end{theorem}

Rokhlin's conditional measures can be used to define the disintegration of measures. Let $Y$ be a Besicovitch space, that is a complete separable metric space on which the Besicovitch covering lemma holds, such as Euclidean spaces or ultra-metric spaces. Let $\phi:X\to Y$ be a measurable map. By Rokhlin \cite{Roh60}, the family
\[
\eta=\{\phi^{-1}(y):y\in Y\}
\]
forms a measurable partition of $X$, and the $\sigma$-algebra $\widehat{\eta}$ coincide with the sub $\sigma$-algebra generated by $\phi^{-1}(\mathcal{B}_Y)$ module $\mu$-measure zero sets, where $\mathcal{B}_Y$ is the Borel $\sigma$-algebra of $Y$. For such a measurable partition $\eta$, the conditional measures $\mu^\eta_x$ can be viewed as a disintegration of $\mu$ with respect to $\mu\circ \phi^{-1}$: we denote by $B(y,r)$ the closed ball in $Y$ centred at $y$ with radius $r$. For $r>0$ and $x\in X$ denote by
\[
B_\phi(x,r)=\phi^{-1}(B(\phi(x),r)).
\]
For $A\in \mathcal{B}$ with $\mu(A)>0$ denote by
\begin{equation*}\label{cdm1}
\mu_A(B)=\frac{\mu(B\cap A)}{\mu(A)}, \ B\in\mathcal{B}
\end{equation*}
the conditional measure of $\mu$ given $A$. For $x\in X$, the weak limit
\[
\lim_{r\to 0}\mu_{B_\phi(x,r)},
\]
if exists, defines a probability measure on the fibre $\phi^{-1}(\phi(x))=\eta(x)$, as the disintegration of $\mu$ with respect to $\mu\circ \phi^{-1}$ at $\phi(x)$. By the differentiation theory of measures (see \cite[Corollary 2.14 (2)]{Mattila} for example) and the uniqueness of conditional measures, these two family of measures coincide on a set of full $\mu$-measure. In other words, we may say that for $\mu$-a.e. $x\in X$,
\begin{equation*}\label{cdint}
\mu^{\eta}_{x}=\lim_{r\to 0}\mu_{B_\phi(x,r)}.
\end{equation*}
See also \cite{Simmons} for a proof of this statement. 

\subsection{Shannon-McMillan-Breiman theorem}

Consider a Lebesgue space  $(X,\mathcal{B},\mu)$ and let $(X,\mathcal{B},\mu,T)$ be an invertible measure-preserving dynamical system (m.p.d.s.). For a countable measurable partition $\xi$ denote by
\[
H_\mu(\xi)=-\sum_{B\in\xi} \mu(B)\log \mu(B)
\]
the entropy of $\xi$ with respect to $\mu$. Let $\mathcal{P}$ be a countable measurable partition of $X$. For $n\ge 1$ define
\[
\mathcal{P}^{n}=\bigvee_{k=0}^{n-1}T^{-k}\mathcal{P}.
\]
Note that for $x\in X$ we have
\[
\mathcal{P}^{n}(x)=\mathcal{P}(x)\vee T^{-1}\mathcal{P}(T(x))\vee \cdots \vee T^{-n+1}\mathcal{P}(T^{n-1}(x)).
\]
Let 
\[
h_{\mu}(T,\mathcal{P})=\lim_{n\to\infty} \frac{1}{n} H_{\mu}(\mathcal{P}^{n})
\]
be the entropy of $T$ and $\mathcal{P}$ with respect to $\mu$. We state here the classical Shannon-McMillan-Breiman (SMB) theorem for invariant measures (see \cite[Theorem 9.3.1]{VO} for example).

\begin{theorem}[SMB Theorem]\label{SMBinv}
If $H_\mu(\mathcal{P})<\infty$, then the limit
\[
h_\mu(T,\mathcal{P},x)=\lim_{n\to\infty} \frac{-\log \mu(\mathcal{P}^n(x))}{n}
\]
exists for $\mu$-a.e. $x\in X$. Moreover,
\[
h_\mu(T,\mathcal{P})=\int_X h_\mu(T,\mathcal{P},x) \, \mu(\mathrm{d}x).
\]
If $\mu$ is $T$-ergodic, then $h_\mu(T,\mathcal{P},x)=h_\mu(T,\mathcal{P})$ for $\mu$-a.e. $x\in X$.
\end{theorem}

Let $\eta$ be a measurable partition. For $\mu$-a.e. $x\in X$ define the conditional information
\[
I_\mu(\mathcal{P}|\widehat{\eta})(x)=-\log \mu^\eta_{x}(\mathcal{P}(x));
\]
the local conditional entropy
\[
h_{\mu}(\mathcal{P}|\widehat{\eta})(x)=\mathbb{E}_\mu(I_\mu(\mathcal{P}|\widehat{\eta}) | \mathcal{I})(x),
\]
where $\mathcal{I}=\{B\in\mathcal{B}: T^{-1}B=B\}$ is the $T$-invariant $\sigma$-algebra of $X$. Define the conditional entropy
\[
H_\mu(\mathcal{P}|\widehat{\eta})=\mathbb{E}_\mu(I_\mu(\mathcal{P}|\widehat{\eta}))=\mathbb{E}_\mu(h_{\mu}(\mathcal{P}|\widehat{\eta})).
\]
We have the following analogue of Shannon-McMillan-Breiman theorem for conditional measures, that we will use frequently later on.

\begin{proposition}\label{SMB}
If for $\mu$-a.e. $x\in X$ and $n\ge 1$
\begin{equation}
T^{-n}(\eta(T^n(x)))=\mathcal{P}^{n}(x)\cap \eta(x), \label{TetaP}
\end{equation}
and $H_\mu(\mathcal{P}|\widehat{\eta})<\infty$, then for $\mu$-a.e. $x\in X$,
\[
\lim_{n\to\infty} \frac{-\log \mu^{\eta}_x(\mathcal{P}^{n}(x))}{n}=h_{\mu}(\mathcal{P}|\widehat{\eta})(x).
\]
\end{proposition}

\begin{proof}
By \eqref{TetaP} we have that for $n\ge 1$,
\begin{equation}\label{T-neta}
T^{-n}\eta=\mathcal{P}^{n}\vee \eta.
\end{equation}
It is well-known (see \cite{Parry81} for example) that for $f\in L^1(X,\mathcal{B},\mu)$ and a sub $\sigma$-algebra $\mathcal{A}$, we have
\[
\mathbb{E}_\mu(f|\mathcal{A})\circ T^n=\mathbb{E}_\mu(f\circ T^n|T^{-n}\mathcal{A}),
\]
where $T^{-n}\mathcal{A}$ is the sub-$\sigma$-algebra $\{T^{-n}(B):B\in \mathcal{A}\}$. It is easy to see that $\widehat{T^{-n}\eta}=T^{-n}\widehat{\eta}$. By Theorem \ref{Rohthm} and the fact that $\mathbf{1}_{T^n(B)}\circ T^n=\mathbf{1}_{B}$ (since $T$ is invertible) we have for $n\ge 1$ and $B\in\mathcal{B}$, for $\mu$-a.e. $x$,
\begin{align*}
\mu^{\eta}_{T^n(x)}(T^n(B))=&\ \mathbb{E}_{\mu}(\mathbf{1}_{T^n(B)}|\widehat{\eta})(T^n(x))\\
=&\ \mathbb{E}_{\mu}(\mathbf{1}_{T^n(B)}\circ T^n|T^{-n}\widehat{\eta})(x)\\
=&\ \mathbb{E}_{\mu}(\mathbf{1}_{B}|\widehat{T^{-n}\eta})(x).
\end{align*}
Furthermore, by \eqref{T-neta}, $T^{-n}\eta=\mathcal{P}^{n}\vee \eta$, then by Theorem \ref{Rohthm} and the definition of conditional expectation, for $\mu$-a.e. $x$,
\begin{align*}
\mu^{\eta}_{T^n(x)}(T^n(B))=&\ \mathbb{E}_{\mu}(\mathbf{1}_B|\widehat{ \mathcal{P}^{n}\vee \eta})(x)\\
=&\ \mathbb{E}_{\mu}(\mathbf{1}_B|\widehat{ \mathcal{P}^{n}}\vee \widehat{\eta})(x)\\
=&\ \mathbb{E}_{\mu^\eta_x}(\mathbf{1}_B|\widehat{\mathcal{P}^{n}})(x)\\
=&\ \frac{\mu^\eta_x(B\cap \mathcal{P}^{n}(x))}{\mu^\eta_x(\mathcal{P}^{n}(x))}.\numberthis\label{inid}
\end{align*}
Applying \eqref{TetaP} to $T(x)$ we get for $n\ge 2$,
\begin{align*}
\mathcal{P}^{n-1}(T(x))\cap \eta(T(x))&=\ T^{-n+1}\eta(T^{n-1}(T(x)))\\
&=\ T(T^{-n}\eta(T^{n}(x)))\\
&=\ T(\mathcal{P}^{n}(x)\cap \eta(x)).
\end{align*}
Note that by $\mathcal{P}(x)\supset \mathcal{P}^{n}(x)$, and by \eqref{TetaP} for $n=1$, we have
\begin{align*}
\mathcal{P}^{n}(x)\cap \eta(x) &=\ \mathcal{P}^{n}(x)\cap (\mathcal{P}(x)\cap \eta(x))\\
&=\ \mathcal{P}^{n}(x)\cap T^{-1}\eta(T(x)).
\end{align*}
Therefore, since $T$ is invertible,
\[
T(\mathcal{P}^{n}(x)\cap \eta(x))=T(\mathcal{P}^{n}(x)\cap T^{-1}\eta(T(x)))=T(\mathcal{P}^{n}(x))\cap \eta(T(x)),
\]
which implies that
\begin{equation}\label{P0n}
\mathcal{P}^{n-1}(T(x))\cap \eta(T(x))=T(\mathcal{P}^{n}(x))\cap \eta(T(x)).
\end{equation}
Since the conditional measure $\mu^\eta_{T(x)}$ is carried by $\eta(T(x))$, by \eqref{P0n} we have for $\mu$-a.e $x$,
\begin{align*}
\mu^\eta_{T(x)}(T(\mathcal{P}^{n}(x)))&=\ \mu^\eta_{T(x)}(T(\mathcal{P}^{n}(x))\cap \eta(T(x)))\\
&=\ \mu^\eta_{T(x)}(\mathcal{P}^{n-1}(T(x))\cap \eta(T(x)))\\
&=\ \mu^\eta_{T(x)}(\mathcal{P}^{n-1}(T(x))).\numberthis\label{eTP}
\end{align*}
Since $\mathcal{P}^{n}(x)\cap \mathcal{P}(x)=\mathcal{P}^{n}(x)$, by applying \eqref{inid} for $B\in \mathcal{P}^{n}$ and then by applying \eqref{eTP} we have for $\mu$-a.e. $x\in B$ (so that $ \mathcal{P}^{n}(x)=B$) and $n\ge 2$,
\begin{align*}
 \frac{\mu^\eta_{x}(\mathcal{P}^{n}(x))}{\mu^\eta_{x}(\mathcal{P}(x))}&=\ \frac{\mu^\eta_{x}(\mathcal{P}^{n}(x)\cap \mathcal{P}(x))}{\mu^\eta_{x}(\mathcal{P}(x))}\\
 &=\ \frac{\mu^\eta_{x}(B\cap \mathcal{P}(x))}{\mu^\eta_{x}(\mathcal{P}(x))}\\
  &=\ \mu^\eta_{T(x)}(T(B))\\
 &=\ \mu^\eta_{T(x)}(T(\mathcal{P}^{n}(x)))\\
 &=\ \mu^\eta_{T(x)}(\mathcal{P}^{n-1}(T(x))).\numberthis\label{zmin}
\end{align*}
Taking over all $B\in \mathcal{P}^{n}$ we obtain that for $\mu$-a.e. $x\in X$ and $n\ge 2$,
\[
\mu^\eta_{x}(\mathcal{P}^{n}(x))=\mu^\eta_{x}(\mathcal{P}(x))\cdot \mu^\eta_{T(x)}(\mathcal{P}^{n-1}(T(x))).
\]
By iteration we have for $\mu$-a.e. $x\in X$ and $n\ge 1$,
\begin{equation}\label{meP}
\mu^\eta_{x}(\mathcal{P}^{n}(x))=\prod_{k=0}^{n-1}\mu^\eta_{T^k(x)}(\mathcal{P}(T^k(x))).
\end{equation}
Recall that the conditional information is defined as
\[
I_\mu(\mathcal{P}|\widehat{\eta})(x)=-\log \mu^\eta_{x}(\mathcal{P}(x)).
\]
Then by \eqref{meP} we have for $\mu$-a.e. $x\in X$ and $n\ge 1$,
\[
\frac{-\log \mu^\eta_{x}(\mathcal{P}^{n}(x))}{n}=\frac{1}{n}\sum_{k=0}^{n-1}I_\mu(\mathcal{P}|\widehat{\eta})\circ T^k(x).
\]
Then by Birkhoff ergodic theorem we have for $\mu$-a.e. $x$,
\begin{align*}
\lim_{n\to\infty}\frac{-\log \mu^\eta_{x}(\mathcal{P}^{n}(x))}{n}&=\ \mathbb{E}_{\mu}( I_\mu(\mathcal{P}|\widehat{\eta}) |\mathcal{I})(x).
\end{align*}
which yields the conclusion.
\end{proof}

\subsection{Morse covering and dominated convergence of conditional information}\label{mct}

Later in Section \ref{mcfm} when dealing with the local dimension of product measures, it would be more convenient to work with sets that are slightly more general than balls. In particular we shall consider the so-called $\gamma$-Morse sets. Assume that $Y=\mathbb{R}^k$ is a Euclidean space of dimension $k\ge 1$. Let $\phi:X\mapsto Y$ be a measurable map and let $E=\phi(X)\subset \mathbb{R}^k$ be a bounded set. Following \cite[Section 1.2]{FL07} for $\gamma\ge 1$ we say that the family $\mathcal{U}=\{U\}$ of nonempty subsets of $\mathbb{R}^k$ is a fine $\gamma$-Morse covering of $E$ if for each $y\in E$:
\begin{itemize}
\item[(i)] there exists a $U\in \mathcal{U}$ such that there exists $r>0$ such that
\[
B(y,r)\subset U\subset B(y,\gamma r)
\]
and for each $a\in B(y,r)$, $b\in U$ we have $pa+(1-p)b\in U$ for all $p\in (0,1)$.
\item[(ii)] let $\mathcal{U}_y\subset \mathcal{U}$ be the collection of elements in $\mathcal{U}$ satisfying (i) for $y$, then $\inf\{|U|:U\in\mathcal{U}_y\}=0$, where $|U|$ stands for the diameter of $U$.
\end{itemize}
For example, when $k=2$, a family of rectangles whose aspect ratios are uniformly bounded is a fine $\gamma$-Morse covering of $E$ for some $\gamma\ge 1$ if for each $y\in E$ there is a sequence of rectangles in the family centred at $y$ whose diameter decreases to $0$.

Let $\xi$ be a measurable partition of $X$. Note that, for $\mu$-typical $x\in X$ at which the conditional measure $\mu^{\xi}_{x}$ is well-defined, the fibre measure space $(\xi(x),\mathcal{B}|_{\xi(x)}, \mu^{\xi}_{x})$ is also a Lebesgue space. Therefore we may apply Rokhlin's conditional measure theorem to $\mu^{\xi}_{x}$ with respect to the measurable partition $\eta=\{\phi^{-1}(y):y\in E\}$ restricted to $\xi(x)$.  Let $ \widehat{\xi}\vee \widehat{\eta}$ denote the smallest sub $\sigma$-algebra containing both $\widehat{\xi}$ and $\widehat{\eta}$, which is equal to $\widehat{\xi\vee \eta}$. We have the following extension of \cite[Lemma 3.3, Proposition 3.5]{FH09}.
\begin{proposition}\label{P3.5}
Let $(X,\mathcal{B},\mu)$ be a Lebesgue space. Let $\mathcal{U}=\{U\}$ be a fine $\gamma$-Morse covering of $E=\phi(X)$ for some $\gamma\ge 1$. Let $\mathcal{P}$ be a countable measurable partition of $X$. Then
\begin{itemize}
\item[(1)] Let $A\in\mathcal{B}$. For $\mu$-a.e. $x\in X$, for any subsequence $\{U_n\}\subset\mathcal{U}_{\phi(x)}$ with $|U_n|\searrow 0$ as $n\to\infty$,
\[
\lim_{n\to\infty} \frac{\mu^{\xi}_{x}(\phi^{-1}(U_n)\cap A)}{\mu^{\xi}_{x}(\phi^{-1}(U_n))}=\mathbb{E}_{\mu^{\xi}_{x}}(\mathbf{1}_A|  \widehat{\eta})(x)=\mathbb{E}_\mu(\mathbf{1}_A|\widehat{\xi\vee \eta})(x),
\]
and
\[
\lim_{n\to \infty} -\log \frac{\mu^{\xi}_{x}(\phi^{-1}(U_n)\cap \mathcal{P}(x))}{\mu^{\xi}_{x}(\phi^{-1}(U_n))}=I_{\mu^{\xi}_{x}}(\mathcal{P}|\widehat{\eta})(x)=I_\mu(\mathcal{P}|\widehat{\xi\vee \eta})(x).
\]

\item[(2)] Set
\[
g(x)=\sup_{U\in\mathcal{U}_{\phi(x)}}-\log \frac{\mu^{\xi}_{x}(\phi^{-1}(U)\cap \mathcal{P}(x))}{\mu^{\xi}_{x}(\phi^{-1}(U))}
\]
and assume $H_\mu(\mathcal{P})<\infty$. Then $g\ge 0$ and $\int_X g \,\mathrm{d}\mu \le H_\mu(\mathcal{P})+C<\infty$, where the constant $C$ only depends on $k$ and $\gamma$.
\end{itemize}
\end{proposition}

\begin{remark}
In the case when $\gamma=1$, i.e., when $\mathcal{U}$ is a family of balls centred in $E$, we can work on a more general metric space rather than Euclidean spaces. In particular we may take $Y$ to be a Besicovitch space. All the statements in Proposition \ref{P3.5} still hold in this case. See \cite[Lemma 2.5]{Feng20}. 
\end{remark}

\begin{proof}[Proof of Proposition \ref{P3.5}]
The proof is almost identical to that of \cite[Lemma 3.3, Proposition 3.5]{FH09}. We only need the analogues of the differentiation theorem of measures and Hardy-Littlewood maximal inequality in terms of $\gamma$-Morse sets instead of balls. The former can be found in \cite[Theorem 1.153 \& Remark 1.154]{FL07}. The later is an easy consequence (see for example the proof of \cite[Theorem 2.19 (2)]{Mattila}) of Besicovitch covering theorem for $\gamma$-Morse covering, namely the Morse covering theorem \cite[Theorem 1.144]{FL07}.
\end{proof}

In the last we state a variant of Maker's ergodic theorem, see \cite[Theorem 1]{Mak} and \cite[Corollary 1.6 p.96]{Mane}.

\begin{theorem}\label{met}
Let $(X,\mathcal{B},\mu,T)$ be a m.p.d.s.. If the functions $f_{k}(x)$ $(k=1,2,\cdots)$, $f$ and the positive function $g(x)$ are in $L^1(X,\mathcal{B},\mu)$ with $|f_{k}(x)| \le g(x)$ and $\lim_{k\to\infty} f_{k}(x)=f(x)$ for $\mu$-a.e. $x\in X$, then
\[
\lim_{n\to\infty}\frac{1}{n}\sum_{k=0}^{n-1} f_{k}(T^kx)=\lim_{n\to\infty}\frac{1}{n}\sum_{k=0}^{n-1} f_{n-k}(T^kx)=\mathbb{E}_{\mu}(f|\mathcal{I})(x)
\]
for $\mu$-a.e. $x\in X$, where $\mathcal{I}=\{B\in\mathcal{B}: T^{-1}B=B\}$.
\end{theorem}

\section{Symbolic space, dimension of fibre measures and Mandelbrot cascades}\label{sec3}
	
\subsection{Two-sided symbolic space}\label{sds}
For an integer $a\ge 2$ denote by $\Sigma_a=\{1,\ldots,a\}^\mathbb{Z}$ the two-sided symbolic space of $a$ letters. For an infinite word  $\bm{i}=(i_n)_{n\in\mathbb{Z}} \in\Sigma_a$ and $-\infty\le m\le n\le \infty$ denote by
\[
\bm{i}|_m^n=(i_k)_{k=m}^n
\]
the sequence of letters of $\bm{i}$ from position $m$ to $n$. In particular we denote by
\[
\bm{i}^-=\bm{i}|_{-\infty}^0 \text{ and }  \bm{i}^+=\bm{i}|_1^\infty 
\]
its one-sided infinite words. We may view $\pi^\pm:\bm{i}\mapsto \bm{i}^\pm$ as the canonical mapping from $\Sigma_a$ onto the one-sided symbolic space
\[
\Sigma_a^\pm=\{\bm{i}^\pm:\bm{i}\in\Sigma_a\}.
\]
We may write $\Sigma_a$ as $\Sigma_a^-\times\Sigma_a^+$ and $\bm{i}\in\Sigma_a$ as $(\bm{i}^-,\bm{i}^+)$. 

For $\bm{i}\in\Sigma_a$ and $n\ge 0$ we denote by
\[
\bm{i}^-|_n=\bm{i}|_{-n+1}^0  \text{ and }  \bm{i}^+|_n=\bm{i}|_1^n,
\]
with the convention that $\bm{i}^\pm|_0=\emptyset$ denotes the empty word. Let
\[
\Sigma_a^{\pm,*}=\{\bm{i}^\pm|_{n}: \bm{i}\in\Sigma_a, n\ge 0\}
\]
be the set of all finite words in $\Sigma_a^\pm$. Note that $\Sigma_a^{\pm,*}$ is countable. For $u=u_m\cdots u_0\in \Sigma_a^{-,*}$ define the cylinder in $\Sigma_a^-$ at $u$ by
\[
[u]^-=\{\bm{i}^-\in\Sigma_a^-: i_k=u_k \text{ for } k=m,\cdots,0\},
\]
and for $v=v_1\cdots v_n\in \Sigma_a^{+,*}$ define the cylinder in $\Sigma_a^+$ at $v$ by
\[
[v]=\{\bm{i}^+\in\Sigma_a^+: i_k=v_k \text{ for } k=1,\cdots, n\}
\]
with the convention that $[\emptyset]=\Sigma_a^+$ and $[\emptyset]^-=\Sigma_a^-$.

Let $\Sigma_a^*=\Sigma_a^{-,*}\times \Sigma_a^{+,*}$. For $(u,v)\in \Sigma_a^*$ the product set $[u]^-\times[v]$ is called a cylinder of $\Sigma_a$ at $(u,v)$. Let $\mathcal{B}_a$ the Borel $\sigma$-algebra of $\Sigma_a$ generated by cylinders. Also we denote by $\mathcal{B}_a^\pm$ the sub Borel $\sigma$-algebra of $\mathcal{B}_a$ restricted to $\Sigma_a^\pm$. It is easy to see that $\mathcal{B}_a^\pm$ is the $\sigma$-algebra generated by the cylinders in $\Sigma_a^\pm$.

Given $\delta\in(0,1)$, we may equip $\Sigma_a^\pm$ with a metric $d_\delta^\pm$ defined as
\[
d_{\delta}^\pm(\bm{i}^\pm,\bm{j}^\pm)=\delta^{\inf\{n\ge0: \bm{i}^\pm|_n=\bm{j}^\pm|_n \}},\;\; \bm{i}^\pm,\bm{j}^\pm\in\Sigma^\pm_a.
\]
Note that $\mathcal{B}_a^\pm$ is the Borel $\sigma$-algebra of $\Sigma_a^\pm$ with respect to the metric $d_\delta^\pm$. Finally we may define a metric $d_\delta$ on $\Sigma_a$ by
\[
d_\delta(\bm{i},\bm{j})=\max\{d_\delta^-(\bm{i}^-,\bm{j}^-),d_\delta^+(\bm{i}^+,\bm{j}^+)\}, \ \bm{i},\bm{j}\in\Sigma_a.
\]
It is easy to see that $(\Sigma_a,d_\delta)$ is a complete separable metric space, hence any Borel probability measure on it is isomorphic to a Lebesgue space.

\subsection{Natural extension}\label{necm}

Let $\sigma(\bm{i})=(i_{n+1})_{n\in\mathbb{Z}}$ denote the left-shift operator on $\Sigma_a$. A probability measure $\mu$ on $(\Sigma_a,\mathcal{B}_a)$ is $\sigma$-invariant if $\mu=\mu\circ\sigma^{-1}$. In such a case $(\Sigma_a,\mathcal{B}_a,\sigma,\mu)$ forms an invertible m.p.d.s..

The left-shift operator induces the left-shift operator on $\Sigma_a^+$, which is also denoted by $\sigma$. Let $\mu^+$ be the push-forward measure on $\Sigma_a^+$ of $\mu$ through $\pi^+$, that is,
\[
\mu^+(B)=\mu\circ (\pi^+)^{-1}(B)=\mu(\Sigma_a^-\times B),\ B\in\mathcal{B}_a^+.
\]
Then $\mu^+$ is $\sigma$-invariant and $(\Sigma_a^+,\mathcal{B}_a^+,\sigma,\mu^+)$ is a factor of $(\Sigma_a,\mathcal{B}_a,\sigma,\mu)$.

By Rokhlin's natural extension theorem \cite{Roh60} we may turn this the other way around: starting with a one-sided m.p.d.s.  $(\Sigma_a^+,\mathcal{B}_a^+,\sigma,\mu^+)$, there exists a unique $\sigma$-invariant probability measure $\mu$ on the two-sides symbolic space $(\Sigma_a,\mathcal{B}_a)$ such that $\mu^+=\mu\circ (\pi^+)^{-1}$. More precisely, consider the space
\[
\widecheck{\Sigma}^+_a=\{\widecheck{\bm{i}}=(\ldots,\bm{i}^+_{-1},\bm{i}^+_{0},\bm{i}^+_{1},\bm{i}^+_{2},\ldots)\in (\Sigma_a^+)^\mathbb{Z}: \sigma(\bm{i}^+_{n})=\bm{i}^+_{n+1} \text{ for all } n\in\mathbb{Z}\}.
\]
Let $\widecheck{\mathcal{B}}_a$ be the smallest $\sigma$-algebra containing sets of the form $\{\widecheck{\bm{i}}\in \widecheck{\Sigma}^+_a: \bm{i}^+_{k} \in B\}$ for $k\le 1$ and $B\in \mathcal{B}_a^+$. Let $\widecheck{\sigma}(\widecheck{\bm{i}})=(\bm{i}^+_{n+1})_{n\in \mathbb{Z}}$  be the lifted left-shift operator on $\widecheck{\Sigma}^+_a$. Then $(\widecheck{\Sigma}^+_a,\widecheck{\mathcal{B}}_a,\widecheck{\sigma})$ forms a two-sided shift dynamical system. Define the projection $\mathrm{Proj}_1: \widecheck{\Sigma}^+_a \to \Sigma_a^+$ by
\[
\mathrm{Proj}_1(\widecheck{\bm{i}})=\bm{i}^+_1.
\]
Then $\mathrm{Proj}_1 \circ \widecheck{\sigma}=\sigma\circ \mathrm{Proj}_1$. By Rokhlin \cite{Roh60} there exists a unique invariant probability measure $\widecheck{\mu}$ on $(\widecheck{\Sigma}^+_a,\widecheck{\mathcal{B}}_a,\widecheck{\sigma})$ such that
\begin{equation}\label{cp}
\widecheck{\mu} \circ (\mathrm{Proj}_1)^{-1}=\mu^+.
\end{equation}
Furthermore, we have
\begin{equation}\label{lift}
\widecheck{\mu}(\{\widecheck{\bm{i}}\in \widecheck{\Sigma}_a: \bm{i}_k^+\in B_k \text{ for } k=m,m+1,\ldots,0,1\})=\mu^+\big(\bigcap_{k=m,\ldots,1} \sigma^{m-k} B_k\big)
\end{equation}
for all $n\le 1$ and $B_k\in\mathcal{B}_a^+$ for $k=m,\ldots,1$. Moreover, $\widecheck{\mu}$ is ergodic if and only if $\mu^+$ is ergodic. We may identify the natural extension $\widecheck{\Sigma}^+_a$ by $\Sigma_a$: for $\widecheck{\bm{i}}=(\ldots,\bm{i}^+_{-1},\bm{i}^+_{0},\bm{i}^+_{1},\bm{i}^+_{2},\ldots)\in \widecheck{\Sigma}^+_a$ with $\bm{i}^+_{1}=\bm{i}^+=i_1i_2\cdots\in\Sigma_a^+$, for all $k\ge 2$ we have that $\bm{i}^+_k=\sigma^k(\bm{i}^+)=i_ki_{k+1}\cdots$ is uniquely determined by $\bm{i}^+$. On the other hand, since $\sigma(\bm{i}^+_{0})=\bm{i}^+_{1}$, there is an $i_0\in \{1,\ldots,a\}$ such that $\bm{i}^+_{0}=i_0i_1i_2\cdots$, thus $i_0$ and $\bm{i}^+$ determines $\bm{i}^+_{0}$. Similarly, since $\sigma(\bm{i}^+_{-1})=\bm{i}^+_{0}$, there is an $i_{-1}\in\{1,\ldots,a\}$ such that $\bm{i}^+_{-1}=i_{-1}\bm{i}_{0}$, thus $i_{-1}$, $i_0$ and $\bm{i}^+$ determines $\bm{i}^+_{-1}$, and so on. Hence there is a one-to-one correspondence between two-sided infinite word $\bm{i}=\cdots i_{-1}i_0 i_1i_2\cdots\in \Sigma_a$ and the sequence $\widecheck{\bm{i}}\in \widecheck{\Sigma}^+_a$. Correspondingly, we may identify the $\sigma$-algebra  $\widecheck{\mathcal{B}}_a$ as the Borel $\sigma$-algebra $\mathcal{B}_a$ on $\Sigma_a$, and identify the natural extension measure $\widecheck{\mu}$ as a $\sigma$-invariant measure $\mu$ on $\Sigma_a$. Furthermore, for $-\infty<m<0<n<\infty$ and $u=u_m\cdots u_0\in \Sigma_a^{-,*}$, $v=v_1\cdots v_n\in\Sigma_a^{+,*}$, by applying \eqref{lift} with $B_k=[u_k]$ for $k=m,\ldots,0$ and $B_1=[v_1\cdots v_n]$ we obtain
\begin{equation}\label{cyl}
\mu([u]^-\times[v])=\mu^+([uv]),
\end{equation}
where $uv=u_m\cdots u_0v_1\cdots v_n$ is the concatenation of $u$ and $v$.

\subsection{Fibre measures and dimension}\label{Fim}
Let $(\Sigma_a,\mathcal{B}_a,\sigma,\mu)$ be a m.p.d.s.. Recall the canonical mapping $\pi^-:\Sigma_a\mapsto \Sigma_a^-$. We have that $(\Sigma_a,\mathcal{B}_a,\mu)$ is a Lebesgue space, and $(\Sigma_a^-,d_\delta^-)$ is a complete separable ultra-metric space. Therefore the family
\[
\eta=\{(\pi^-)^{-1}(\bm{i}^-)\}_{\bm{i}^-\in \Sigma_a^-}
\]
forms a measurable partition. We denote by $\mu^{\eta}_{\bm{i}}$ the conditional measure of $\mu$ with respect to $\eta$ at $\mu$-a.e. $\bm{i}\in \Sigma_a$, which is supported by the fibre
\[
\eta(\bm{i})=(\pi^-)^{-1}(\bm{i}^-)=\{\bm{i}^-\}\times\Sigma_a^+.
\]
These conditional measures are defined ``for $\mu$-a.e. $\bm{i}=(\bm{i}^-,\bm{i}^+)\in \Sigma_a$", but since the measurable partition only depends on $\bm{i}^-$, it is equivalent to ``for $\mu^-$-a.e. $\bm{i}^-\in \Sigma_a^-$", where $\mu^-:=\mu\circ (\pi^-)^{-1}$ is the projection of $\mu$ to $\Sigma_a^-$. The push-forward measure
\[
\mu_{\bm{i}^-}^+:=\mu^{\eta}_{\bm{i}}\circ (\pi^+)^{-1}
\]
is then the same probability measure but defined on $\Sigma_a^+$. By Theorem \ref{Rohthm} we have
\begin{equation}\label{cly2}
\int_{\Sigma_a} f(\bm{i}) \, \mu(\mathrm{d}\bm{i})=\int_{\Sigma_a^-}\int_{\Sigma_a^+}  f(\bm{i}^-,\bm{i}^+) \, \mu_{\bm{i}^-}^+(\mathrm{d}\bm{i}^+)\, \mu^-(\mathrm{d}\bm{i}^-)
\end{equation}
for measurable functions $f$ on $(\Sigma_a,\mathcal{B}_a)$. In particular for any $B\in \mathcal{B}_a^+$,
\begin{equation}\label{fdm}
\mu^+(B)=\int_{\Sigma_a^-} \mu_{\bm{i}^-}^+(B)\, \mu^-(\mathrm{d}\bm{i}^-).
\end{equation}

In \cite[equation (10.1)]{HS12} Hochman and Shmerkin used a folklore result that $\mu^-$-typical $\mu_{\bm{i}^-}^+$ and $\mu^+$ have the same exact dimension when $\mu^+$ is ergodic. For readers' convenience we shall present a proof here.

\begin{lemma}\label{lem1}
Let $(\Sigma_a^+,\mathcal{B}_a^+,\sigma,\mu^+)$ be a one-sided m.p.d.s. and let $(\Sigma_a,\mathcal{B}_a,\sigma,\mu)$ be its natural extension. If $\mu^+$ is exact-dimensional on $(\Sigma_a^+,d_\delta^+)$, in particular if $\mu^+$ is ergodic, then for $\mu^-$-a.e. $\bm{i}^-\in\Sigma_a^-$, the fibre measure $\mu_{\bm{i}^-}^+$ is exact-dimensional on $(\Sigma_a^+,d_\delta^+)$ with the same dimension.
\end{lemma}

\begin{proof}
Consider the finite partition $\mathcal{P}_a=\{\Sigma_a^-\times[i]\}_{i=1,\cdots,a}$ of $\Sigma_a$. For $\bm{i}\in\Sigma_a$ and $n\ge 1$ it is easy to see that
\begin{align*}
\mathcal{P}_a^{n}(\bm{i})&=\ \mathcal{P}_a(\bm{i})\vee \sigma^{-1}\mathcal{P}_a(\sigma(\bm{i}))\vee \cdots \vee \sigma^{-n+1}\mathcal{P}_a(\sigma^{n-1}(\bm{i}))\\
&=\ \Sigma_a^-\times[\bm{i}^+|_n].
\end{align*}
Therefore the unique element in $\sigma^{-n}\eta$ containing $\sigma^n(\bm{i})$ satisfies
\[
(\sigma^{-n}\eta)(\sigma^n(\bm{i}))=\{\bm{i}^-\}\times[\bm{i}^+|_n]=\mathcal{P}_a^{n}(\bm{i})\cap \eta(\bm{i}),
\]
meaning that $\eta$ and $\mathcal{P}_a$ satisfy \eqref{TetaP}. Note that $\sigma$ is invertible on $\Sigma_a$, so by applying Proposition \ref{SMB} we obtain that for $\mu$-a.e. $\bm{i}\in\Sigma_a$,
\begin{equation}\label{ld11}
\lim_{n\to\infty}\frac{-\log \mu^\eta_{\bm{i}}(\mathcal{P}_a^{n}(\bm{i}))}{n}=h_\mu(\mathcal{P}_a|\widehat{\eta})(\bm{i}).
\end{equation}
Note that, by \eqref{cly2}, for $\mu$-a.e. $\bm{i}\in\Sigma_a$,
\[
\mu^\eta_{\bm{i}}(\mathcal{P}_a^{n}(\bm{i}))=\mu_{\bm{i}^-}^+([\bm{i}^+|_n]).
\]
Then \eqref{ld11} implies that, for $\mu$-a.e. $\bm{i}\in\Sigma_a$, or equivalently for $\mu^-$-a.e. $\bm{i}^-\in\Sigma_a^-$, for $\mu_{\bm{i}^-}^+$-a.e. $\bm{i}^+\in \Sigma_a^+$,
\begin{equation}\label{ld12}
\lim_{n\to\infty}\frac{\log \mu_{\bm{i}^-}^+([\bm{i}^+|_n])}{\log |[\bm{i}^+|_n]|_{d_\delta^+}}=\frac{h_\mu(\mathcal{P}_a|\widehat{\eta})(\bm{i})}{-\log \delta}.
\end{equation}
Define the positive measurable functions $f(\bm{i})=-\log \mu^\eta_{\bm{i}}(\mathcal{P}_a(\bm{i}))$ and
\[
f_k(\bm{i})=-\log \frac{\mu(B_{\pi^-}(\bm{i},\delta^{k})\cap \mathcal{P}_a(\bm{i}))}{\mu(B_{\pi^-}(\bm{i},\delta^{k}))}
\]
for $k\ge 0$, where
\begin{equation*}\label{pbb}
B_{\pi^-}(\bm{i},\delta^{k})=(\pi^-)^{-1}(B_{d_\delta^-}(\pi^-(\bm{i}),\delta^{k}))=(\pi^-)^{-1}([\bm{i}^-|_{k}]^-)=[\bm{i}^-|_{k}]^-\times\Sigma_a^+.
\end{equation*}
By \eqref{cyl} we deduce that for $\bm{i}\in\Sigma_a$, $k\ge 1$ and $u\in \Sigma_a^*$, 
\begin{align*}
\frac{\mu(B_{\pi^-}(\bm{i},\delta^{k})\cap \mathcal{P}_a(\bm{i}))}{\mu(B_{\pi^-}(\bm{i},\delta^{k}))}&=\frac{\mu(B_{\pi^-}(\bm{i},\delta^{k})\cap (\Sigma_a^-\times [u]))}{\mu(B_{\pi^-}(\bm{i},\delta^{k}))}\\
&=\frac{\mu([\bm{i}^-|_{k}]^-\times [u])}{\mu([\bm{i}^-|_{k}]^-\times\Sigma_a^+)}\\
&=\frac{\mu^+([i_{-k+1}\cdots i_0u])}{\mu^+([i_{-k+1}\cdots i_0])},\numberthis \label{sp11}
\end{align*}
and $\frac{\mu(B_{\pi^-}(\bm{i},\delta^{0})\cap \mathcal{P}_a(\bm{i}))}{\mu(B_{\pi^-}(\bm{i},\delta^{0}))}=\mu^+([u])$.

For $n\ge 1$ and $\bm{i}^+=i_1i_2\cdots$ we may write
\[
\mu^+([\bm{i}^+|_n])=\prod_{k=0}^{n-1}\frac{\mu^+([i_1\cdots i_ki_{k+1}])}{\mu^+([i_1\cdots i_{k}])}
\]
with the convention that $\mu^+([i_1\cdots i_{0}])=1$. Note that by \eqref{sp11},
\[
f_k\circ \sigma^k(\bm{i})=-\log \frac{\mu^+([i_1\cdots i_k i_{k+1}])}{\mu^+([i_1\cdots i_{k}])}
\]
so that
\[
-\log \mu^+([\bm{i}^+|_n])=\sum_{k=0}^{n-1} f_{k}\circ \sigma^k(\bm{i}).
\]
By Proposition \ref{P3.5}(1) we have that $f_k(\bm{i})\to f(\bm{i})$ as $k\to\infty$ for $\mu$-a.e. $\bm{i}\in\Sigma_a$. Furthermore, by Proposition \ref{P3.5}(2), there is a measurable function $g\in L^1(\Sigma_a,\mathcal{B}_a,\mu)$ such that $f_k(\bm{i})\le g(\bm{i})$ for $\mu$-a.e. $\bm{i}$. 
Hence, by Theorem \ref{met}, for $\mu$-a.e. $\bm{i}\in \Sigma_a$,
\begin{align*}
\lim_{n\to \infty} \frac{-\log \mu^+([\bm{i}^+|_n])}{n} = &\ \lim_{n\to \infty} \frac{1}{n} \sum_{k=0}^{n-1} f_{k}\circ \sigma^k(\bm{i})\\
= &\ \mathbb{E}_{\mu}(f|\mathcal{I})(\bm{i}).
\end{align*}
Noting that the left-hand side of the above identity does not depend on $\bm{i}^-$, and $[\bm{i}^+|_n]=B_{d_\delta^+}(\bm{i}^+,\delta^n)$, so we deduce that for $\mu$-a.e. $\bm{i}\in\Sigma_a$, hence for $\mu^+$-a.e. $\bm{i}^+\in\Sigma_a^+$,
\[
\lim_{n\to \infty} \frac{\log \mu^+(B_{d_\delta^+}(\bm{i}^+,\delta^n))}{\log \delta^n} =  \frac{\mathbb{E}_{\mu}(f|\mathcal{I})(\bm{i})}{-\log \delta}.
\]
Finally, by assumption, $\mu^+$ is exact-dimensional on $(\Sigma_a^+,d_\delta^+)$ with dimension $\frac{h}{-\log \delta}$ for some $h\ge 0$, therefore we deduce that $ \mathbb{E}_{\mu}(f|\mathcal{I})(\bm{i})=h$ for $\mu$-a.e. $\bm{i}\in \Sigma_a$.
Noting that by the definition of $f$ we have $\mathbb{E}_{\mu}(f|\mathcal{I})(\bm{i})=h_\mu(\mathcal{P}_a|\widehat{\eta})(\bm{i})$ for $\mu$-a.e. $\bm{i}\in\Sigma_a$, therefore by \eqref{ld12} we deduce that for $\mu^-$-a.e. $\bm{i}^-$, $\mu_{\bm{i}^-}^+$ is exact-dimensional on $(\Sigma_a^+,d_\delta^+)$ with the same dimension $\frac{h}{-\log \delta}$.
\end{proof}

\subsection{Mandelbrot cascades measures and dimension}\label{sec:mcm}	
	
Let $V$ be a non-negative random variable defined on a probability space $(\Omega,\mathcal{F},\mathbb{P})$ with expectation $\mathbb{E}(V)=1$. Let $\{V_{u}: u\in \Sigma_a^{+,*}\}$ be a sequence of i.i.d. random variables with the same law as $V$, indexed by the countable set $\Sigma_a^{+,*}$. Let $\lambda$ be a Borel probability measure on $(\Sigma_a^+,\mathcal{B}_a^+)$. For $k\ge 1$ define a random measure $\widetilde{\lambda}_k$ on $\Sigma_a^+$ by
\begin{equation}\label{mcm}
\widetilde{\lambda}_{k}(\mathrm{d}\bm{i}^+)=Q^V_{\bm{i}^+|_k} \, \lambda(\mathrm{d}\bm{i}^+),\ \bm{i}^+\in\Sigma_a^+,
\end{equation}
where we denote by
\[
Q^V_{\bm{i}^+|_k}=Q^V_{i_1\cdots i_k}=V_{i_1}V_{i_1i_2}\cdots V_{i_1\cdots i_k}.
\]
By assumption $\{\widetilde{\lambda}_k\}_{k\ge 1}$ forms a measure-valued martingale, hence the limit
\[
\widetilde{\lambda}=\lim_{k\to\infty} \widetilde{\lambda}_k
\]
exists almost surely. But they may be degenerate, i.e., almost surely $\|\widetilde{\lambda}\|=\widetilde{\lambda}(\Sigma_a^+)=0$. When it is non-degenerate, we call $\widetilde{\lambda}$ a Mandelbrot cascade measure of $\lambda$. 

The non-degeneracy of $\widetilde{\lambda}$ depends on the dimension of $\lambda$ and the entropy of the random variable $V$, namely
\[
h_V=\mathbb{E}(V\log V).
\]
Denote by $\overline{\dim}_P $ the upper packing dimension.  We have the following theorem by Barral and Jin \cite{BJ21}.
\begin{theorem}[Theorem 2.3 in \cite{BJ21}]\label{BJ21}
On the metric space $(\Sigma_a^+,d_\delta^+)$, if $\un{\dim}_H\, \lambda > \frac{h_V}{-\log \delta}$, then $\widetilde{\lambda}$ is non-degenerate, i.e. $\mathbb{E}(\|\widetilde{\lambda}\|)=1$. Moreover, almost surely conditioned on $\|\widetilde{\lambda}\|>0$,
\[
\un{\dim}_H\, \lambda-\frac{h_V}{-\log \delta}\le \un{\dim}_H\, \widetilde{\lambda} \le \overline{\dim}_P\, \widetilde{\lambda} \le \overline{\dim}_P\, \lambda-\frac{h_V}{-\log \delta}.
\]
In particular, if $\lambda$ is exact-dimensional with dimension $\dim \lambda$ then almost surely conditioned on $\|\widetilde{\lambda}\|> 0$, $\widetilde{\lambda}$ is exact-dimensional with dimension $\dim \lambda-\frac{h_V}{-\log \delta}$.
\end{theorem}

For future reference we shall give a more precise definition of the probability space on which the i.i.d. sequence $\{V_u:u\in \Sigma_a^{+,*}\}$ is defined. First take a random variable $V$ that is defined on a standard probability space (i.e. a Lebesgue space) $(\Omega,\mathcal{F},\mathbb{P})$. We may assume that $\Omega$ is a Polish space and take a metric $d_\Omega$ on $\Omega$ under which $\Omega$ is complete and separable, and $\mathcal{F}$ is the Borel $\sigma$-algebra w.r.t. $d_\Omega$. Define the countable product space
\[
(\Omega_a,\mathcal{F}_a,\mathbb{P}_a)=\bigotimes_{u\in \Sigma_a^{+,*}}(\Omega^{u},\mathcal{F}^{u},\mathbb{P}^{u}),
\]
where $(\Omega^{u},\mathcal{F}^{u},\mathbb{P}^{u})=(\Omega,\mathcal{F},\mathbb{P})$ for each $u\in \Sigma_a^{+,*}$. We may equip $\Omega_a$ with the product metric $d_{\Omega_a}=\max\{d_{\Omega^u}:u\in\Sigma_a^{+,*}\}$. Then we have that $(\Omega_a,\mathcal{F}_a,\mathbb{P}_a)$ is also a Lebesgue space, $(\Omega_a,d_{\Omega_a})$ is complete and separable, and that $\mathcal{F}_a$ is the Borel $\sigma$-algebra w.r.t. $d_{\Omega_a}$. For $u\in \Sigma_a^{+,*}$ define the projection
\[
\pi_{a}^{u}:\Omega_a\to\Omega^{u}.
\]
By letting $V_{u}=V\circ\pi_{a}^{u}$ for $u\in\Sigma_a^{+,*}$, we obtain a sequence of i.i.d. random variables on $(\Omega_a,\mathcal{F}_a,\mathbb{P}_a)$ with the same law as $V$.

For the Mandelbrot cascade measure $\widetilde{\lambda}$ of a measure $\lambda$, when necessary, we will use the notation $\widetilde{\lambda}_{\bm{\omega}}$ to indicate its dependence on $\bm{\omega}\in\Omega_a$. Note that for an instance of a Mandelbrot cascade measure $\widetilde{\lambda}_{\bm{\omega}}$, when non-degenerate, i.e. $\|\widetilde{\lambda}_{\bm{\omega}}\|>0$ is positive, its total mass is not necessarily equal to $1$, hence we shall denote by
\[
\bar{\lambda}_{\bm{\omega}}=\frac{\widetilde{\lambda}_{\bm{\omega}}}{\|\widetilde{\lambda}_{\bm{\omega}}\|}
\]
its normalisation to a probability measure when $\|\widetilde{\lambda}_{\bm{\omega}}\|>0$.

\subsection{The Peyri\`ere measure}

Let $\mu^+$ be a $\sigma$-ergodic probability measure on $\Sigma_a^+$ and let $\mu$ be its natural extension to $\Sigma_a$. It is easy to check that $h_{\mu^+}(\sigma,\mathcal{P}_a^+)=h_{\mu}(\sigma,\mathcal{P}_a)$, where $\mathcal{P}^+_a=\{[i]\}_{i=1,\cdots ,a}$ is the first generation cylinder partition of $\Sigma_a^+$, and recall that $\mathcal{P}_a=\{\Sigma_a^-\times[i]\}_{i=1,\cdots,a}$. From now on we shall use
\begin{equation}\label{mtemu}
h_\mu=h_{\mu^+}(\sigma,\mathcal{P}_a^+)=h_{\mu}(\sigma,\mathcal{P}_a)
\end{equation}
to denote this entropy. By SMB Theorem we have that $\mu^+$ is exact-dimensional on $(\Sigma_a^+,d_\delta^+)$ with dimension
\[
\frac{h_\mu}{-\log \delta}.
\]
Then, by Lemma \ref{lem1}, for $\mu^-$-a.e. $\bm{i}^-\in\Sigma_a^-$, $\mu_{\bm{i}^-}^+$ is exact-dimensional with the same dimension $\frac{h_{\mu}}{-\log \delta}$ on $(\Sigma_a^+,d_\delta^+)$. Hence by Theorem \ref{BJ21} we have for $\mu^-$-a.e. $\bm{i}^-$, the Mandelbrot cascade measure $\widetilde{\mu}_{\bm{i}^-}$ is non-degenerate, and for $\mathbb{P}_a$-a.e. $\bm{\omega}\in \Omega_a$ with $\|\widetilde{\mu}_{\bm{i}^-,\bm{\omega}}^+\|> 0$,  $\widetilde{\mu}_{\bm{i}^-,\bm{\omega}}^+$ is exact-dimensional with  dimension
\[
\frac{h_{\mu}-h_V}{-\log \delta}.
\]

Let $\widetilde{\mathcal{A}}_a$ be a semi-algebra of the form 
\begin{align*}
\{(\bm{i},\bm{\omega}):\ &\bm{i}\in[u]^-\times [v],V_{v_1\cdots v_k}(\bm{\omega})\in B_{v_1\cdots v_k},\; k=1,\cdots,n\}, 
\end{align*}
where $u\in \Sigma_a^{-,*}$, $n\ge 1$, $v=v_1\cdots v_n\in\Sigma_a^{+,*}$ and $\{B_{v_1\cdots v_k}\}_{1\le k\le n}$ are Borel subsets of $[0,\infty)$. Let $\widetilde{\mathcal{F}}_a$ be the $\sigma$-algebra generated by $\widetilde{\mathcal{A}}_a$. Denote by $\widetilde{\Sigma}_a=\Sigma_a\times\Omega_a$ and define a probability measure $\mathbb{Q}_{\mu}$ on $(\widetilde{\Sigma}_a,\widetilde{\mathcal{F}}_a)$ as
\begin{equation}\label{pmq}
\int_{\widetilde{\Sigma}_a} f(\bm{i},\bm{\omega})\, \mathbb{Q}_{\mu}(\mathrm{d}(\bm{i},\bm{\omega}))
=\int_{\Omega_a}\int_{\Sigma_a^-}\int_{\Sigma_a^+} f(\bm{i}^-,\bm{i}^+,\bm{\omega}) \, \widetilde{\mu}_{\bm{i}^-,\bm{\omega}}^+(\mathrm{d}\bm{i}^+)\mu^-(\mathrm{d}\bm{i}^-)\mathbb{P}_a(\mathrm{d}\bm{\omega})
\end{equation}
for $\widetilde{\mathcal{F}}_a$-measurable functions $f$. We call $\mathbb{Q}_\mu$ the Peyri\`ere measure of $\mu$.

For $\bm{\omega}=(\omega_u)_{u\in\Sigma_a^{+,*}}\in\Omega_a$ we may write it as $\bm{\omega}=(\omega_i,\bm{\omega}_i)_{i=1,\cdots,a}$, where
\[
\bm{\omega}_i=(\omega_{iu})_{u\in\Sigma_a^{+,*}}.
\]
For $i=1,\cdots,a$ define the projection
\[
\kappa_i(\bm{\omega})=\bm{\omega}_i.
\]
We may define a skew product, denoted by $\widetilde{\sigma}$, on the product space $\widetilde{\Sigma}_a$ by
\[
\widetilde{\sigma}(\bm{i},\bm{\omega})=(\sigma(\bm{i}), \kappa_{i_1}(\bm{\omega})).
\]
Note that $\widetilde{\sigma}^2(\bm{i},\bm{\omega})=\widetilde{\sigma}(\sigma(\bm{i}), \kappa_{i_1}(\bm{\omega}))=(\sigma^2(\bm{i}),\kappa_{i_2}(\kappa_{i_1}(\bm{\omega}))$, and
\[
\kappa_{i_2}(\kappa_{i_1}(\bm{\omega})=\kappa_{i_2}((\omega_{i_1u})_{u\in\Sigma_a^{+,*}})=\kappa_{i_2}((\omega_{i_1i},(\omega_{i_1iu})_{u\in\Sigma_a^{+,*}})_{i=1,\cdots,a})=(\omega_{i_1i_2u})_{u\in\Sigma_a^{+,*}}.
\]
Therefore for $i_1\cdots i_n\in\Sigma_a^{+,*}$ and $\bm{\omega}=(\omega_u)_{u\in\Sigma_a^{+,*}}$ we may denote by
\begin{equation}\label{omegau}
\kappa_{i_n}\circ \cdots\circ \kappa_{i_1}(\bm{\omega})=(\omega_{i_1\cdots i_nu})_{u\in\Sigma_a^{+,*}}:=\bm{\omega}_{i_1\cdots i_n},
\end{equation}
so that for $\bm{i}=\cdots i_0i_1i_2\cdots \in \Sigma_a$,
\[
\widetilde{\sigma}^n(\bm{i},\bm{\omega})=(\sigma^n(\bm{i}),\bm{\omega}_{i_1\cdots i_n}).
\]
The skew product $\widetilde{\sigma}$ is not necessarily invertible, since for $i=1,\cdots, a$ the pre-image $\kappa_i^{-1}(\bm{\omega})$ takes the form
\[
\kappa_i^{-1}(\bm{\omega})=\{(\omega_j,\bm{\omega}_j)_{j=1,\cdots,a}:\omega_j\in \Omega,\bm{\omega}_j\in\Omega_a \text{ and }\bm{\omega}_j=\bm{\omega} \text{ if }j=i\}.
\]

First we have the following lemma about the invariance of Peyri\`ere measures.

\begin{lemma}\label{lminv}
The Peyri\`ere measure $\mathbb{Q}_{\mu}$ is $\widetilde{\sigma}$-invariant.
\end{lemma}

\begin{proof}
Recall in Sections \ref{Fim} that $\eta=\{(\pi^-)^{-1}(\bm{i}^-)\}_{\bm{i}^-\in\Sigma_a^-}$ is a measurable partition, and that $\eta$ and $\mathcal{P}_a$ satisfy \eqref{TetaP}. Recall that for $n\ge 1$ and $\bm{i}\in\Sigma_a$,
\[
\mathcal{P}_a^{n}(\bm{i})=\Sigma_a^-\times[\bm{i}^+|_n].
\]
By \eqref{inid} we have that for any $B\in\mathcal{B}_a$, for $\mu$-a.e. $\bm{i}\in\Sigma_a$ and $n\ge 1$,
\[
\mu^{\eta}_{\sigma^n(\bm{i})}(\sigma^n(B))=\frac{\mu^\eta_{\bm{i}}(B\cap \mathcal{P}^{n}_a(\bm{i}))}{\mu^\eta_{\bm{i}}(\mathcal{P}^{n}_a(\bm{i}))}
\]
In particular if we choose $B=\Sigma_a^-\times [uv]$ for some $u=u_1\cdots u_n,v\in\Sigma_a^{+,*}$, then we have $\sigma^n(B)=[u]^-\times[v]$ and $B\cap \mathcal{P}_a^{n}(\bm{i})=B$ for $\bm{i}\in \Sigma_a^-\times [u]$, therefore, for $\mu$-a.e. $\bm{i}=\cdots i_0i_1\cdots \in\Sigma_a^-\times [u]$,
\[
\frac{\mu^{+}_{\bm{i}^-}([uv])}{\mu^{+}_{\bm{i}^-}([u])}=\frac{\mu^\eta_{\bm{i}}(B\cap \mathcal{P}^{n}_a(\bm{i}))}{\mu^\eta_{\bm{i}}(\mathcal{P}^{n}(\bm{i}))}=\mu^\eta_{\sigma^n(\bm{i})}(\sigma^n(B))=\mu^{+}_{\sigma^n(\bm{i})^-}([v])=\mu^{+}_{\bm{i}^-u}([v]).
\]
This implies that for $\mu^-$-a.e. $\bm{i}^-\in\Sigma_a^-$ and $u,v\in \Sigma_a^{+,*}$,
\begin{equation}\label{cdm12}
\frac{\mu^{+}_{\bm{i}^-}([uv])}{\mu^{+}_{\bm{i}^-}([u])}=\mu^{+}_{\bm{i}^-u}([v]).
\end{equation}
Together with \eqref{mcm} we obtain the following factorisation:  for $\mu^-$-a.e. $\bm{i}^-\in\Sigma_a^-$, for $\mathbb{P}_a$-a.e. $\bm{\omega}$, for $u,v\in\Sigma_a^{+,*}$,
\begin{equation}\label{ssss}
\widetilde{\mu}_{\bm{i}^-,\bm{\omega}}^+([u v])=Q^V_{u}(\bm{\omega})\cdot\mu_{\bm{i}^-}^+([u])\cdot \widetilde{\mu}_{\bm{i}^-u,\bm{\omega}_{u}}^{+}([v]).
\end{equation}
Taking $u=1,\cdots,a$ this implies that, for a measurable function $f$ on $(\widetilde{\Sigma}_a,\widetilde{\mathcal{F}}_a)$,
\begin{align*}
&\int_{\widetilde{\Sigma}_a} f\circ \widetilde{\sigma}(\bm{i},\bm{\omega})\, \mathbb{Q}_{\mu}(\mathrm{d}(\bm{i},\bm{\omega}))\\
=&\  \int_{\Omega_a}\int_{\Sigma_a^-}\int_{\Sigma_a^+}  f(\bm{i}^-i_1,\sigma\bm{i}^+,\kappa_{i_1}(\bm{\omega}))\, \widetilde{\mu}^+_{\bm{i}^-,\bm{\omega}}(\mathrm{d}\bm{i}^+)\mu^-(\mathrm{d}\bm{i}^-)\mathbb{P}_a(\mathrm{d}\bm{\omega})\\
=&\  \int_{\Omega_a}\int_{\Sigma_a^-}\int_{\Sigma_a^+} \sum_{u=1}^a f(\bm{i}^-u,\bm{j}^+,\bm{\omega}_u)\, \mu_{\bm{i}^-}([u])V_{u}\cdot \widetilde{\mu}_{\bm{i}^-u,\bm{\omega}_u}^{+}(\mathrm{d}\bm{j}^+)\mu^-(\mathrm{d}\bm{i}^-)\mathbb{P}_a(\mathrm{d}\bm{\omega})\\
=&\  \int_{\Omega_a}\int_{\Sigma_a^-}\int_{\Sigma_a^+} \sum_{u=1}^a  f(\bm{i}^-u,\bm{j}^+,\bm{\omega})\, \mu_{\bm{i}^-}([u])\widetilde{\mu}^+_{\bm{i}^-u,\bm{\omega}}(\mathrm{d}\bm{j}^+)\mu^-(\mathrm{d}\bm{i}^-)\mathbb{P}_a(\mathrm{d}\bm{\omega})\\
=&\  \int_{\Omega_a}\int_{\Sigma_a^-}\int_{\Sigma_a^+}f(\bm{j}^-,\bm{j}^+,\bm{\omega})\, \widetilde{\mu}^+_{\bm{j}^-,\bm{\omega}}(\mathrm{d}\bm{j}^+)\mu^-(\mathrm{d}\bm{j}^-)\mathbb{P}_a(\mathrm{d}\bm{\omega})\\
=&\ \int_{\widetilde{\Sigma}_a} f(\bm{i},\bm{\omega})\, \mathbb{Q}_{\mu}(\mathrm{d}(\bm{i},\bm{\omega})),
\end{align*}
where we have used the facts that $V_u$ has expectation $1$, $f(\bm{i}^-u,\bm{j}^+,\bm{\omega}_u)\,\widetilde{\mu}_{\bm{i}^-u,\bm{\omega}_u}^{+}(\mathrm{d}\bm{j}^+)$ is independent of  $V_u$ and has the same law as $f(\bm{i}^-u,\bm{j}^+,\bm{\omega})\widetilde{\mu}^+_{\bm{i}^-u,\bm{\omega}}(\mathrm{d}\bm{j}^+)$ under $\mathbb{P}_a$, and that, by \eqref{cdm12},
\[
\mu_{\bm{i}^-}([u])\widetilde{\mu}^+_{\bm{i}^-u,\bm{\omega}}(\mathrm{d}\bm{j}^+)\mu^-(\mathrm{d}\bm{i}^-)=\widetilde{\mu}^+_{\bm{j}^-,\bm{\omega}}(\mathrm{d}\bm{j}^+)\mu^-(\mathrm{d}\bm{j}^-)
\]
for $\bm{j}^-=\bm{i}^-u\in [u]^-$.
\end{proof}

Then we have the following lemma about the ergodicity.
\begin{lemma}\label{erg}
The Peyri\`ere measure $\mathbb{Q}_{\mu}$ is $\widetilde{\sigma}$-ergodic.
\end{lemma}
\begin{proof}	
It is enough to prove that for $B_1,B_2\in\widetilde{\mathcal{A}}_a$ with $\mathbb{Q}_{\mu}(B_1)\cdot\mathbb{Q}_{\mu}(B_2)>0$, we have
	\begin{equation}
	\label{eqn:ergsuff1}
	\lim_{k\to\infty}\mathbb{Q}_{\mu}(\widetilde{\sigma}^{-k}(B_1)\cap B_2)>0.
	\end{equation}
Note that since the sets $B_1,B_2\in\widetilde{\mathcal{A}}_a$ are only involved with random variables $V_u$ with words $u\in\Sigma_a^{+,*}$ of bounded length, for $k$ large enough all the random variables involved in the sets $\widetilde{\sigma}^{-k}B_1$ and $B_2$ will be independent. More precisely, for $l=1,2$ take $(u^l,v^l)=(u_{m_l}^l\cdots u_{0}^l,v_1^l\cdots v_{n_l}^l)$ such that
\begin{align*}
B_l=\{(\bm{i},\bm{\omega})&:\bm{i}\in[u^l]^-\times[v^l],V_{v^l_1\cdots v^l_k}\in B_{v^l_1\cdots v^l_k},\; k=1,\dots,n_l\}.
\end{align*}
For $k>n_2$ we have that all the words related to the random variables in $\{V_u\}_{u\in \Sigma_a^{+,*}}$ appeared in $\widetilde{\sigma}^{-k}B_1$ have word lengths at least $n_2+1$, and therefore all the random variables appeared in $\widetilde{\sigma}^{-k}B_1$ are independent of $B_2$. This implies that  
\begin{align*}
\mathbb{Q}_{\mu}(\widetilde{\sigma}^{-k}(B_1)\cap B_2)=\prod_{l=1,2}\prod_{k=1}^{n_l}\mathbb{E}_{\mathbb{P}_a}(V_{v^l_1\cdots v^l_k}\mathbf{1}_{\{V_{v^l_1\cdots v^l_k}\in B_{v^l_1\cdots v^l_k}\}})\times\mu(\sigma^{-k}(U_1)\cap U_2),
\end{align*}
where for $l=1,2$, $U_l$ are the projection of $B_l$ to $\Sigma_a$. From $\mathbb{Q}_{\mu}(B_1)\cdot\mathbb{Q}_{\mu}(B_2)>0$ we know that for $l=1,2$,
\[
\prod_{k=1}^{n_l}\mathbb{E}_{\mathbb{P}_a}(V_{v^l_1\cdots v^l_k}\mathbf{1}_{\{V_{v^l_1\cdots v^l_k}\in B_{v^l_1\cdots v^l_k}\}})>0 \text{ and } \mu(U_l)>0.
\]
Then, since $\mu$ is $\sigma$-ergodic, we have that for $k$ large enough, $\mu(\sigma^{-k}(U_1)\cap U_2)>0$. Therefore, \eqref{eqn:ergsuff1} holds.
\end{proof}

\section{Iterated function system and exact-dimensionality}\label{ifsed} 

In this section we shall consider a general IFS $\bm{I}=\{f_i\}_{i=1}^a$ of the form
\[
f_i(x)=r_i x+t_i,\; r_i\in(0,1), t_i\in\mathbb{R},
\]
so that the exact-dimensional result in this section is proved in its full generality for self-similar IFSs. For a finite word $i_1\cdots i_k\in \{1,\ldots,a\}^k$ denote by $f_{i_1\cdots i_k}=f_{i_1}\circ\cdots\circ f_{i_k}$. We define the canonical mapping $\Phi_{\bm{I}}\colon \Sigma_a^+\to \mathbb{R}$ as 
\[
\Phi_{\bm{I}}(\bm{i}^+)=\lim_{k\to\infty}f_{\bm{i}^+|_k}(0).
\]
Denote by $\mathcal{B}_{\bm{I}}$ the $\sigma$-field generated by $\Phi_{\bm{I}}^{-1}(\mathcal{B}(\mathbb{R}))$. Let $A_{\bm{I}}=\Phi_{\bm{I}}(\Sigma_a^+)$ be the self-similar set (or the attractor) of $\bm{I}$. Let $R_{\bm{I}}=|A_{\bm{I}}|$ be the diameter of $A_{\bm{I}}$.
	
For a $\sigma$-ergodic measure $\mu^+$ on $\Sigma_a^+$, the push-forward measure
\[
\Phi_{\bm{I}}(\mu^+)=\mu^+\circ \Phi_{\bm{I}}^{-1}
\]
defines a probability measure on $A_{\bm{I}}$, which is called a self-similar measure if $\mu^+$ is a Bernoulli measure on $\Sigma_a^+$. The exact-dimensionality of $\Phi_{\bm{I}}(\mu^+)$ is proved by Feng and Hu \cite{FH09}, and in our notation the dimension is equal to
\[
\dim \Phi_{\bm{I}}(\mu^+)=\frac{h_\mu-h_{\mu,\bm{I}}}{\chi_{\mu,\bm{I}}},
\]
where $h_\mu$ is the measure-theoretic entropy, $h_{\mu,\bm{I}}$ is the conditional entropy that reflects how $\bm{I}$ overlaps with respect to $\mu$, and
\[
\chi_{\mu,\bm{I}}=-\sum_{i=1}^a \mu^+([i])\log r_i
\]
is the associated Lyapunov exponent. In general the conditional entropy $h_{\mu,\bm{I}}$ is implicit, but by the recent work of Jordan and Rapaport \cite{JRa21} we may compute this entropy when the IFS $\bm{I}$ satisfies the ESC.

The exact-dimensionality of $\Phi_{\bm{I}}(\widetilde{\mu}^+)$ when $\mu^+$ is a Bernoulli measure on $\Sigma_a^+$ is proved in \cite{FJ14}. Here we shall prove a more general result for $\Phi_{\bm{I}}(\widetilde{\mu}^+)$ by using the exact-dimensionality of $\Phi_{\bm{I}}(\widetilde{\mu}^+_{\bm{i}^-})$ for  $\mu^{-}$-a.e. $\bm{i}^-$, where $\mu$ is the natural extension of $\mu^+$ to $\Sigma_a$, $\mu^-=\mu\circ (\pi^-)^{-1}$ and $\mu^+_{\bm{i}^-}$ is the fibre measures on $\Sigma_a^+$.

\subsection{Exact-dimensionality of push-forward fibre cascade measures}

Recall the first generation cylinder partition $\mathcal{P}_a^+=\{[i]\}_{i=1,\cdots, a}$. Recall the conditional information
\[
I_{\lambda}(\mathcal{P}_a^+|\mathcal{B}_{\bm{I}})=\sum_{B\in\mathcal{P}_a^+}-\mathbf{1}_{B}\log \mathbb{E}_{\lambda}(\mathbf{1}_{B}|\mathcal{B}_{\bm{I}})
\]
of a measure $\lambda$ on $\Sigma_a^+$. Define
\begin{align*}
h_{\mu,V,\bm{I}}=&\ \int_{\Omega_a} \int_{\Sigma_a^-} \int_{\Sigma_a^+} I_{\bar{\mu}^+_{\bm{i}^-,\bm{\omega}}}(\mathcal{P}_a^+|\mathcal{B}_{\bm{I}})(\bm{i}^+) \,\widetilde{\mu}^+_{\bm{i}^-,\bm{\omega}}(\mathrm{d} \bm{i}^+) \mu^-(\mathrm{d}\bm{i}^-) \mathbb{P}_a(\mathrm{d}\bm{\omega})\\
=&\ \int_{\Omega_a} \int_{\Sigma_a^-} H_{\bar{\mu}^+_{\bm{i}^-,\bm{\omega}}}(\mathcal{P}_a^+|\mathcal{B}_{\bm{I}}) \cdot \|\widetilde{\mu}^+_{\bm{i}^-,\bm{\omega}}\| \, \mu^-(\mathrm{d}\bm{i}^-) \mathbb{P}_a(\mathrm{d}\bm{\omega}).
\end{align*}

\begin{theorem}\label{ed1}
For $\mu^{-}$-a.e. $\bm{i}^-\in \Sigma_a^-$, for $\mathbb{P}_a$-a.e. $\bm{\omega}\in \Omega_a$ with $\|\widetilde{\mu}_{\bm{i}^-,\bm{\omega}}^+\|>0$, $\Phi_{\bm{I}}(\widetilde{\mu}^+_{\bm{i}^-,\bm{\omega}})$  is exact dimensional with dimension
\[
D_{\mu,V,\bm{I}}=\frac{h_{\mu}-h_V-h_{\mu,V,\bm{I}}}{\chi_{\mu,\bm{I}}}.
\]
\end{theorem}

\begin{proof}

For $u\in\Sigma_a^{+,*}$ and a set $B\in \mathcal{B}_a^+$ we denote by
\[
uB=\{u\bm{i}^+:\bm{i}^+\in B\}.
\]
By using \cite[Lemma 3.7]{FJ14} we have that for $\bm{i}^+=i_1i_2\cdots\in\Sigma_a^+$ and $r>0$,
\begin{align}
\Phi_{\bm{I}}^{-1}(B(\Phi_{\bm{I}}(\bm{i}),r_{\bm{i}^+|_1}r))\cap \mathcal{P}_a^+(\bm{i}^+)=&\ \sigma^{-1}\Phi_{\bm{I}}^{-1}(B(\Phi_{\bm{I}}(\sigma\bm{i}^+),r))\cap \mathcal{P}_a^+(\bm{i}^+)\nonumber\\
=&\ i_1\Phi_{\bm{I}}^{-1}(B(\Phi_{\bm{I}}(\sigma\bm{i}^+),r)).\label{l3.7}
\end{align}
For $n\ge 0$ denote by
\[
B_{{\bm{I}}}(\bm{i}^+,n)=\Phi_{\bm{I}}^{-1}(B(\Phi_{\bm{I}}(\bm{i}^+),R_{\bm{I}}\cdot r_{\bm{i}^+|_n})),
\]
where $r_{\bm{i}^+|_n}=r_{i_1}\cdots r_{i_n}$ for $n\ge 1$ and $r_{\emptyset}=1$. Then by \eqref{l3.7} we have that for $n\ge 1$ and $0\le k\le n-1$,
\begin{align*}
B_{{\bm{I}}}(\sigma^k\bm{i}^+,n-k)\cap \mathcal{P}_a^+(\sigma^k\bm{i}^+)=&\ \sigma^{-1} B_{{\bm{I}}}(\sigma^{k+1}\bm{i}^+,n-k-1)\cap \mathcal{P}_a^+(\sigma^k\bm{i}^+)\\
=&\ i_{k+1}B_{{\bm{I}}}(\sigma^{k+1}\bm{i}^+,n-k-1).
\end{align*}
This implies
\begin{align*}
\frac{\bar{\mu}_{\bm{i}^-,\bm{\omega}}^+(\bm{i}^+|_kB_{{\bm{I}}}(\sigma^k\bm{i}^+,n-k))}{\bar{\mu}_{\bm{i}^-,\bm{\omega}}^+(\bm{i}^+|_{k+1}B_{{\bm{I}}}(\sigma ^{k+1}\bm{i},n-k-1))}=&\ \frac{\bar{\mu}_{\bm{i}^-,\bm{\omega}}^+(\bm{i}^+|_kB_{{\bm{I}}}(\sigma^k\bm{i}^+,n-k))}{\bar{\mu}_{\bm{i}^-,\bm{\omega}}^+(\bm{i}^+|_k(B_{{\bm{I}}}(\sigma^k\bm{i},n-k)\cap \mathcal{P}_a^+(\sigma^k\bm{i}^+)))}.
\end{align*}
For $u, v\in\Sigma_a^{+,*}$ with $\widetilde{\mu}^+_{\bm{i}^-,\bm{\omega}}([u])>0$, by \eqref{ssss} we have
\[
\bar{\mu}_{\bm{i}^-,\bm{\omega}}^+([uv])=\frac{\widetilde{\mu}^+_{\bm{i}^-,\bm{\omega}}([uv])}{\widetilde{\mu}^+_{\bm{i}^-,\bm{\omega}}([u])}=\frac{Q^V_u(\bm{\omega})\cdot \mu^+_{\bm{i}^-}([u])\cdot \widetilde{\mu}_{\bm{i}^-u,\bm{\omega}_u}^{+}([v])}{Q^V_u(\bm{\omega})\cdot \mu^+_{\bm{i}^-}([u])\cdot \widetilde{\mu}_{\bm{i}^-u,\bm{\omega}_u}^{+}(\Sigma_a^+)}=\bar{\mu}_{\bm{i}^-u,\bm{\omega}_u}^{+}([v]).
\]
This implies that for $\mu^{-}$-a.e. $\bm{i}^-$ and $\mathbb{P}_a$-a.e. $\bm{\omega}$ with $\widetilde{\mu}^+_{\bm{i}^-,\bm{\omega}}([\bm{i}|_n])>0$ we have
\begin{align}
\frac{\widetilde{\mu}^+_{\bm{i}^-,\bm{\omega}}(B_{{\bm{I}}}(\bm{i},n))}{\widetilde{\mu}^+_{\bm{i}^-,\bm{\omega}}([\bm{i}|_n])}=&\ \prod_{k=0}^{n-1} \frac{\bar{\mu}^+_{\bm{i}^-,\bm{\omega}}(\bm{i}^+|_kB_{{\bm{I}}}(\sigma^k\bm{i}^+,n-k))}{\bar{\mu}^+_{\bm{i}^-,\bm{\omega}}(\bm{i}^+|_{k+1}B_{{\bm{I}}}(\sigma ^{k+1}\bm{i}^+,n-k-1))}\nonumber\\
=&\ \prod_{k=0}^{n-1}\frac{\bar{\mu}^+_{\bm{i}^-,\bm{\omega}}(\bm{i}^+|_kB_{{\bm{I}}}(\sigma^k\bm{i}^+,n-k))}{\bar{\mu}^+_{\bm{i}^-,\bm{\omega}}(\bm{i}^+|_k(B_{{\bm{I}}}(\sigma^k\bm{i}^+,n-k)\cap \mathcal{P}_a^+(\sigma^k\bm{i}^+)))}\nonumber\\
=&\ \prod_{k=0}^{n-1}\frac{\bar{\mu}_{\bm{i}^-i_1\cdots i_k,\bm{\omega}_{i_1\cdots i_k}}^{+}(B_{{\bm{I}}}(\sigma^k\bm{i}^+,n-k))}{\bar{\mu}_{\bm{i}^-i_1\cdots i_k,\bm{\omega}_{i_1\cdots i_k}}^{+}(B_{{\bm{I}}}(\sigma^k\bm{i}^+,n-k)\cap \mathcal{P}_a^+(\sigma^k\bm{i}^+))}.\label{ddds}
\end{align}
For $n\ge 0$ define the non-negative measurable function on $\Sigma_a\times \Omega_a$:
\[
f_n(\bm{i},\bm{\omega})=-\log \frac{\bar{\mu}^+_{\bm{i}^-,\bm{\omega}}(B_{{\bm{I}}}(\bm{i}^+,n)\cap \mathcal{P}_a^+(\bm{i}^+))}{\bar{\mu}^+_{\bm{i}^-,\bm{\omega}}(B_{{\bm{I}}}(\bm{i}^+,n))},
\]
with the convention that $\frac{0}{0}=1$. For $k\ge 1$ we have for $(\bm{i},\bm{\omega})\in \Sigma_a\times \Omega_a$,
\[
\widetilde{\sigma}^k(\bm{i},\bm{\omega})=(\sigma^k\bm{i}, \bm{\omega}_{i_1\cdots i_k}).
\]
Therefore
\[
f_{n-k}\circ \widetilde{\sigma}^k(\bm{i},\bm{\omega})=\log \frac{\bar{\mu}_{\bm{i}^-i_1\cdots i_k,\bm{\omega}_{i_1\cdots i_k}}^{+}(B_{{\bm{I}}}(\sigma^k\bm{i}^+,n-k))}{\bar{\mu}_{\bm{i}^-i_1\cdots i_k,\bm{\omega}_{i_1\cdots i_k}}^{+}(B_{{\bm{I}}}(\sigma^k\bm{i}^+,n-k)\cap \mathcal{P}_a^+(\sigma^k\bm{i}^+))}.
\]
Thus by \eqref{ddds} we obtain that
\[
\log \frac{\widetilde{\mu}^+_{\bm{i}^-,\bm{\omega}}(B_{{\bm{I}}}(\bm{i},n))}{\widetilde{\mu}^+_{\bm{i}^-,\bm{\omega}}([\bm{i}|_n])} = \frac{1}{n}\sum_{k=0}^{n-1} f_{n-k}\circ \widetilde{\sigma}^k(\bm{i},\bm{\omega}).
\]
By Proposition \ref{P3.5} we have that for $\mu^{-}$-a.e. $\bm{i}^-$, for $\mathbb{P}_a$-a.e. $\bm{\omega}$ with $\|\widetilde{\mu}_{\bm{i}^-,\bm{\omega}}\|>0$, for $\bar{\mu}_{\bm{i}^-,\bm{\omega}}$-a.e. $\bm{i}^+$,
\[
\lim_{n\to\infty}f_n(\bm{i}^-,\bm{i}^+,\bm{\omega})=I_{\bar{\mu}^+_{\bm{i}^-,\bm{\omega}}}(\mathcal{P}_a^+|\mathcal{B}_{\bm{I}})(\bm{i}^-,\bm{i}^+),
\]
and there is a $L^1(\Sigma_a^+,\mathcal{B}_a^+,\bar{\mu}^+_{\bm{i}^-,\bm{\omega}})$ function $g=g(\bm{i}^-,\cdot,\bm{\omega})$ with $\int g \,\mathrm{d}\bar{\mu}^+_{\bm{i}^-,\bm{\omega}} \le H_{\bar{\mu}^+_{\bm{i}^-,\bm{\omega}}}(\mathcal{P}_a^+)+C\le a\log a+C$ such that $f_n\le g$ holds for all $n\ge 0$. Hence by the exact-dimensionality of $\mu^-$-a.e. $\widetilde{\mu}^+_{\bm{i}^-,\bm{\omega}}$ on $(\Sigma_a^+,d_\delta^+)$ from Theorem \ref{BJ21}, and then by Lemma \ref{erg} and Theorem \ref{met}, we have that for $\mathbb{Q}_\mu$-a.e. $(\bm{i},\bm{\omega})$, or equivalently, for $\mu^-$-a.e. $\bm{i}^-$, for $\mathbb{P}_a$-a.e. $\bm{\omega}$ with $\|\widetilde{\mu}^+_{\bm{i}^-,\bm{\omega}}\|>0$, for $\widetilde{\mu}^+_{\bm{i}^-,\bm{\omega}}$-a.e. $\bm{i}^+$,
\begin{align*}
\lim_{n\to\infty}\frac{-\log\widetilde{\mu}_{\bm{i}^-}(B_{{\bm{I}}}(\bm{i}^+,n))}{n}=&\ \lim_{n\to\infty}\frac{-\log \widetilde{\mu}^+_{\bm{i}^-}([\bm{i}^+|_n])}{n}-\frac{1}{n}\sum_{k=0}^{n-1} f_{n-k}\circ \widetilde{\sigma}_a^k(\bm{i}^-,\bm{i}^+,\bm{\omega})\\
=&\ h_{\mu}-h_V-h_{\mu,V,\bm{I}}.
\end{align*}
Finally, for the Lyapunov exponent, we have that, by Birkhoff ergodic theorem and \eqref{mucd}, $\mathbb{Q}_\mu$-a.s.,
\begin{align*}
\lim_{n\to\infty}\frac{-\log r_{\bm{i}|_n}}{n}=&\ \int_{\Omega_a}\int_{\Sigma_a^-}\int_{\Sigma_a^+} -\log r_{\bm{i}^+|_1} \, \widetilde{\mu}^+_{\bm{i}^-,\bm{\omega}}(\mathrm{d}\bm{i}^+)\mu^-(\mathrm{d}\bm{i}^-)\mathbb{P}_a(\mathrm{d}\bm{\omega})\\
=&\ \int_{\Sigma_a^+} -\log r_{\bm{i}^+|_1} \, \mu^+(\mathrm{d}\bm{i}^+)\\
=&\ \sum_{i=1}^a -\mu^+([i])\log r_i=\chi_{\mu,\bm{I}}.
\end{align*}
This yields the conclusion.
\end{proof}

\subsection{Exact-dimensionality of push-forward cascade measures}

Recall that by SMB Theorem we have that $\mu^+$ is exact-dimensional on $(\Sigma_a^+,d_\delta^+)$ with dimension
\[
\frac{h_\mu}{-\log \delta}.
\]
Hence by Theorem \ref{BJ21} we have that the Mandelbrot cascade measure $\widetilde{\mu}^+$ is non-degenerate, and for $\mathbb{P}_a$-a.e. $\bm{\omega}\in \Omega_a$ with $\|\widetilde{\mu}^+_{\bm{\omega}}\|>0$, $\widetilde{\mu}^+_{\bm{\omega}}$ is exact-dimensional on $(\Sigma_a^+,d_\delta^+)$ with the same dimension as the fibre cascade measures $\widetilde{\mu}^+_{\bm{i}^-}$, which is
\[
\frac{h_{\mu}-h_V}{-\log \delta}.
\]
We expect the same shall hold when we push these measures to the Euclidean space through the canonical mapping $\Phi_{\bm{I}}$. By \eqref{fdm} we may deduce that, by using the uniform integrable martingale convergence theorem (see for example the proof of Corollary 2.12 in \cite{BJ21}), for $\mathbb{P}_a$-a.e. $\bm{\omega}\in \Omega_a$ with $\|\widetilde{\mu}^+_{\bm{\omega}}\|>0$, for $u\in\Sigma_a^{+,*}$,
\begin{equation}\label{mucd}
\widetilde{\mu}_{\bm{\omega}}^+([u])=\int_{\Sigma_a^-} \widetilde{\mu}_{\bm{i}^-,\bm{\omega}}^+([u]) \, \mu^-(\mathrm{d} \bm{i}^-).
\end{equation}
Therefore, by the definition of lower Hausdorff dimension and Theorem \ref{ed1}, we have that for $\mathbb{P}_a$-a.e. $\bm{\omega}$ with $\|\widetilde{\mu}^+_{\bm{\omega}}\|>0$,
\[
\underline{\dim}_H \Phi_{\bm{I}}(\widetilde{\mu}^+_{\bm{\omega}})\ge \mathrm{essinf}_{\bm{i}^-\sim \mu^-} \underline{\dim}_H \Phi_{\bm{I}}(\widetilde{\mu}^+_{\bm{i}^-,\bm{\omega}})=D_{\mu,V,\bm{I}}.
\]
The other side of the inequality is much harder to derive. The following theorem is essential to our proofs. This result ensures that, using fibre measures is enough to provide sharp lower bound of the dimension. The proof follows \cite[Section 10.2]{LY85II}, see also \cite[Proof of (C2)]{FH09,Feng20}.

\begin{theorem}\label{ed22}
For $\mathbb{P}_a$-a.e. $\bm{\omega}$ with $\|\widetilde{\mu}^+_{\bm{\omega}}\|>0$, $\Phi_{\bm{I}}(\widetilde{\mu}^+_{\bm{\omega}})$  is exact-dimensional with dimension $D_{\mu,V,\bm{I}}$.
\end{theorem}

\begin{proof}
We only need to show that $\overline{\dim}_H \Phi_{\bm{I}}(\widetilde{\mu}_{\bm{\omega}}^+) \le D_{\mu,V,\bm{I}}$. 

For $\bm{i}^+\in\Sigma^+_a$ and $r>0$ recall the notation $B_{\Phi_{\bm{I}}}(\bm{i}^+,r)=\Phi_{{\bm{I}}}^{-1}B(\Phi_{{\bm{I}}}(\bm{i}^+),r)$. For $(\bm{i}^-,\bm{i}^+,\bm{\omega})\in\Sigma_a\times\Omega_a$ define
\[
\bar{s}_{0}(\bm{i}^-,\bm{i}^+,\bm{\omega})=\limsup_{r\to 0} \frac{\log \widetilde{\mu}_{\bm{\omega}}^+(B_{\Phi_{\bm{I}}}(\bm{i}^+,r))}{\log r}
\]
and
\[
\bar{s}_{1}(\bm{i}^-,\bm{i}^+,\bm{\omega})=\limsup_{r\to 0} \frac{\log \widetilde{\mu}^+_{\bm{i}^-,\bm{\omega}}(B_{\Phi_{\bm{I}}}(\bm{i}^+,r))}{\log r}.
\]
Then, by Theorem \ref{ed1}, it is enough to show that for $\mu^-$-a.e. $\bm{i}^-$, for $\mathbb{P}_a$-a.e. $\bm{\omega}$ with $\|\widetilde{\mu}^+_{\bm{i}^-,\bm{\omega}}\|>0$, for $\widetilde{\mu}^+_{\bm{i}^-,\bm{\omega}}$-a.e. $\bm{i}^+$,
\[
\bar{s}_{0}(\bm{i}^-,\bm{i}^+,\bm{\omega})\le \bar{s}_{1}(\bm{i}^-,\bm{i}^+,\bm{\omega}).
\]

Write $h=h_{\mu}-h_V$ and $\chi=\chi_{\mu,\bm{I}}$. By Lemma \ref{lem1} and Theorem \ref{BJ21}, we have that for $\mathbb{P}_a$-a.e. $\bm{\omega}$ with $\|\widetilde{\mu}_{\bm{\omega}}^+\|>0$, 
\begin{equation}\label{x1}
\lim_{n\to\infty} \frac{-\log\widetilde{\mu}_{\bm{\omega}}^+([\bm{i}^+|_n])}{n}=h \text{ and } \lim_{n\to\infty}\frac{-\log r_{\bm{i}^+|n}}{n}=\chi \text{ for } \widetilde{\mu}_{\bm{\omega}}^+\text{-a.e. } \bm{i}^+\in\Sigma_a^+;
\end{equation}
and for $\mu^-$-a.e. $\bm{i}^-\in\Sigma_a^-$ with $\|\widetilde{\mu}^+_{\bm{i}^-,\bm{\bm{\omega}}}\|>0$,
\begin{equation}\label{x2}
\lim_{n\to\infty} \frac{-\log\widetilde{\mu}^+_{\bm{i}^-,\bm{\bm{\omega}}}([\bm{i}^+|_n])}{n}=h \text{ for } \widetilde{\mu}^+_{\bm{i}^-,\bm{\bm{\omega}}}\text{-a.e. } \bm{i}^+\in\Sigma_a^+.
\end{equation}
We claim that for $\mu^-$-a.e. $\bm{i}^-$, for $\mathbb{P}_a$-a.e. $\bm{\omega}$ with $\|\widetilde{\mu}^+_{\bm{i}^-,\bm{\omega}}\|>0$, for $\widetilde{\mu}^+_{\bm{i}^-,\bm{\omega}}$-a.e. $\bm{i}^+$,
\[
0=\frac{h-h}{-\chi}\ge \bar{s}_{0}(\bm{i}^-,\bm{i}^+,\bm{\omega})-\bar{s}_{1}(\bm{i}^-,\bm{i}^+,\bm{\omega}).
\]
Suppose that the above statement is not true. Then for $\mathbb{P}_a$-positive $\bm{\omega}$ there is a subset $U_{\bm{\omega}}\subset \Sigma_a$ with $\int \mathbf{1}_{U_{\bm{\omega}}}(\bm{i}^-,\bm{i}^+) \widetilde{\mu}^+_{\bm{i}^-,\bm{\omega}}(\mathrm{d}\bm{i}^+)\mu^-(\mathrm{d}\bm{i}^-)>0$ such that for $(\bm{i}^-,\bm{i}^+)\in U_{\bm{\omega}}$,
\begin{equation}\label{x3}
0< \bar{s}_{0}(\bm{i}^-,\bm{i}^+,\bm{\omega})-\bar{s}_{1}(\bm{i}^-,\bm{i}^+,\bm{\omega}).
\end{equation}
From now on we fix an instance of $\bm{\omega}$ such that all \eqref{x1}, \eqref{x2} and \eqref{x3} hold. The rest of the arguments will be deterministic and we shall omit the notation $\bm{\omega}$.

One can find constants $s_0$, $s_1$ such that $0<s_0-s_1$, and for any $\epsilon>0$ one can find a subset $U_\epsilon\subset U$ with $\int \mathbf{1}_{U_{\epsilon}}(\bm{i}^-,\bm{i}^+) \widetilde{\mu}^+_{\bm{i}^-}(\mathrm{d}\bm{i}^+)\mu^-(\mathrm{d}\bm{i}^-)>0$ such that for $(\bm{i}^-,\bm{i}^+)\in U_\epsilon$,
\[
\bar{s}_0(\bm{i}^-,\bm{i}^+)>s_0-\epsilon/2 \text{ and } \bar{s}_1(\bm{i}^-,\bm{i}^+) <s_1+\epsilon/2.
\]
Fix $\epsilon\in (0,\chi/3)$. There exists $n_0:U_\epsilon\to\mathbb{N}$ such that for $\mu^-$-a.e. $\bm{i}^-$, for $\widetilde{\mu}^+_{\bm{i}^-}$-a.e. $\bm{i}^+$ with $(\bm{i}^-,\bm{i}^+)\in U_\epsilon$ and $n>n_0(\bm{i}^-,\bm{i}^+)$,
\[\text{(C1): }
\begin{cases}
\frac{1}{n(-\chi+2\epsilon)}\log \widetilde{\mu}_{\bm{i}^-}^+(B_{\Phi_{{\bm{I}}}}(\bm{i}^+,e^{n(-\chi+2\epsilon)}))<s_1+\epsilon;\\
-\frac{1}{n}\log \widetilde{\mu}^+_{\bm{i}^-}([\bm{i}^+|_n])>h-\epsilon;\\
[\bm{i}^+|_n]\subset B_{\Phi_{{\bm{I}}}}(\bm{i}^+,e^{n(-\chi+2\epsilon)});\\
-\frac{1}{n}\widetilde{\mu}^+([\bm{i}^+|_n])<h+\epsilon.
\end{cases}
\]
Take $N_0$ such that the set $\Delta=\{(\bm{i}^-,\bm{i}^+)\in U_\epsilon: n_0(\bm{i}^-,\bm{i}^+)\le N_0\}$ satisfies
\[
\int \mathbf{1}_{\Delta}(\bm{i}^-,\bm{i}^+) \widetilde{\mu}^+_{\bm{i}^-}(\mathrm{d}\bm{i}^+)\mu^-(\mathrm{d}\bm{i}^-)>0.
\]
By Proposition \ref{P3.5}(1) there exists $c>0$ and $\Delta'\subset\Delta$ with $\int \mathbf{1}_{\Delta'}(\bm{i}^-,\bm{i}^+) \widetilde{\mu}^+_{\bm{i}^-}(\mathrm{d}\bm{i}^+)\mu^-(\mathrm{d}\bm{i}^-)>0$ such that for $(\bm{i}^-,\bm{i}^+)\in \Delta'$, there exists $n= n(\bm{i}^-,\bm{i}^+)\ge N_0$ such that
\[\text{(C2): }
\begin{cases}
\frac{\widetilde{\mu}^+_{\bm{i}^-}(B_{\Phi_{{\bm{I}}}}(\bm{i}^+,e^{n(-\chi+2\epsilon)})\cap \Delta_{\bm{i}^-}^+)}{\widetilde{\mu}^+_{\bm{i}^-}(B_{\Phi_{{\bm{I}}}}(\bm{i}^+,e^{n(-\chi+2\epsilon)}))}>c;\\
\frac{1}{n(-\chi+2\epsilon)}\log \widetilde{\mu}^+_{\bm{i}^-}(B_{\Phi_{{\bm{I}}}}(\bm{i}^+,2e^{n(-\chi+2\epsilon)}))>s_0-\epsilon;\\
\frac{\log(1/c)}{n}<\epsilon,
\end{cases}
\]
where we denote by $\Delta_{\bm{i}^-}^+=\pi^+(\Delta\cap \eta(\bm{i}))$.

Take $(\bm{i}^-,\bm{i}^+)\in \Delta'$ such that (C1) and (C2) are satisfied with $n=n(\bm{i}^-,\bm{i}^+)$. First we have
\[
\widetilde{\mu}^+_{\bm{i}^-}(B_{\Phi_{{\bm{I}}}}(\bm{i}^+,e^{n(-\chi+2\epsilon)})\cap \Delta^+_{\bm{i}^-}) \ge c\widetilde{\mu}^+_{\bm{i}^-}(B_{\Phi_{{\bm{I}}}}(\bm{i}^+,e^{n(-\chi+2\epsilon)})) \ge ce^{n(-\chi+2\epsilon)(s_1+\epsilon)}.
\]
But for each $(\bm{j}^-,\bm{j}^+)\in \Sigma_a^-\times B_{\Phi_{{\bm{I}}}}(\bm{i}^+,e^{n(-\chi+2\epsilon)})\cap \Delta$, we have
\[
\widetilde{\mu}^+_{\bm{j}^-}([\bm{j}^+|_n])<e^{-n(h-\epsilon)}.
\]
Therefore, for $\bm{j}^-=\bm{i}^-$, the number of distinct $[\bm{j}^+|_n]$ intersecting $B_{\Phi_{{\bm{I}}}}(\bm{i}^+,e^{n(-\chi+2\epsilon)})\cap \Delta_{\bm{i}^-}^+$ is larger than
\[
\widetilde{\mu}^+_{\bm{i}^-}(B_{\Phi_{{\bm{I}}}}(\bm{i}^+,e^{n(-\chi+2\epsilon)})\cap \Delta^+_{\bm{i}^-})\cdot e^{-n(h-\epsilon)}\ge ce^{n(-\chi+2\epsilon)(s_1+\epsilon)}e^{n(h-\epsilon)}.
\]
For each $(\bm{j}^-,\bm{j}^+)\in \Delta$ such that $[\bm{j}^+|_n]$ intersecting $B_{\Phi_{{\bm{I}}}}(\bm{i}^+,e^{n(-\chi+2\epsilon)})\cap \Delta^+_{\bm{i}^-}$ we have $|\Phi_{{\bm{I}}}(\bm{i}^+)-\Phi_{{\bm{I}}}(\bm{j}^+)|\le e^{n(-\chi+2\epsilon)})$ and $[\bm{j}^+|_n]\subset B_{\Phi_{{\bm{I}}}}(\bm{j}^+,e^{n(-\chi+2\epsilon)})$. This implies that
\[
[\bm{j}^+|_n]\subset B_{\Phi_{{\bm{I}}}}(\bm{i}^+,2e^{n(-\chi+2\epsilon)}).
\]
In the meantime, for such a $(\bm{j}^-,\bm{j}^+)$, by (C1),
\[
\widetilde{\mu}^+([\bm{j}^+|_n])\ge e^{-n(h+\epsilon)}.
\]
Hence
\begin{align*}
e^{n(-\chi+2\epsilon)(s_0-\epsilon)}\ge &\ \widetilde{\mu}^+(B_{\Phi_{{\bm{I}}}}(\bm{i}^+,2e^{n(-\chi+2\epsilon)}))\\
\ge &\ \#\{u\in\{1,\cdots,a\}^n: [u]\cap B_{\Phi_{{\bm{I}}}}(\bm{i}^+,e^{n(-\chi+2\epsilon)})\cap \Delta^+_{\bm{i}^-}\neq\emptyset\}\cdot e^{-n(h+\epsilon)}\\
\ge &\ ce^{n(-\chi+2\epsilon)(s_1+\epsilon)-n(h-\epsilon)}e^{-n(h+\epsilon)}\\
\ge &\ e^{-n\epsilon}e^{n(-\chi+2\epsilon)(s_1+\epsilon)}e^{n(h-\epsilon)}e^{-n(h+\epsilon)}.
\end{align*}
This yields that
\[
(-\chi+2\epsilon)(s_0-\epsilon)\ge -\epsilon+(-\chi+2\epsilon)(s_1+\epsilon)+(h-\epsilon)-(h+\epsilon).
\]
Letting $\epsilon\to0$ we obtain
\[
-\chi s_0\ge -\chi s_1,
\]
or equivalently $s_0\le s_1$ since $\chi>0$, which is contradict to $s_0-s_1>0$.
\end{proof}

\section{Skew product symbolic space and Mandelbrot cascades}\label{pir}

\subsection{Skew product space}\label{idsed}

Let $a, b\ge 2$ be two integers. Denote by $\Sigma_{a,b}=\Sigma_a\times\Sigma_b$. For $\bm{v}=(\bm{i},\bm{j})\in \Sigma_a\times\Sigma_b$ and $n\ge 0$ denote by
\[
\bm{v}^-=(\bm{i}^-,\bm{j}^-),\ \bm{v}^+=(\bm{i}^+,\bm{j}^+), \ \bm{v}^-|_n=(\bm{i}^-|_n,\bm{j}^-|_n) \text{ and } \bm{v}^+|_n=(\bm{i}^+|_n,\bm{j}^+|_n). 
\]
Sometimes we shall also write $\bm{v}=(\bm{v}^-,\bm{v}^+)$. Write $\Sigma_{a,b}^{\pm,*}=\{\bm{v}^\pm|_n: \bm{v}\in \Sigma_{a,b},\ n\ge 0\}$ the set of pairs of finite words in $\Sigma_{a,b}^\pm=\{\bm{v}^\pm:\bm{v}\in \Sigma_{a,b}\}$. For $u=(u^a,u^b)\in \Sigma^{-,*}_{a,b}$ denote by
\[
[u]^-=[u^a]^-\times[u^b]^-,
\]
and for $v=(v^a,v^b)\in \Sigma_{a,b}^{+,*}$ denote by
\[
[v]=[v^a]\times[v^b].
\]
For $u=(u^a,u^b), v=(v^a,v^b)\in \Sigma^{\pm,*}_{a,b}$ denote by $uv=(u^av^a,u^bv^b)$ its concatenation. For $u=(u^a,u^b)\in \Sigma^{\pm,*}_{a,b}$, for $\bm{v}^+=(\bm{i}^+,\bm{j}^+)\in \Sigma_{a,b}^+$ write $u\bm{v}^+=(u^a\bm{i}^+,u^b\bm{j}^+)$; and for $\bm{v}^-=(\bm{i}^-,\bm{j}^-)\in \Sigma_{a,b}^-$ write $\bm{v}^-u=(\bm{i}^-u^a,\bm{j}^-u^b)$.

Let $\delta,\rho\in (0,1)$.  Denote by
\[
\alpha=\log \rho/\log \delta.
\]
We assume that $\alpha$ is irrational. Without loss of generality we also assume that $\delta<\rho$ so that $\alpha\in(0,1)$. Let
\[
R(t)=t+\alpha \mod 1
\]
be the $\alpha$-rotation on $\mathbb{T}=[0,1] \mod 1$.

Let $X=\mathbb{T}\times \Sigma_{a,b}=\mathbb{T}\times \Sigma_{a,b}^-\times \Sigma_{a,b}^+$. We equip $\Sigma_{a,b}$ with the metric $d_{\delta,\rho}=\max\{d_\delta,d_\rho\}$ and then equip $X$ with the metric $d_X=\max\{d_\mathbb{T},d_{\delta,\rho}\}$, where $d_\mathbb{T}$ denotes the distance metric on $\mathbb{T}$. Let $\mathcal{B}_X=\mathcal{B}_\mathbb{T}\otimes\mathcal{B}_{a,b}$ denotes the corresponding Borel $\sigma$-algebras.

Let $A=[1-\alpha,1)$ and $A^c=\mathbb{T}\setminus A=[0,1-\alpha)$. For $t\in \mathbb{T}$ and $\bm{v}=(\bm{i},\bm{j})\in \Sigma_{a,b}$ define
\[
\sigma_t(\bm{v})=\left\{\begin{array}{ll}(\sigma(\bm{i}),\sigma(\bm{j})),& \text{if } t\in A;\\
(\bm{i},\sigma(\bm{j})),& \text{if } t\in A^c.
\end{array}\right.
\]
Then on $X$ define the skew product
\[
T(x)=(R(t),\sigma_t(\bm{v})) \text{ for } x=(t,\bm{v})\in X.
\]
Note that $T$ is invertible on $X$.

\subsection{Dimension of fibre measures} \label{ldfm} 

Let $\mu^+$ and $\nu^+$ be $\sigma$-ergodic measures on $\Sigma_a^+$ and $\Sigma_b^+$ respectively. Let $\mu$ and $\nu$ be their natural extensions to $\Sigma_a$ and $\Sigma_b$ respectively. Let $\ell$ denote the Lebesgue measure on $\mathbb{T}$. Define the product measure
\[
\Theta=\ell\times\mu\times\nu
\]
It is easy to verify that $\Theta$ is $T$-invariant.

Let $\phi$ be the canonical projection from $X$ to its subspace $\mathbb{T}\times \Sigma_{a,b}^-$. Let
\[
\eta=\{\phi^{-1}(t,\bm{v}^-)\}_{(t,\bm{v}^-)\in \mathbb{T}\times \Sigma_{a,b}^-}
\]
be its associated measurable partition. It is obvious to see that for $x=(t,\bm{v}^-,\bm{v}^+)\in X$ the unique element in $\eta$ containing $x$ is
\[
\eta(x)=\{t\}\times\{\bm{v}^-\}\times \Sigma_{a,b}^+.
\]
It is also easy to see that $d_X$ restricted on $\eta(x)$ is the isometry to the metric $d_{\delta,\rho}^+=\max\{d_\delta^+,d_\rho^+\}$ on $\Sigma_{a,b}^+$. Let $\Theta^\eta_x$ be the conditional measures of $\Theta$ with respect to the measurable partition $\eta$. We have the following lemma regarding the dimension of $\Theta^\eta_x$.

\begin{lemma}\label{tex}
For $\Theta$-a.e. $x\in X$, the conditional measure $\Theta^\eta_x$ is exact-dimensional on $(\eta(x),d_X)$ with dimension
\[
\frac{h_{\mu}}{-\log \delta}+\frac{h_{\nu}}{-\log \rho}.
\]
\end{lemma}

\begin{proof}
Define the following finite partition of $X$:
\begin{equation}\label{P}
\mathcal{P}=\{A^c\times \Sigma_{a,b}^-\times(\Sigma_a^+\times [j]), A\times \Sigma_{a,b}^-\times ([i]\times [j])\}_{i\in\{1,\ldots,a\},j\in\{1,\ldots,b\}}.
\end{equation}
For $t\in \mathbb{T}$ and $n\ge 1$ denote by
\[
\tau_n(t)=\sum_{k=0}^{n-1}\mathbf{1}_{\{R^k(t)\in A\}}.
\]
Fix an $x=(t,\bm{v})=(t,\bm{i},\bm{j})\in X$ and $n\ge 1$ for now. We may write
\begin{equation}\label{Tnx}
T^n(x)= (R^n(t),\sigma_{n,t}(\bm{v})) \text{ with } \sigma_{n,t}(\bm{v})=(\sigma^{\tau_n(t)}\bm{i},\sigma^n\bm{j}),
\end{equation}
where $\sigma^{\tau_n(t)}$ denotes the identity map if $\tau_n(t)=0$. Recall the notation
\[
\mathcal{P}^{n}=\bigvee_{k=0}^{n-1}T^{-k}\mathcal{P}
\]
and that
\[
\mathcal{P}^{n}(x)=\mathcal{P}(x)\vee T^{-1}\mathcal{P}(T(x))\vee \cdots \vee T^{-n+1}\mathcal{P}(T^{n-1}(x)).
\]
Denote by $\bm{v}^+|_{n,t}=(\bm{i}^+|_{\tau_n(t)},\bm{j}^+|_n)\in \Sigma_{a,b}^{+,*}$. Then it is easy to check that
\begin{equation}\label{P0n-1}
\mathcal{P}^{n}(x)=\mathcal{Q}^{n}(t)\times \Sigma_{a,b}^-\times [\bm{v}^+|_{n,t}],
\end{equation}
where $\mathcal{Q}=\{A,A^c\}$ is the finite partition of $\mathbb{T}$ by $A$ and
\[
\mathcal{Q}^{n}=\bigvee_{k=0}^{n-1}R^{-k}\mathcal{Q}.
\]
By \eqref{Tnx} we also have
\[
\eta(T^n(x))=\{R^n(t)\}\times\{\bm{v}^-\bm{v}^+|_{n,t}\}\times \Sigma_{a,b}^+,\]
where $\bm{v}^-\bm{v}^+|_{n,t}=(\bm{i}^-i_1\cdots i_{\tau_n(t)},\bm{j}^-j_1\cdots j_n)$. Therefore
\[
T^{-n}(\eta(T^n(x)))=\{t\}\times\{\bm{v}^-\}\times [\bm{v}^+|_{n,t}]=\mathcal{P}^{n}(x)\cap \eta(x),
\]
meaning that $\eta$ and $\mathcal{P}$ satisfy \eqref{TetaP}.

By the uniqueness of conditional measures (Theorem \ref{Rohthm}), we have for $\Theta$-a.e. $x=(t,\bm{v})=(t,\bm{i},\bm{j})\in X$,
\begin{align*}
\Theta^\eta_x(\mathcal{P}^{n}(x))&=\ \Theta^\eta_x(\mathcal{Q}^{n}(t)\times \Sigma_{a,b}^-\times [\bm{v}^+|_{n,t}])\\
&=\ \Theta^\eta_x(\{t\}\times\{\bm{v}^-\}\times [\bm{v}^+|_{n,t}])\\
&=\ \mu_{\bm{i}^-}^+([\bm{i}^+|_{\tau_n(t)}])\times\nu_{\bm{j}^-}^+([\bm{j}^+|_n]).
\end{align*}
Note that for each $t\in\mathbb{T}$ and $n\ge 1$ we have another interpretation of $\tau_n(t)$ by
\[
\tau_n(t)=\sum_{k=0}^{n-1}\mathbf{1}_{\{R^k(t)\in A\}}=\mathrm{Int}(t+n\alpha),
\]
where $\mathrm{Int}(t+n\alpha)$ stands for the integer part of $t+n\alpha$. This implies
\begin{equation}\label{taunt}
|\tau_n(t)-n\alpha|\le t+1\le 2.
\end{equation}
Hence for each $t\in \mathbb{T}$,
\[
\lim_{n\to\infty}\frac{\tau_n(t)}{n}=\alpha.
\]
Since $\mu^+$ and $\nu^+$ are ergodic measures, by Proposition \ref{SMB}  and Lemma \ref{lem1} we have for $\Theta$-a.e. $x=(t,\bm{i},\bm{j})$,
\begin{align*}
h_{\Theta}(\mathcal{P}|\widehat{\eta})(x)&=\ \lim_{n\to\infty} \frac{-\log \Theta^\eta_x(\mathcal{P}_0^n(x))}{n}\\
&=\ \lim_{n\to\infty} \frac{-\log \mu^+_{\bm{i}^-}([\bm{i}^+|_{\tau_n(t)}])-\log \nu^+_{\bm{j}^-}([\bm{j}^+|_n])}{n}\\
&=\ \lim_{n\to\infty} \frac{\tau_n(t)}{n}\frac{-\log \mu^+_{\bm{i}^-}([\bm{i}^+|_{\tau_n(t)}])}{\tau_n(t)}+\frac{-\log \nu^+_{\bm{j}^-}([\bm{j}^+|_n])}{n}\\
&=\ \alpha h_{\mu}+h_{\nu}.\numberthis \label{ahmn}
\end{align*}
Furthermore for $x=(t,\bm{i},\bm{j})\in X$ and $n\ge 1$ we have
\begin{align*}
|\mathcal{P}^{n}(x)\cap \eta(x)|_{d_X}&=\ \max\{|[\bm{i}^+|_{\tau_n(t)}]|_{d_\delta^+},|[\bm{j}^+|_n]|_{d_\rho^+}\}\\
&=\ \max\{\delta^{\tau_n(t)},\rho^n\}.
\end{align*}
By \eqref{taunt} we have $\delta^{\tau_n(t)}\le \delta^{-2}\delta^{n\alpha}=\delta^{-2}\rho^n$. Hence
\begin{equation}\label{radiusc}
\rho^n\le |\mathcal{P}^{n}(x)\cap \eta(x)|_{d_X}\le \delta^{-2}\rho^n.
\end{equation}
Together with \eqref{ahmn} this implies that for $\Theta$-a.e. $x\in X$
\[
\lim_{n\to\infty} \frac{\log \Theta^\eta_x(\mathcal{P}^n(x))}{\log |\mathcal{P}^n(x)\cap \eta(x)|_{d_X}}=\frac{\alpha h_{\mu}+h_{\nu}}{-\log\rho}=\frac{h_{\mu}}{-\log \delta}+\frac{h_{\nu}}{-\log \rho}.
\]
Note that, by Theorem \ref{Rohthm}, ``for $\Theta$-a.e. $x\in X$" is equivalent to ``for $\Theta$-a.e $y$, for $\Theta^\eta_{y}$-a.e $x\in \eta(y)$". Furthermore, since for $x\in \eta(y)$, $\eta(x)=\eta(y)$, for $\Theta^\eta_{y}$-a.e $x\in \eta(y)$ we have $\Theta^\eta_x=\Theta^\eta_{y}$. Therefore we may deduce that for $\Theta$-a.e $y\in X$, for $\Theta^\eta_{y}$-a.e $x\in \eta(y)$,
\[
\lim_{n\to\infty} \frac{\log \Theta^\eta_{y}(\mathcal{P}^n(x))}{\log |\mathcal{P}^n(x)\cap \eta(x)|_{d_X}}=\frac{h_{\mu}}{-\log \delta}+\frac{h_{\nu}}{-\log \rho},
\]
which yields the conclusion.
\end{proof}

\subsection{Dimension of fibre ergodic decomposition measures}\label{dfedm}

The product measure $\Theta$ is not necessarily $T$-ergodic, even if all $\ell$, $\mu$ and $\nu$ are ergodic. Nevertheless we may use ergodic decomposition to obtain a family of $T$-ergodic measures. More precisely, recall $\mathcal{I}=\{B\in \mathcal{B}_X:T^{-1}B=B\}$ the $T$-invariant $\sigma$-algebra, one can find a $\Theta$-measurable partition $\zeta\subset \mathcal{I}$ of $X$ consisting of $T$-invariant subsets such that $\widehat{\zeta}=\mathcal{I}$ modulo set of zero $\Theta$-measure. Furthermore, denote by $\Theta^\zeta_x$ the conditional measures of $\Theta$ w.r.t. $\zeta$, then we have the decomposition
\begin{equation}\label{med1}
\Theta=\int_X \Theta^\zeta_x \,\Theta(\mathrm{d} x),
\end{equation}
and each $\Theta$-typical conditional measure $\Theta^\zeta_x$ is a $T$-ergodic measure on $(\zeta(x),\mathcal{B}|_{\zeta(x)})$ hence on $(X,\mathcal{B})$. See \cite[Chapter 5]{VO} for example.

We may select an ergodic measure $\Theta^\zeta_{z}$ at a $\Theta$-typical $z\in X$, then we may select a fibre measure $(\Theta^\zeta_{z})^\eta_{y}$ on $\zeta(z)\cap \eta(y)$ at a $\Theta^\zeta_{z}$-typical $y\in \zeta(x)$. The following theorem ensures that after ergodic decomposing and conditioning, the fibre measures $(\Theta^\zeta_{z})^\eta_{y}$ still have the right dimension.

\begin{theorem}\label{ed3}
For $\Theta$-a.e. $z\in X$, for $\Theta^\zeta_{z}$-a.e. $y\in \zeta(x)$, the conditional measure $(\Theta^\zeta_{z})^\eta_{y}$ is exact-dimensional on $(\zeta(z)\cap \eta(y), d_X)$ with dimension
\[
\frac{h_{\mu}}{-\log \delta}+\frac{h_{\nu}}{-\log \rho}.
\]
\end{theorem}

\begin{proof}
Since $\zeta\subset \mathcal{I}$, we have $\zeta(x)=\zeta(T(x))$ and $T^{-1}\zeta=\zeta$. Since $\eta$ and $\mathcal{P}$ satisfy \eqref{TetaP}, we may deduce that
\begin{equation}\label{TetaP2}
T^{-n}(\eta\vee\zeta(T^n(x)))=\eta\vee\zeta(x)\cap \mathcal{P}^{n}(x),
\end{equation}
and
\begin{equation}\label{P0n2}
\mathcal{P}^{n}(T(x))\cap \eta(T(x))\cap \zeta(T(x))=\, T(\mathcal{P}^{n+1}(x))\cap \eta(T(x))\cap \zeta(T(x)),
\end{equation}
meaning that $\eta\vee \zeta$ and $\mathcal{P}$ also satisfy \eqref{TetaP}. By Proposition \ref{SMB} we have for $\Theta$-a.e. $x\in X$,
\begin{equation}\label{ldez}
\lim_{n\to\infty}\frac{-\log\Theta^{\eta\vee\zeta}_x(\mathcal{P}^{n}(x))}{n}=h_\Theta(\mathcal{P}|\widehat{\eta\vee\zeta})(x).
\end{equation}
By the uniqueness of conditional measures (Theorem \ref{Rohthm}), we have for $\Theta$-a.e. $x\in X$,
\[
\Theta^{\eta\vee\zeta}_x=(\Theta^{\eta}_x)^\zeta_x=(\Theta^{\zeta}_x)^\eta_x.
\]
We may reinterpret \eqref{ldez} as for $\Theta$-a.e. $y\in X$, for $\Theta^\eta_{y}$-a.e. $z\in \eta(y)$, for $(\Theta^\eta_{y})^\zeta_{z}$-a.e. $x\in  \eta(y)\cap\zeta(z)$,
\[
\lim_{n\to\infty}\frac{-\log(\Theta^{\eta}_{x})^{\zeta}_{x}(\mathcal{P}^{n}(x))}{n}=h_\Theta(\mathcal{P}|\widehat{\eta\vee\zeta})(x).
\]
Note that for $x\in \eta(y)\cap \zeta(z)$, $\eta(x)=\eta(y)$ and $\zeta(x)=\zeta(z)$, therefore for $\Theta^\eta_{y}$-a.e. $z\in \eta(y)$, for $(\Theta^\eta_{y})^\zeta_{z}$-a.e. $x\in \eta(y)\cap \zeta(z)$,
\[
(\Theta^{\eta}_x)^{\zeta}_x=(\Theta^{\eta}_{y})^{\zeta}_{z}.
\]
Thus we obtain that for $\Theta$-a.e. $y\in X$, for $\Theta^\eta_{y}$-a.e. $z\in \eta(y)$, for $(\Theta^\eta_{y})^\zeta_{z}$-a.e. $x\in  \eta(y)\cap\zeta(z)$,
\[
\lim_{n\to\infty}\frac{-\log(\Theta^{\eta}_{y})^{\zeta}_{z}(\mathcal{P}^{n}(x))}{n}=h_\Theta(\mathcal{P}|\widehat{\eta\vee\zeta})(x).
\]
Using \eqref{radiusc} we deduce that for $\Theta$-a.e. $y\in X$, for $\Theta^\eta_{y}$-a.e. $z\in \eta(y)$, for $(\Theta^\eta_{y})^\zeta_{z}$-a.e. $x\in  \eta(y)\cap\zeta(z)$,
\begin{equation}\label{ld2}
\lim_{n\to\infty}\frac{\log (\Theta^{\eta}_{y})^{\zeta}_{z}(\mathcal{P}^{n}(x)) }{\log |\mathcal{P}^{n}(x)\cap \eta(x)|_{d_X}}=\frac{h_\Theta(\mathcal{P}|\widehat{\eta\vee\zeta})(x)}{-\log \rho}.
\end{equation}
By the fact that $(\Theta^\eta_{y})^\zeta_{z}$ is a conditional measure of  $\Theta^\eta_{y}$, hence, by the definition of lower Hausdorff dimension, their essential local dimension can not be greater than the lower Hausdorff dimension of $\Theta^\eta_{y}$, which is $\frac{\alpha h_\mu+h_\nu}{-\log \rho}$ by Lemma \ref{tex}, we deduce that for $\Theta$-a.e. $x\in X$,
\begin{equation}\label{hmuT}
\alpha h_\mu+h_\nu\ge h_\mu(\mathcal{P}|\widehat{\eta\vee\zeta})(x).
\end{equation}

Next we will show that the reverse inequality also holds. Using the fact that for $\Theta^\zeta_x$-a.e. $y\in\zeta(x)$, $\Theta^\zeta_{y}=\Theta^\zeta_{x}$, we deduce that for $\Theta$-a.e. $x\in X$,
\begin{align*}
h_\Theta(\mathcal{P}|\widehat{\eta\vee\zeta})(x)&=\ \mathbb{E}_\Theta( -\log \Theta^{\eta\vee\zeta}_\cdot(\mathcal{P}(\cdot)) |\mathcal{I})(x)\\
&=\ \int_{\zeta(x)} -\log (\Theta^{\zeta}_{y})^{\eta}_{y}(\mathcal{P}(y)) \, \Theta^\zeta_x(\mathrm{d}y)\\
&=\ \int_{\zeta(x)} -\log (\Theta^{\zeta}_x)^{\eta}_{y}(\mathcal{P}(y)) \, \Theta^\zeta_x(\mathrm{d}y)\\
&=\ H_{\Theta^\zeta_x}(\mathcal{P}|\widehat{\eta}).\numberthis\label{hH}
\end{align*}
By \cite[Lemma 4.4 (i), (iii)]{FH09} (see also \cite{Parry81}), and that $T^{-1}\eta=\eta\vee \mathcal{P}$, for $n\ge 2$ we have
\begin{align*}
H_{\Theta_x^\zeta}(\mathcal{P}^{n}|\widehat{\eta})&=\ H_{\Theta_x^\zeta}(\mathcal{P}\vee T^{-1}(\mathcal{P}^{n-1})|\widehat{\eta})\\
&=\ H_{\Theta_x^\zeta}(\mathcal{P}|\widehat{\eta})+H_{\Theta_x^\zeta}(T^{-1}(\mathcal{P}^{n-1})|\widehat{\eta\vee \mathcal{P}} )\\
&=\ H_{\Theta_x^\zeta}(\mathcal{P}|\widehat{\eta})+H_{\Theta_x^\zeta}(T^{-1}(\mathcal{P}^{n-1})|\widehat{T^{-1}\eta} )\\
&=\ H_{\Theta_x^\zeta}(\mathcal{P}|\widehat{\eta})+H_{\Theta_x^\zeta}(T^{-1}(\mathcal{P}^{n-1})|T^{-1}\widehat{\eta} )\\
&=\ H_{\Theta_x^\zeta}(\mathcal{P}|\widehat{\eta})+H_{\Theta_x^\zeta}(\mathcal{P}^{n-1}|\widehat{\eta} ),
\end{align*}
where in the last step we have used the fact that $\Theta_x^\zeta$ is $T$-invariant. By iteration this implies that for $\Theta$-a.e. $x\in X$ and $n\ge 1$,
\[
H_{\Theta_x^\zeta}(\mathcal{P}^{n}|\widehat{\eta})=nH_{\Theta_x^\zeta}(\mathcal{P}|\widehat{\eta}).
\]
Since $-\log$ is a convex function, by Jensen's inequality and Theorem \ref{Rohthm} we have
\begin{align*}
H_{\Theta_x^\zeta}(\mathcal{P}^{n}|\widehat{\eta})&=\ \int_X -\log (\Theta_x^\zeta)^\eta_{y}(\mathcal{P}^{n}(y))  \,\Theta_x^\zeta(\mathrm{d}y)\\
&\ge \ -\log \int_X (\Theta_x^\zeta)^\eta_{y}(\mathcal{P}^{n}(y)) \,\Theta_x^\zeta(\mathrm{d}y)\\
&= \ -\log \int_X \sum_{B\in\mathcal{P}^{n}} \mathbf{1}_B(y)(\Theta_x^\zeta)^\eta_{y}(B) \,\Theta_x^\zeta(\mathrm{d}y)\\
&= \ -\log \sum_{B\in\mathcal{P}^{n}} \int_X (\Theta_x^\zeta)^\eta_{y}(B) \,\Theta_x^\zeta(\mathrm{d}y)\\
&= \ -\log \sum_{B\in\mathcal{P}^{n}} \Theta_x^\zeta(B)\\
&= \  -\log\Theta_x^\zeta(\mathcal{P}^{n}(x)),
\end{align*}
where we have used the fact that $(\Theta_x^\zeta)^\eta_{y}(B)=0$ if $y\not \in B$ and $\Theta_x^\zeta(B)=0$ if $x\not\in B$. Therefore for $\Theta$-a.e. $x\in X$ and $n\ge 1$,
\[
H_{\Theta_x^\zeta}(\mathcal{P}|\widehat{\eta})\ge \frac{ -\log\Theta_x^\zeta(\mathcal{P}^{n}(x))}{n}.
\]
In other words, for $\Theta$-a.e. $x\in X$, for $\Theta^\zeta_x$-a.e. $y\in \zeta(x)$ and $n\ge 1$,
\[
H_{\Theta_{x}^\zeta}(\mathcal{P}|\widehat{\eta})=H_{\Theta_{y}^\zeta}(\mathcal{P}|\widehat{\eta})\ge \frac{ -\log\Theta_{y}^\zeta(\mathcal{P}^{n}(y))}{n}=\frac{ -\log\Theta_{x}^\zeta(\mathcal{P}^{n}(y))}{n},
\]
where we have used again the fact that for $\Theta^\zeta_x$-a.e. $y\in \zeta(x)$, $\Theta^\zeta_{y}=\Theta^\zeta_x$. By applying the SMB theorem to the ergodic measure $\Theta^\zeta_x$ we deduce that for $\Theta$-a.e. $x\in X$,
\begin{equation}\label{HTheta}
H_{\Theta_x^\zeta}(\mathcal{P}|\widehat{\eta})\ge h_{\Theta_x^\zeta}(T,\mathcal{P}).
\end{equation}

On the other hand, by Jacob theorem (see \cite[Theorem 9.6.2]{VO} for example) we have that
\begin{equation}\label{Jacob}
\int_X h_{\Theta_x^\zeta}(T,\mathcal{P}) \, \Theta(\mathrm{d}x)=h_\Theta(T,\mathcal{P}).
\end{equation}
Recall that $\Theta=\ell\times\mu\times\nu$ where $\mu$ (resp. $\nu$) is the natural extension of $\mu^+$ (resp. $\nu^+$) to $\Sigma_a$ (resp. $\Sigma_b$), by \eqref{cyl} for $x=(t,\bm{v})=(t,\bm{i},\bm{j})$ we have
\begin{align*}
\Theta(\mathcal{P}^{n}(x))&=\ \Theta(\mathcal{Q}^{n}(t)\times \Sigma_{a,b}^-\times [\bm{v}^+|_{n,t}])\\
&=\ \ell(\mathcal{Q}^{n}(t))\cdot \mu(\Sigma_a^-\times[\bm{i}^+|_{\tau_n(t)}])\cdot \nu(\Sigma_b^-\times[\bm{j}^+|_n])\\
&=\ \ell(\mathcal{Q}^{n}(t))\cdot \mu^+([\bm{i}^+|_{\tau_n(t)}])\cdot \nu^+([\bm{j}^+|_n]).
\end{align*}
Therefore, by using SMB Theorem for the ergodic measures $\ell$, $\mu^+$ and $\nu^+$ respectively, we have for $\Theta$-a.e. $x=(t,\bm{i},\bm{j})\in X$,
\begin{align*}
\lim_{n\to\infty}\frac{-\log \Theta(\mathcal{P}^{n}(x))}{n}&=\, \lim_{n\to\infty}\frac{-\log \ell(\mathcal{Q}^{n}(t))\cdot \mu^+([\bm{i}^+|_{\tau_n(t)}])\cdot \nu^+([\bm{j}^+|_n]) }{n}\\
&=\, \lim_{n\to\infty}\frac{-\log \ell(\mathcal{Q}^{n}(t))}{n}+\frac{\tau_n(t)}{n}\frac{-\log \mu^+([\bm{i}^+|_{\tau_n(t)}])}{\tau_n(t)}+ \frac{-\log\nu^+([\bm{j}^+|_n])}{n} \\
&=\, h_\ell(R,\mathcal{Q})+\alpha h_{\mu}+h_{\nu}\\
&=\, \alpha h_{\mu}+h_{\nu},
\end{align*}
where we have used the well-known fact that $h_\ell(R,\mathcal{Q})=0$, since $R$ is an irrational rotation and $\mathcal{Q}$ is a partition of $\mathbb{T}$ consisting of two intervals (see \cite[p.245]{Petersen} for example). Finally by using Theorem \ref{SMBinv} for $\Theta$, we deduce that
\[
h_\Theta(T,\mathcal{P})= \alpha h_{\mu}+h_{\nu}.
\]
Together with \eqref{hmuT}, \eqref{hH}, \eqref{HTheta} and \eqref{Jacob} we deduce that for $\Theta$-a.e. $x\in X$,
\[
h_\mu(\mathcal{P}|\widehat{\eta\vee\zeta})(x)=\alpha h_{\mu}+h_{\nu}.
\]
By \eqref{ld2} this yields the conclusion.
\end{proof}

\subsection{Mandelbrot cascades on fibre ergodic decomposition measures}

From now on, to simplify the notation, we shall use $\theta=\Theta^\zeta_{z}$ to denote a $T$-ergodic measure obtained by conditioning $\Theta$ with respect to $\zeta$ at a $\Theta$-typical point $z\in X$. Note that $\theta=\Theta^\zeta_{z}$ is carried by $\zeta(z)\subset X$, therefore it is equivalent to say ``for $\theta$-a.e. $x\in \zeta(z)$" and ``for $\theta$-a.e. $x\in X$". When there is no confusion, we shall also use the notation
\[
\int_{\mathcal{M}(X)} f(\theta) \,\Theta(\mathrm{d}\theta)=\int_{X} f(\Theta^\zeta_z) \,\Theta(\mathrm{d}z)
\]
for any continuous function $f$ on the space of probability measures $\mathcal{M}(X)$ on $X$ under the weak star topology, so a $\Theta$-typical $\theta$ means the conditional measure $\Theta^\zeta_z$ as a $\Theta$-typical $z\in X$.

Recall the measurable partition
\[
\eta=\{\phi^{-1}(t,\bm{v}^-):(t,\bm{v}^-)\in \mathbb{T}\times \Sigma_{a,b}^-\}
\]
obtained from the canonical projection $\phi: X\mapsto \mathbb{T}\times \Sigma_{a,b}^-$. Let $\theta^\eta_{x}$ for $\theta$-a.e. $x\in X$ be the conditional measures of $\theta$ with respect to $\eta$. Let $\varphi$ be the canonical projection from $X$ to its subspace $\Sigma_{a,b}^+$. Note that for $x=(t,\bm{v}^-,\bm{v}^+)\in X$, the element in $\eta$ containing $x$ is
\[
\eta(x)=\{t\}\times\{\bm{v}^-\}\times\Sigma_{a,b}^+,
\]
which only depends on $t$ and $\bm{v}^-$, therefore we may write
\[
\theta^+_{t,\bm{v}^-}=\theta^\eta_{x}\circ \varphi^{-1}
\]
the projection of $\theta^\eta_x$ onto $\Sigma_{a,b}^+$ for $\theta^-:=\theta\circ \phi^{-1}$-a.e. $(t,\bm{v}^-)\in \mathbb{T}\times\Sigma_{a,b}^-$. 


For $\theta^-$-a.e. $(t,\bm{v}^-)\in \mathbb{T}\times\Sigma_{a,b}^-$ and $k\ge 1$ define the random measure
\begin{equation}\label{tta}
\widetilde{\theta}^+_{t,\bm{v}^-,\bm{\omega},k}(\mathrm{d}\bm{v}^+)=Q_{\bm{v}^+|_{k,t}}(\bm{\omega})\, \theta_{t,\bm{v}^-}^{+}(\mathrm{d}\bm{v}^+), \ \bm{v}^+\in \Sigma_{a,b}^+, \ \bm{\omega}\in\Omega_{a,b},
\end{equation}
where for $\bm{v}^+=(\bm{i}^+,\bm{j}^+)$, $t\in \mathbb{T}$ and $k\ge 1$ we denote by
\[
Q_{\bm{v}^+|_{k,t}}(\bm{\omega})=Q_{\bm{i}^+|_{\tau_{k}(t)}}^V(\bm{\omega}^a) Q_{\bm{j}^+|_{k}}^W(\bm{\omega}^b) \text{ for } \bm{\omega}=(\bm{\omega}^a,\bm{\omega}^b).
\]
For $\theta^-$-a.e. $(t,\bm{v}^-)$ the sequence $\{\widetilde{\theta}^+_{t,\bm{v}^-,\cdot,k}\}_{k\ge 1}$ forms a measure-valued martingale with respect to the natural filtration $(\sigma(Q_{\bm{v}^+|_{k,t}}:\bm{v}^+\in\Sigma_{a,b}^+))_{k\ge 1}$, hence the weak limit
\[
\widetilde{\theta}_{t,\bm{v}^-,\bm{\omega}}^{+}=\lim_{k\to\infty} \widetilde{\theta}_{t,\bm{v}^-,\bm{\omega},k}^{+}
\]
exists for $\mathbb{P}_{a,b}$-a.e. $\bm{\omega}\in\Omega_{a,b}$. We call $\widetilde{\theta}_{t,\bm{v}^-}^{+}$ the Mandelbrot cascade measure of $\theta^+_{t,\bm{v}^-}$.

Regarding the non-degeneracy of $\widetilde{\theta}_{t,\bm{v}^-}^{+}$, note that, by independence,
\begin{align*}
\mathbb{E}_{\mathbb{P}_{a,b}}(Q_{\bm{v}^+|_{k,t}} \log Q_{\bm{v}^+|_{k,t}})&=\ \mathbb{E}_{\mathbb{P}_{a,b}}(Q_{\bm{i}^+|_{\tau_{k}(t)}}^V Q_{\bm{j}^+|_{k}}^W  \log Q_{\bm{i}^+|_{\tau_{k}(t)}}^V Q_{\bm{j}^+|_{k}}^W )\\
&=\ \tau_{k}(t)\mathbb{E}(V\log V)+k\mathbb{E}(W\log W)\\
&=\ \tau_{k}(t) h_V+k h_W.
\end{align*}
By \eqref{taunt} we have
\[
k(\alpha h_V+h_W)-2h_V\le \tau_{k}(t) h_V+k h_W\le k(\alpha h_V+h_W)+2h_V.
\]
Together with \eqref{radiusc}, for $k$ large enough we have
\[
\frac{\mathbb{E}_{\mathbb{P}_{a,b}}(Q_{\bm{v}^+|_{k,t}} \log Q_{\bm{v}^+|_{k,t}})}{-\log |[\bm{v}^+|_{k,t}]|_{d_{\delta,\rho}^+}}\approx \frac{k(\alpha h_V+h_W)}{-k\log \rho}=\frac{h_V}{-\log \delta}+\frac{h_W}{-\log \rho}.
\]
By Theorem \ref{ed3} we have for $\theta^-$-a.e. $(t,\bm{v}^-)\in \mathbb{T}\times\Sigma_{a,b}^-$, $\theta_{t,\bm{v}^-}^{+}$ is exact-dimensional on $(\Sigma_{a,b}^+,d_{\delta,\rho}^+)$ with dimension
\[
\frac{h_\mu}{-\log \delta}+\frac{h_\nu}{-\log \rho},
\]
and we have assumed that $h_V<h_\mu$ and $h_W<h_\nu$. Therefore following the same line of proofs as in \cite[Theorem 2.3]{BJ21} we may deduce that for $\theta^-$-a.e. $(t,\bm{v}^-)\in \mathbb{T}\times\Sigma_{a,b}^-$, the Mandelbrot cascade measure $\widetilde{\theta}_{t,\bm{v}^-}^{+}$ is non-degenerate, and for $\mathbb{P}_{a,b}$-a.e. $\bm{\omega}\in \Omega_{a,b}$ with $\|\widetilde{\theta}_{t,\bm{v}^-,\bm{\omega}}^{+}\|>0$, $\widetilde{\theta}_{t,\bm{v}^-,\bm{\omega}}^{+}$ is exact-dimensional on $(\Sigma_{a,b}^+,d_{\delta,\rho}^+)$ with dimension
\begin{equation}\label{ed123}
\frac{h_\mu-h_V}{-\log \delta}+\frac{h_\nu-h_W}{-\log \rho}.
\end{equation}

\subsection{The Peyri\`ere measure on the skew product space}\label{esps}
Consider the product space
\[
\widetilde{X}=X\times\Omega_{a,b}=\mathbb{T}\times\Sigma_{a,b}\times\Omega_{a,b}.
\]
Define a skew product $\widetilde{T}$ on $\widetilde{X}$ by
\[
\widetilde{T}(t,\bm{v},\bm{\omega})=(R(t),\sigma_t(\bm{v}), \kappa_{t,\bm{v}}(\bm{\omega})),
\]
where for $t\in\mathbb{T}$, $\bm{v}=(\bm{i},\bm{j})\in \Sigma_{a,b}$ and $\bm{\omega}=(\bm{\omega}^a,\bm{\omega}^b)\in\Omega_{a,b}$ we denote by
\[
\kappa_{t,\bm{v}}(\bm{\omega})=(\omega^a_{\bm{i}^+|_{\tau_1(t)}},\omega^b_{\bm{j}^+|_1})
\]
with the convention that $\omega^a_{\bm{i}^+|_{\tau_1(t)}}=\omega^a$ if $\tau_1(t)=0$. For $t\in\mathbb{T}$ and $n\ge 1$ denote by
\begin{equation}\label{Cnt}
C_{n,t}=\{\bm{v}^+|_{n,t}: \bm{v}^+\in \Sigma_{a,b}^+\}.
\end{equation}
Let $\widetilde{\mathcal{A}}$ be a semi-algebra on $\widetilde{X}$ of the form 
\begin{align*}
\{(t,\bm{v},\bm{\omega}):\  \;t\in(c_1,c_2),\ \bm{v}\in[u]^-\times[v_t],&\ V_{v^a_1\cdots v^a_k}(\bm{\omega}^a)\in B^a_{v^a_1\cdots v^a_k},\; k=1,\cdots,n^a, \\
 &\ W_{v^b_1\cdots v^b_{k'}}(\bm{\omega}^b)\in B^b_{v^b_1\cdots v^b_{k'}}, k'=1,\cdots,n^b\}, 
\end{align*}
where $0\le c_1< c_2\le1$, $u=(u^a,u^b)\in \Sigma_{a,b}^{-,*}$, $v_t=(v^a,v^b)\in C_{n,t}$ with $v^l=v^l_1\cdots v^l_{n^l}\in \Sigma_l^{+,*}$ and $\{B^l_{v^l_1\cdots v^l_k}\}_{1\le k\le n^l}$ Borel subsets of $[0,\infty)$ for $l=a,b$. Let $\widetilde{\mathcal{F}}$ be the $\sigma$-algebra generated by $\widetilde{\mathcal{A}}$. Define a probability measure $\mathbb{Q}_{\theta}$ on $(\widetilde{X},\widetilde{\mathcal{F}})$ as
\begin{equation}\label{pmqtheta}
\mathbb{E}_{\mathbb{Q}_\theta}(f)= \int_{\Omega_{a,b}}\int_{\mathbb{T}\times\Sigma_{a,b}^-}\int_{\Sigma_{a,b}^+} f(t,\bm{v}^-,\bm{v}^+,\bm{\omega}) \, \widetilde{\theta}^+_{t,\bm{v}^-,\bm{\omega}}(\mathrm{d}\bm{v}^+) \theta^-(\mathrm{d}(t,\bm{v}^-))\mathbb{P}_{a,b}(\mathrm{d}\bm{\omega})
\end{equation}
for $\widetilde{\mathcal{F}}$-measurable functions $f$. We call $\mathbb{Q}_{\theta}$ the Peyri\`ere measure of $\theta$. Similar to Lemma \ref{lminv} we have
\begin{lemma}
The Peyri\`ere measure $\mathbb{Q}_{\theta}$ is $\widetilde{T}$-invariant. 
\end{lemma}
\begin{proof}

Recall the finite partition $\mathcal{P}$ of $X$ in \eqref{P}, that $\eta$ and  $\mathcal{P}$ satisfy \eqref{TetaP}, and that for $n\ge 1$ and $x=(t,\bm{v})\in X$,
\[
\mathcal{P}^{n}(x)=\mathcal{Q}^{n}(t)\times \Sigma_{a,b}^-\times [\bm{v}^+|_{n,t}].
\]
By \eqref{inid} we have that for any $B\in\mathcal{B}_X$, for $\theta$-a.e. $x\in X$ and $n\ge 1$,
\[
\theta^{\eta}_{T^n(x)}(T^n(B))=\frac{\theta^\eta_{x}(B\cap \mathcal{P}^{n}(x))}{\theta^\eta_{x}(\mathcal{P}^{n}(x))}.
\]
For $B\in\mathcal{B}_X$ of the form
\[
B=\mathcal{Q}^{n}(t)\times \Sigma_{a,b}^-\times [uv],
\]
where $t\in\mathbb{T}$, $n\ge 1$, $u\in C_{n,t}$ and $v\in \Sigma_{a,b}^{+,*}$, it is easy to see that
\[
T^n(B)=R^n(\mathcal{Q}^{n}(t))\times [u]^-\times[v]
\]
and that $B\cap \mathcal{P}^{n}(x)=B$ for $x\in B$, therefore, for $\theta$-a.e. $x=(t,\bm{v}^-,\bm{v}^+)\in B$,
\[
\frac{\theta^+_{t,\bm{v}^-}([uv])}{\theta^+_{t,\bm{v}^-}([u])}=\frac{\theta^\eta_{x}(B\cap \mathcal{P}^{n}(x))}{\theta^\eta_{x}(\mathcal{P}^{n}(x))}=\theta^\eta_{T^n(x)}(T^n(B))=\theta^{+}_{R^n(t),\bm{v}^-u}([v]).
\]
In other words, for $\theta^-$-a.e. $(t,\bm{v}^-)$, for $n\ge 1$, $u\in C_{n,t}$ and $v\in\Sigma_{a,b}^{+,*}$,
\begin{equation}\label{cdm34}
\frac{\theta^+_{t,\bm{v}^-}([uv])}{\theta^+_{t,\bm{v}^-}([u])}=\theta^{+}_{R^n(t),\bm{v}^-u}([v]).
\end{equation}
Together with \eqref{tta} we obtain the following factorisation:  for $\theta^-$-a.e. $(t,\bm{v}^-)$, for $\mathbb{P}_{a,b}$-a.e. $\bm{\omega}$, for $n\ge 1$, $u\in C_{n,t}$ and $v\in\Sigma_{a,b}^{+,*}$,
\begin{equation}\label{ss11}
\widetilde{\theta}_{t,\bm{v}^-,\bm{\omega}}^+([uv])=Q_{u}(\bm{\omega})\cdot\theta_{t,\bm{v}^-}^+([u])\cdot \widetilde{\theta}_{R^n(t),\bm{v}^-u,\bm{\omega}_{u}}^{+}([v]),
\end{equation}
where for $\bm{\omega}=(\bm{\omega}^a,\bm{\omega}^b)$ and $u=(u^a,u^b)\in\Sigma_{a,b}^{+,*}$ we denote by (recalling \eqref{omegau})
\[
\bm{\omega}_u=(\bm{\omega}^a_{u^a},\bm{\omega}^b_{u^b}).
\]
Applying this for $n=1$ we deduce that, for a measurable function $f$ on $(\widetilde{X},\widetilde{\mathcal{F}})$,
\begin{align*}
&\int_{\widetilde{X}} f\circ \widetilde{T}(\widetilde{x})\, \mathbb{Q}_{\theta}(\mathrm{d}\widetilde{x})\\
=&\ \int_{\Omega_{a,b}}\int_{\mathbb{T}\times\Sigma_{a,b}^-}\int_{\Sigma_{a,b}^+}  f(R(t),\sigma_t\bm{v},\kappa_{t,\bm{v}}(\bm{\omega}))\, \widetilde{\theta}^+_{t,\bm{v}^-,\bm{\omega}}(\mathrm{d}\bm{v}^+) \theta^-(\mathrm{d}(t,\bm{v}^-))\mathbb{P}_{a,b}(\mathrm{d}\bm{\omega})\\
=&\  \int_{\Omega_{a,b}}\int_{\mathbb{T}\times\Sigma_{a,b}^-}\int_{\Sigma_{a,b}^+} \sum_{u\in C_{1,t}} f(R(t),\sigma_t\bm{v},\bm{\omega}_u)\cdot\\
&\qquad\qquad\qquad\qquad\qquad\qquad\quad\  \theta^+_{t,\bm{v}^-}([u])Q_{u}(\bm{\omega})\widetilde{\theta}_{R(t),\bm{v}^-u,\bm{\omega}_u}^{+}(\mathrm{d}(\sigma_t\bm{v})^+)\theta^-(\mathrm{d}(t,\bm{v}^-))\mathbb{P}_{a,b}(\mathrm{d}\bm{\omega})\\
=&\  \int_{\Omega_{a,b}}\int_{\mathbb{T}\times\Sigma_{a,b}^-}\int_{\Sigma_{a,b}^+} \sum_{u\in C_{1,t}} f(R(t),\sigma_t\bm{v},\bm{\omega})\theta^+_{t,\bm{v}^-}([u])\widetilde{\theta}_{R(t),\bm{v}^-u,\bm{\omega}}^{+}(\mathrm{d}(\sigma_t\bm{v})^+)\theta^-(\mathrm{d}(t,\bm{v}^-))\mathbb{P}_{a,b}(\mathrm{d}\bm{\omega})\\
=&\  \int_{\Omega_{a,b}}\int_{\mathbb{T}\times\Sigma_{a,b}^-}\int_{\Sigma_{a,b}^+} f(R(t),\sigma_t\bm{v},\bm{\omega}) \widetilde{\theta}_{R(t),(\sigma_t\bm{v})^-,\bm{\omega}}^{+}(\mathrm{d}(\sigma_t\bm{v})^+)\theta^-(\mathrm{d}(R(t),(\sigma_t\bm{v})^-))\mathbb{P}_{a,b}(\mathrm{d}\bm{\omega})\\
=&\  \int_{\Omega_{a,b}}\int_{\mathbb{T}\times\Sigma_{a,b}^-}\int_{\Sigma_{a,b}^+} f(s,\bm{w},\bm{\omega}) \widetilde{\theta}_{s,\bm{w}^-,\bm{\omega}}^{+}(\mathrm{d}\bm{w}^+)\theta^-(\mathrm{d}(s,\bm{w}^-))\mathbb{P}_{a,b}(\mathrm{d}\bm{\omega})\\
=&\ \int_{\widetilde{X}} f(\widetilde{x})\, \mathbb{Q}_{\theta}(\mathrm{d}\widetilde{x}),
\end{align*}
where we have used the facts that $Q_u$ has expectation $1$,
\[
f(R(t),\sigma_t\bm{v},\bm{\omega}_u)\theta^+_{t,\bm{v}^-}([u])\widetilde{\theta}_{R(t),\bm{v}^-u,\bm{\omega}_u}^{+}(\mathrm{d}(\sigma_t\bm{v})^+)\theta^-(\mathrm{d}(t,\bm{v}^-))
\]
is independent of $Q_u$, and has the same law as
\[
f(R(t),\sigma_t\bm{v},\bm{\omega})\theta^+_{t,\bm{v}^-}([u])\widetilde{\theta}_{R(t),\bm{v}^-u,\bm{\omega}}^{+}(\mathrm{d}(\sigma_t\bm{v})^+)\theta^-(\mathrm{d}(t,\bm{v}^-))
\]
under $\mathbb{P}_{a,b}$, and that, by \eqref{ss11},
\[
\begin{split}
f(R(t),\sigma_t\bm{v},\bm{\omega})\theta^+_{t,\bm{v}^-}([u])&\widetilde{\theta}_{R(t),\bm{v}^-u,\bm{\omega}}^{+}(\mathrm{d}(\sigma_t\bm{v})^+)\theta^-(\mathrm{d}(t,\bm{v}^-))\\
&=\ f(R(t),\sigma_t\bm{v},\bm{\omega}) \widetilde{\theta}_{R(t),(\sigma_t\bm{v})^-,\bm{\omega}}^{+}(\mathrm{d}(\sigma_t\bm{v})^+)\theta^-(\mathrm{d}(R(t),(\sigma_t\bm{v})^-))
\end{split}
\]
for $(\sigma_t\bm{v})^-\in [u]^-$.
\end{proof}

Furthermore, since $\theta$ is $T$-ergodic, similar to Lemma \ref{erg}, we have
\begin{lemma}\label{erg2}
The Peyri\`ere measure  $\mathbb{Q}_{\theta}$ is $\widetilde{T}$-ergodic.
\end{lemma}
\begin{proof}	
It is enough to prove that for $B_1,B_2\in \widetilde{\mathcal{A}}$ with $\mathbb{Q}_{\theta}(B_1),\mathbb{Q}_{\theta}(B_2)>0$, we have
\[
\lim_{k\to\infty}\mathbb{Q}_{\theta}(\widetilde{T}^{-k}(B_1)\cap B_2)>0.
\]
Note that since the sets $B_1,B_2\in\widetilde{\mathcal{A}}$ are only involved with random variables $V_u$, $W_v$ with words $u$, $v$ of bounded length, for $k$ large enough all the random variables involved in the sets $\widetilde{T}^{-k}B_1$ and $B_2$ will be independent. The rest of the proof is the same as Lemma \ref{erg}.
\end{proof}

\section{Exact-dimensionality of push-forward skewed fibre Mandelbrot cascade measures}\label{mcfm}

\subsection{Ergodic decomposition of the Peyri\`ere measure}

Define the probability measure $\mathbb{Q}_\Theta$ on $(\widetilde{X},\widetilde{\mathcal{F}})$ by $\mathbb{Q}_\Theta=\ell\times\mathbb{Q}_\mu\times\mathbb{Q}_\nu$, or equivalently,
\begin{equation}\label{QT}
\mathbb{E}_{\mathbb{Q}_\Theta}(f)=\int_{\widetilde{X}} f(t,\bm{i},\bm{j},\bm{\omega}) \,\ell(\mathrm{d}t)\widetilde{\mu}^+_{\bm{i}^-,\bm{\omega}^a}(\mathrm{d}\bm{i}^+)\mu^-(\mathrm{d}\bm{i}^-)\widetilde{\nu}^+_{\bm{j}^-,\bm{\omega}^b}(\mathrm{d}\bm{j}^+)\nu^-(\mathrm{d}\bm{j}^-)\mathbb{P}_{a,b}(\mathrm{d}\bm{\omega})
\end{equation}
for $\widetilde{\mathcal{F}}$-measurable function $f$ on $\widetilde{X}$. By using Lemma \ref{lminv} it is easy to see that $\mathbb{Q}_\Theta$ is $\widetilde{T}$-invariant, but it is not necessarily $\widetilde{T}$-ergodic. By using Lemma \ref{erg2} we have the following lemma about its ergodic decomposition.

\begin{lemma}\label{erg3}
$\mathbb{Q}_{\Theta^\zeta_x}$ for $\Theta$-a.e. $x\in X$ forms an ergodic decomposition of $\mathbb{Q}_{\Theta}$. In particular for  $f\in L^1(\widetilde{X},\widetilde{\mathcal{F}},\mathbb{Q}_\Theta)$ we have for $\mathbb{Q}_\Theta$-a.e. $\widetilde{x}=(x,\bm{\omega})\in \widetilde{X}$,
\[
\lim_{n\to\infty}\frac{1}{n}\sum_{k=0}^{n-1}f\circ \widetilde{T}(\widetilde{x})=\mathbb{E}_{\mathbb{Q}_{\Theta^\zeta_x}}(f).
\]
\end{lemma}

\begin{proof}
Recall the measurable partition $\zeta$ of $X$ in \eqref{med1} that generates the $T$-invariant $\sigma$-algebra $\mathcal{I}$ module $\Theta$-measure zero sets. Define the measurable partition $\widetilde{\zeta}$ of $\widetilde{X}$ by
\[
\widetilde{\zeta}=\{B\times \Omega_{a,b}: B\in \zeta\},
\]
i.e., $\zeta$ with the trivial partition on the probability space side. Note that for $\widetilde{x}=(x,\bm{\omega})$ the element $\widetilde{\zeta}(\widetilde{x})=\zeta(x)\times\Omega_{a,b}$ only depends on $x$. By using the sequence measures $\widetilde{\theta}^+_{t,\bm{v}^-,k}$ and $\widetilde{\mu}^+_{\bm{i}^-,k}$, $\widetilde{\nu}^+_{\bm{j}^-,k}$ in \eqref{tta} and \eqref{mcm} instead of the limiting measures $\widetilde{\theta}^+_{t,\bm{v}^-}$ and $\widetilde{\mu}^+_{\bm{i}^-}$, $\widetilde{\nu}^+_{\bm{j}^-}$ in \eqref{pmqtheta} and \eqref{QT} we may define two sequences of measure-valued martingales $(\mathbb{Q}_{\theta,k})_{k\ge 1}$ and $(\mathbb{Q}_{\Theta,k})_{k\ge 1}$. By the the uniqueness of conditional measures it is easy to verify that for $k\ge 1$ and for $\mathbb{Q}_{\Theta,k}$-a.e. $(x,\bm{\omega})\in \widetilde{X}$, $\mathbb{Q}_{\Theta^\zeta_x,k}=(\mathbb{Q}_{\Theta,k})^{\widetilde{\zeta}}_{\widetilde{x}}$. Note that these conditional measures only depends on $x$, and the projection of $\mathbb{Q}_{\Theta^\zeta_x,k}$ to $X$ is $\Theta$, which is independent of $k$. Then by the weak convergence and the non-degeneracy of these Mandelbrot cascade measures we deduce that for $\Theta$-a.e. $x$, or equivalently for $\mathbb{Q}_\Theta$-a.e. $\widetilde{x}=(x,\bm{\omega})$,
\[
\mathbb{Q}_{\Theta^\zeta_x}=(\mathbb{Q}_\Theta)^{\widetilde{\zeta}}_{\widetilde{x}}.
\]
Since $\Theta^\zeta_x$ is $T$-ergodic, by Lemma \ref{erg2} we have that $\mathbb{Q}_{\Theta^\zeta_x}$ is $\widetilde{T}$-ergodic. By using Birkhoff's ergodic theorem for $\mathbb{Q}_{\Theta^\zeta_x}$ we deduce that for any measurable function $f$ on $(\widetilde{X},\widetilde{\mathcal{F}})$, for $\mathbb{Q}_\Theta$-a.e. $\widetilde{x}=(x,\bm{\omega})\in \widetilde{X}$, for $(\mathbb{Q}_\Theta)^{\widetilde{\zeta}}_{\widetilde{x}}=\mathbb{Q}_{\Theta^\zeta_x}$-a.e. $\widetilde{y}\in \widetilde{\zeta}(\widetilde{x})$,
\[
\lim_{n\to\infty}\frac{1}{n}\sum_{k=0}^{n-1}f\circ \widetilde{T}(\widetilde{y})=\mathbb{E}_{\mathbb{Q}_{\Theta^\zeta_x}}(f)=\mathbb{E}_{(\mathbb{Q}_\Theta)^{\widetilde{\zeta}}_{\widetilde{x}}}(f).
\]
Note that for $\widetilde{y}\in \widetilde{\zeta}(\widetilde{x})$ we have $\widetilde{\zeta}(\widetilde{y})=\widetilde{\zeta}(\widetilde{x})$, hence we deduce that for $\mathbb{Q}_\Theta$-a.e. $\widetilde{x}\in \widetilde{X}$, for $(\mathbb{Q}_\Theta)^{\widetilde{\zeta}}_{\widetilde{x}}$-a.e. $\widetilde{y}\in \widetilde{\zeta}(\widetilde{x})$,
\[
\lim_{n\to\infty}\frac{1}{n}\sum_{k=0}^{n-1}f\circ \widetilde{T}(\widetilde{y})=\mathbb{E}_{(\mathbb{Q}_\Theta)^{\widetilde{\zeta}}_{\widetilde{y}}}(f)=\mathbb{E}_{\mathbb{Q}_{\Theta}}(f | \widehat{\widetilde{\zeta}})(\widetilde{y}),
\]
where we have used Theorem \ref{Rohthm} for the last equality. Theorem \ref{Rohthm} also implies that ``for $\mathbb{Q}_\Theta$-a.e. $\widetilde{x}\in \widetilde{X}$, for $(\mathbb{Q}_\Theta)^{\widetilde{\zeta}}_{\widetilde{x}}$-a.e. $\widetilde{y}\in \widetilde{\zeta}(\widetilde{x})$" is the same as ``for $\mathbb{Q}_\Theta$-a.e. $\widetilde{y}\in \widetilde{X}$", therefore we derive that for $\mathbb{Q}_\Theta$-a.e. $\widetilde{y}\in \widetilde{X}$,
\begin{equation}\label{bet22}
\lim_{n\to\infty}\frac{1}{n}\sum_{k=0}^{n-1}f\circ \widetilde{T}(\widetilde{y})=\mathbb{E}_{\mathbb{Q}_{\Theta}}(f | \widehat{\widetilde{\zeta}})(\widetilde{y}).
\end{equation}
Comparing \eqref{bet22} with the Birkhoff's ergodic theorem for $\mathbb{Q}_{\Theta}$ we deduce that $\widehat{\widetilde{\zeta}}=\widetilde{\mathcal{I}}=\{B\in \widetilde{\mathcal{F}}: \widetilde{T}^{-1}(B)=B\}$ module $\mathbb{Q}_\Theta$-measure zero sets. Hence $(\mathbb{Q}_\Theta)^{\widetilde{\zeta}}_{\widetilde{x}}$ for $\mathbb{Q}_\Theta$-a.e. $\widetilde{x}\in \widetilde{X}$, or equivalently $\mathbb{Q}_{\Theta^\zeta_x}$ for $\Theta$-a.e. $x\in X$, is an ergodic decomposition of $\mathbb{Q}_\Theta$.
\end{proof}

\subsection{Canonical mapping and exact-dimensionality}

We may attach two IFSs $\bm{I}$ and $\bm{J}$ to the two symbolic spaces $\Sigma_a^+$ and $\Sigma_b^+$, where $\bm{I}=\{f_i\}_{i=1}^a$ takes the form
	\[
	f_i(x)=\delta x+t_i,\ t_i\in\mathbb{R},
	\]
	and $\bm{J}=\{g_j\}_{j=1}^b$ takes the form
	\[
	g_j(x)=\rho x+s_j, \ s_j\in \mathbb{R}.
	\]
Note that we do not need to assume any separation condition here. For convenience we shortly denote by
\[
\Phi=\Phi_{\bm{I}}\times \Phi_{\bm{J}}:\Sigma_a^+\times\Sigma_b^+\to \mathbb{R}^2
\]
and $E=\Phi(\Sigma_a^+\times\Sigma_b^+)$.

Let $\widetilde{\phi}$ be the canonical projection from $\widetilde{X}$ to its subspace $\mathbb{T}\times\Sigma_{a,b}^-\times\Omega_{a,b}$, and let $\widetilde{\varphi}$ be the canonical mapping from $\widetilde{X}$ to its subspace $\Sigma_{a,b}^+$.

Denote by $\mathbb{Q}_\Theta^-=\mathbb{Q}_\Theta\circ \widetilde{\phi}^{-1}$. By \eqref{QT} we have
\[
\mathbb{Q}_\Theta^-(\mathrm{d}(t,\bm{v}^-,\bm{\omega}))=\|\widetilde{\mu}^+_{\bm{i}^-,\bm{\omega}^a}\| \mu^-(\mathrm{d}\bm{i}^-)\|\widetilde{\nu}^+_{\bm{j}^-,\bm{\omega}^b}\| \nu^-(\mathrm{d}\bm{j}^-) \ell(\mathrm{d}t)\mathbb{P}_{a,b}(\mathrm{d}\bm{\omega})
\]
for $\bm{v}^-=(\bm{i}^-,\bm{j}^-)$ and $\bm{\omega}=(\bm{\omega}^a,\bm{\omega}^b)$, and we may write
\begin{equation}\label{PQT}
\mathbb{Q}_\Theta(\mathrm{d}\widetilde{x})=\bar{\lambda}^+_{\bm{v}^-,\bm{\omega}}(\mathrm{d}\bm{v}^+)\mathbb{Q}_\Theta^-(\mathrm{d}(t,\bm{v}^-,\bm{\omega})),
\end{equation}
for $\widetilde{x}=(t,\bm{v},\bm{\omega})\in \widetilde{X}$, where
\[
\bar{\lambda}^+_{\bm{v}^-,\bm{\omega}}(\mathrm{d}\bm{v}^+)=\bar{\mu}^+_{\bm{i}^-,\bm{\omega}^a}(\mathrm{d}\bm{i}^+)\bar{\nu}^+_{\bm{j}^-,\bm{\omega}^b}(\mathrm{d}\bm{j}^+)
\]
for $\bm{v}^-=(\bm{i}^-,\bm{j}^-)$ and $\bm{v}^+=(\bm{i}^+,\bm{j}^+)$. Similarly denote by $\mathbb{Q}_\theta^-=\mathbb{Q}_\theta\circ \widetilde{\phi}^{-1}$. By \eqref{pmqtheta} we have 
\[
\mathbb{Q}_\theta^-(\mathrm{d}(t,\bm{v}^-,\bm{\omega}))=\|\widetilde{\theta}^+_{t,\bm{v}^-,\bm{\omega}}\| \theta^-(\mathrm{d}(t,\bm{v}^-))\mathbb{P}_{a,b}(\mathrm{d}\bm{\omega}),
\]
and we may write
\begin{equation}\label{PQT2}
\mathbb{Q}_\theta(\mathrm{d}\widetilde{x})=\bar{\theta}^+_{t,\bm{v}^-,\bm{\omega}}(\mathrm{d}\bm{v}^+)\mathbb{Q}_\theta^-(\mathrm{d}(t,\bm{v}^-,\bm{\omega})),
\end{equation}
for $\widetilde{x}=(t,\bm{v},\bm{\omega})\in \widetilde{X}$.

By Lemma \ref{lem1}, Theorem \ref{BJ21} and \eqref{ed123} we have that for $\mathbb{Q}_\Theta^-$-a.e. (resp. for $\mathbb{Q}_\theta^-$-a.e.) $(t,\bm{v}^-,\bm{\omega})\in \mathbb{T}\times\Sigma_{a,b}^-\times\Omega_{a,b}$, the measure $\bar{\lambda}^+_{\bm{v}^-,\bm{\omega}}$ (resp. $\bar{\theta}^+_{t,\bm{v}^-,\bm{\omega}}$) is exact-dimensional on $(\Sigma_{a,b}^+,d_{\delta,\rho}^+)$ with dimension
\[
\frac{h_\mu-h_V}{-\log \delta}+\frac{h_\nu-h_W}{-\log \rho}.
\]
By Theorem \ref{ed1} and Theorem \ref{ed22} we have that for $\mathbb{P}_{a,b}$-a.e. $\bm{\omega}\in \Omega_{a,b}$, as product measures, both $\Phi(\bar{\lambda}^+_{\bm{v}^-,\bm{\omega}})=\Phi_{\bm{I}}(\bar{\mu}^+_{\bm{i}^-,\bm{\omega}^a})\times \Phi_{\bm{J}}(\bar{\nu}^+_{\bm{j}^-,\bm{\omega}^b})$ for $\mu^-\times\nu^-$-a.e. $\bm{v}^-\in\Sigma_{a,b}^-$ and $\Phi(\bar{\mu}\times\bar{\nu})=\Phi_{\bm{I}}(\bar{\mu})\times\Phi_{\bm{J}}(\bar{\nu})$ are exact-dimensional on $\mathbb{R}^2$ with the same dimension
\begin{equation}\label{ed34}
D_{\bar{\mu},\bar{\nu},\Phi}=D_{\mu,V,\bm{I}}+D_{\nu,W,\bm{J}}=\frac{h_{\mu}-h_V-h_{\mu,V,\bm{I}}}{-\log \delta}+\frac{h_{\nu}-h_W-h_{\nu,W,\bm{J}}}{-\log \rho}.
\end{equation}
We expect the same holds for $\bar{\theta}^+_{t,\bm{v}^-,\bm{\omega}}$, but this is highly not trivial when the IFSs do not satisfy the OSC. Note that since $\widetilde{T}$ is not invertible, we can not use Proposition \ref{SMB} here.

\begin{theorem}\label{ed2}
For $\mathbb{Q}_\theta^-$-a.e. $(t,\bm{v}^-,\bm{\omega})\in \mathbb{T}\times\Sigma_{a,b}^-\times\Omega_{a,b}$, the measure $\Phi(\bar{\theta}^+_{t,\bm{v}^-,\bm{\omega}})$ is exact-dimensional on $\mathbb{R}^2$ with dimension $D_{\bar{\mu},\bar{\nu},\Phi}$.
\end{theorem}

\begin{proof}
For $t\in\mathbb{T}$, $\bm{v}^+=(\bm{i}^+,\bm{j}^+)\in \Sigma_{a,b}^+$ and $n\ge 0$ define the following rectangle centred at $\Phi(\bm{v}^+)\in \mathbb{R}^2$,
\[
U_{n,t}(\bm{v}^+)=B(\Phi_{\bm{I}}(\bm{i}^+),R_{\bm{I}}\delta^{\tau_n(t)})\times B(\Phi_{\bm{J}}(\bm{j}^+),R_{\bm{J}}\rho^{n}),
\]
where we set $\tau_0(t)=0$. By the fact that $|\tau_n(t)-n\alpha|\le 2$ we have for all $t\in\mathbb{T}$, $\bm{v}^+=(\bm{i}^+,\bm{j}^+)\in \Sigma_{a,b}^+$ and $n\ge 0$,
\[
B(\Phi(\bm{v}^+),r\cdot \rho^n )\subset U_{n,t}(\bm{v}^+)\subset B(\Phi(\bm{v}^+),R\cdot \rho^n),
\]
where $r=\delta^2\min\{R_{\bm{I}},R_{\bm{J}}\}$ and $R=\delta^{-2}\max\{R_{\bm{I}},R_{\bm{J}}\}$. Denote by
\[
\gamma=\frac{R}{r}=\delta^{-4}\frac{\max\{R_{\bm{I}},R_{\bm{J}}\}}{\min\{R_{\bm{I}},R_{\bm{J}}\}}\ge 1.
\]
Then it is easy to see that the family
\begin{equation}\label{Ug}
\mathcal{U}=\{U_{n,t}(\bm{v}^+): t\in\mathbb{T},\ \bm{v}^+=(\bm{i}^+,\bm{j}^+)\in \Sigma_{a,b}^+,\ n\ge 0\}
\end{equation}
is a fine $\gamma$-Morse covering of $E=\Phi(\Sigma_{a,b}^+)=\Phi_{\bm{I}}(\Sigma_a^+)\times\Phi_{\bm{J}}(\Sigma_b^+)$. For $U_{n,t}(\bm{v}^+)\in \mathcal{U}$ shortly denote by
\[
U^\Phi_{n,t}(\bm{v}^+)=\Phi^{-1}(U_{n,t}(\bm{v}^+))
\]
the corresponding set in $\Sigma_{a,b}^+$.

For $u\in \Sigma_{a,b}^{+,*}$ and $B\in \mathcal{B}_{a,b}$ denote by
\[
uB=\{u\bm{v}^+:\bm{v}^+\in B\}.
\]
Recall the notation $C_{n,t}$ in \eqref{Cnt}. For $t\in \mathbb{T}$ define $\mathcal{P}_t=\{[u]:u\in C_{1,t}\}$. Applying \cite[Lemma 3.7]{FJ14} to $(\Sigma_a^+, \Phi_{\bm{I}})$ and $(\Sigma_b^+, \Phi_{\bm{J}})$ respectively we obtain that for $n\ge 1$,
\[
U^\Phi_{n,t}(\bm{v}^+)\cap \mathcal{P}_t(\bm{v}^+)=\bm{v}^+|_{1,t}U^\Phi_{n-1,R(t)}(\sigma_t\bm{v}^+).
\]
By \eqref{ss11} we have
\begin{align*}
\widetilde{\theta}^+_{t,\bm{v}^-,\bm{\omega}}(\bm{v}^+|_{k,t}B)=Q_{\bm{v}^+|_{k,t}}(\bm{\omega})\cdot\theta_{t,\bm{v}^-}^+([\bm{v}^+|_{k,t}])\cdot \widetilde{\theta}_{R^k(t),\bm{v}^-\bm{v}^+|_{k,t},\bm{\omega}_{\bm{v}^+|_{k,t}}}^{+}(B)
\end{align*}
Therefore
\begin{align*}
&\ \frac{\widetilde{\theta}^+_{t,\bm{v}^-,\bm{\omega}}(\bm{v}^+|_{k+1,t}U^\Phi_{n-k-1,R^{k+1}(t)}(\sigma_{k+1,t}\bm{v}^+))}{\widetilde{\theta}^+_{t,\bm{v}^-,\bm{\omega}}(\bm{v}^+|_{k,t}U^\Phi_{n-k,R^k(t)}(\sigma_{k,t}\bm{v}^+))}\\
=&\ \frac{\widetilde{\theta}^+_{R^k(t),\bm{v}^-\bm{v}^+|_{k,t},\bm{\omega}_{\bm{v}^+|_{k,t}}}(\sigma_{k,t}\bm{v}^+|_{1,R^{k}t}U^\Phi_{n-k-1,R^{k+1}(t)}(\sigma_{k,t}\bm{v}^+))}{\widetilde{\theta}^+_{R^k(t),\bm{v}^-\bm{v}^+|_{k,t},\bm{\omega}_{\bm{v}^+|_{k,t}}}(U^\Phi_{n-k,R^k(t)}(\sigma_{k,t}\bm{v}^+))}\\
=&\ \frac{\bar{\theta}^+_{R^k(t),\bm{v}^-\bm{v}^+|_{k,t},\bm{\omega}_{\bm{v}^+|_{k,t}}}(U^\Phi_{n-k,R^k(t)}(\sigma_{k,t}\bm{v}^+)\cap \mathcal{P}_{R^k(t)}(\sigma_{k,t}\bm{v}^+))}{\bar{\theta}^+_{R^k(t),\bm{v}^-\bm{v}^+|_{k,t},\bm{\omega}_{\bm{v}^+|_{k,t}}}(U^\Phi_{n-k,R^k(t)}(\sigma_{k,t}\bm{v}^+))}\\
\end{align*}
Define
\[
F_n(t,\bm{v}^-,\bm{v}^+,\bm{\omega})=-\log \frac{\bar{\theta}^+_{t,\bm{v}^-,\bm{\omega}}(U^\Phi_{n,t}(\bm{v}^+)\cap \mathcal{P}_t(\bm{v}^+))}{\bar{\theta}^+_{t,\bm{v}^-,\bm{\omega}}(U^\Phi_{n,t}(\bm{v}^+))}
\]
Then we have
\[
\log \frac{\widetilde{\theta}^+_{t,\bm{v}^-,\bm{\omega}}(U^\Phi_{n,t}(\bm{v}^+))}{\widetilde{\theta}^+_{t,\bm{v}^-,\bm{\omega}}([\bm{v}^+|_{n,t}])}=\sum_{k=0}^{n-1}F_{n-k}\circ \widetilde{T}^k(t,\bm{v}^-,\bm{v}^+,\bm{\omega}),
\]
where we have used the fact that, by the choices of $R_{\bm{I}}$, $R_{\bm{J}}$, $U^\Phi_{0,t}(\bm{v}^+)=\Sigma_{a,b}^+$ is the whole set for all $t\in \mathbb{T}$ and $\bm{v}^+\in\Sigma_{a,b}^+$.

By Proposition \ref{P3.5}, for $\mathbb{Q}_\theta^-$-a.e. $(t,\bm{v}^-,\bm{\omega})$, for $\bar{\theta}^+_{t,\bm{v}^-,\bm{\omega}}$-a.e. $\bm{v}^+$, as $n\to\infty$, the sequence $F_n(t,\bm{v}^-,\bm{v}^+,\bm{\omega})$ converges to
\[
F(t,\bm{v}^-,\bm{v}^+,\bm{\omega})=I_{\bar{\theta}^+_{t,\bm{v}^-,\bm{\omega}}}(\mathcal{P}_t|\mathcal{B}_{\bm{I}}\otimes\mathcal{B}_{\bm{J}})(\bm{v}^+)
\]
and there exists a positive measurable function $G(t,\bm{v}^-,\cdot,\bm{\omega})\in L^1(\Sigma_{a,b}^+,\mathcal{B}_{a,b}^+,\bar{\theta}^+_{t,\bm{v}^-,\bm{\omega}})$ such that $F_n\le G$ and
\[
\int_{\Sigma_{a,b}^+} G(t,\bm{v}^-,\bm{v}^+,\bm{\omega})\,\bar{\theta}^+_{t,\bm{v}^-,\bm{\omega}}(\mathrm{d}\bm{v}^+) \le ab\log ab+C<\infty.
\]
By using the exact-dimensionality of $\widetilde{\theta}^+_{t,\bm{v}^-,\bm{\omega}}$ in \eqref{ed123} and using Theorem \ref{met} we deduce that for $\mathbb{Q}_\theta$-a.e. $(t,\bm{v},\bm{\omega})$, or equivalently (by \eqref{PQT2}), for $\mathbb{Q}_\theta^-$-a.e. $(t,\bm{v}^-,\bm{\omega})$, for $\bar{\theta}^+_{t,\bm{v}^-,\bm{\omega}}$-a.e. $\bm{v}^+$,
\begin{align*}
\lim_{n\to\infty}\frac{-\log \widetilde{\theta}^+_{t,\bm{v}^-,\bm{\omega}}(U^\Phi_{n,t}(\bm{v}^+))}{n}&=\ \lim_{n\to\infty}\frac{-\log\widetilde{\theta}^+_{t,\bm{v}^-,\bm{\omega}}([\bm{v}^+|_{n,t}]) }{n}-\frac{1}{n}\sum_{k=0}^{n-1} F_{n-k}\circ \widetilde{T}^k(t,\bm{v}^-,\bm{v}^+,\bm{\omega})\\
&=\ \alpha(h_\mu-h_V)+(h_\nu-h_W)-\mathbb{E}_{\mathbb{Q}_\theta}(F).\numberthis \label{ed56}
\end{align*}

Recall that $\theta=\Theta^\zeta_z$ is the ergodic decomposition of $\Theta$ at a $\Theta$-typical $z\in X$. By Lemma \ref{erg3} we may write
\[
\mathbb{E}_{\mathbb{Q}_\Theta}(f)=\int_{\mathcal{M}(X)}\mathbb{E}_{\mathbb{Q}_{\theta}}(f)\,\Theta(\mathrm{d}\theta)
\]
for any measurable function $f$ on $(\widetilde{X},\widetilde{\mathcal{F}})$. Hence by \eqref{PQT} and \eqref{PQT2} we have that
\begin{equation}\label{PQT3}
\begin{split}
\int_{\mathbb{T}\times\Sigma_{a,b}^-\times\Omega_{a,b}}\mathbb{E}_{\bar{\lambda}^+_{\bm{v}^-,\bm{\omega}}}(f^+)&\,\mathbb{Q}_\Theta^-(\mathrm{d}(t,\bm{v}^-,\bm{\omega}))\\
=&\ \int_{\mathcal{M}(X)}\int_{\mathbb{T}\times\Sigma_{a,b}^-\times\Omega_{a,b}}\mathbb{E}_{\bar{\theta}^+_{t,\bm{v}^-,\bm{\omega}}}(f^+)\,\mathbb{Q}_{\theta}^-(\mathrm{d}(t,\bm{v}^-,\bm{\omega})) \Theta(\mathrm{d}\theta)
\end{split}
\end{equation}
for any measurable function $f^+$ on $(\Sigma_{a,b}^+,\mathcal{B}_{a,b}^+)$. By the definition of lower Hausdorff dimension we deduce that the dimension of $\bar{\lambda}^+_{\bm{v}^-,\bm{\omega}}$ on $(\Sigma_{a,b}^+,d_{\delta,\rho}^+)$ can not be lower than the essential lower dimension of $\bar{\theta}^+_{t,\bm{v}^-,\bm{\omega}}$, which is 
\[
\frac{\alpha(h_\mu-h_V)+(h_\nu-h_W)-\mathbb{E}_{\mathbb{Q}_\theta}(F)}{-\log \rho}
\]
by \eqref{ed56}. Comparing this with \eqref{ed34} we deduce that
\begin{equation}\label{EQF}
\mathbb{E}_{\mathbb{Q}_\theta}(F)\ge \alpha h_{\mu,V,\bm{I}} + h_{\nu,W,\bm{J}}.
\end{equation}

On the other hand, by \eqref{PQT2} we have
\[
\mathbb{E}_{\mathbb{Q}_\theta}(F)=\int_{\Sigma_{a,b}^-\times\Omega_{a,b}} H_{\bar{\theta}_{t,\bm{v}^-,\bm{\omega}}}(\mathcal{P}_t|\mathcal{B}_{\bm{I}}\otimes\mathcal{B}_{\bm{J}}) \,\mathbb{Q}_\theta^-(\mathrm{d}(t,\bm{v}^-,\bm{\omega})).
\]
Similarly we have
\[
h_{\mu,V,\bm{I}}=\int_{\Sigma_a^-\times\Omega_a} H_{\bar{\mu}^+_{\bm{i}^-,\bm{\omega}^a}}(\mathcal{P}^+_a|\mathcal{B}_{\bm{I}}) \,\mathbb{Q}_\mu^-(\mathrm{d}(\bm{i}^-,\bm{\omega}^a));
\]
\[
h_{\nu,W,\bm{J}}=\int_{\Sigma_b^-\times\Omega_b} H_{\bar{\nu}^+_{\bm{j}^-,\bm{\omega}^b}}(\mathcal{P}^+_b|\mathcal{B}_{\bm{J}}) \,\mathbb{Q}_\nu^-(\mathrm{d}(\bm{j}^-,\bm{\omega}^b)).
\]
Note that $\mathcal{P}_t=\mathcal{P}^+_a\times\mathcal{P}^+_b$ if $t\in A$ and $\mathcal{P}_t=\{\emptyset,\Sigma^+_a\}\times\mathcal{P}^+_b$ if $t\in A^c$. Therefore
\begin{align*}
&\ \alpha h_{\mu,V,\bm{I}} + h_{\nu,W,\bm{J}}\\
=&\ \int_{\Sigma_{a,b}^-\times\Omega_{a,b}} H_{\bar{\mu}^+_{\bm{i}^-,\bm{\omega}^a}\times \bar{\nu}^+_{\bm{j}^-,\bm{\omega}^b}}(\mathcal{P}_t|\mathcal{B}_{\bm{I}}\otimes\mathcal{B}_{\bm{J}})\, \ell(\mathrm{d}t) \mathbb{Q}_\mu^-(\mathrm{d}(\bm{i}^-,\bm{\omega}^a))\mathbb{Q}_\nu^-(\mathrm{d}(\bm{j}^-,\bm{\omega}^b))\\
=&\ \int_{\Sigma_{a,b}^-\times\Omega_{a,b}} H_{\bar{\lambda}^+_{\bm{v}^-,\bm{\omega}}}(\mathcal{P}_t|\mathcal{B}_{\bm{I}}\otimes\mathcal{B}_{\bm{J}})\, \mathbb{Q}_\Theta^-(\mathrm{d}(t,\bm{v}^-,\bm{\omega})).
\end{align*}
Following the same approach as in the proof of \cite[Proposition 4.17, p.29]{FH09}, by the concavity of conditional entropy as a function of probability measures, and then by Jensen's inequality and \eqref{PQT3} we deduce that
\begin{align*}
\alpha h_{\mu,V,\bm{I}} + h_{\nu,W,\bm{J}}\ge &\ \int_{\mathcal{M}(X)}\int_{\Sigma_{a,b}^-\times\Omega_{a,b}} H_{\bar{\theta}^+_{t,\bm{v}^-,\bm{\omega}}}(\mathcal{P}_t|\mathcal{B}_{\bm{I}}\otimes\mathcal{B}_{\bm{J}})\, \mathbb{Q}_{\theta}^-(\mathrm{d}(t,\bm{v}^-,\bm{\omega})) \Theta(\mathrm{d}\theta)\\
=&\ \int_{\mathcal{M}(X)} \mathbb{E}_{\mathbb{Q}_{\theta}}(F)\, \Theta(\mathrm{d}\theta)
\end{align*}
Together with \eqref{EQF} we deduce that for $\Theta$-typical $\theta$,
\[
\mathbb{E}_{\mathbb{Q}_{\theta}}(F)=\alpha h_{\mu,V,\bm{I}} + h_{\nu,W,\bm{J}}.
\]
Then we get the conclusion by \eqref{ed56}.
\end{proof}

\section{Dimension of Projections}\label{DoP}

\subsection{Projections and entropy}
Recall that we denote by $\Pi$ the set of all orthogonal projections from $\mathbb{R}^2$ to its one-dimensional subspaces, with $\pi_1$, $\pi_2$ being the horizontal and vertical projections respectively. We alternate and parameterise $\Pi\backslash\{\pi_1,\pi_2\}$ by projections $\pi_s^\pm\colon\mathbb{R}^2\to\mathbb{R}$, $s\in\mathbb{R}$, given by
\[
\pi_s^{\pm}(x,y)=\delta^sx\pm y,\;\;(x,y)\in\mathbb{R}^2.
\]
We shall omit the notation $\pm$ for most of the time as it is irrelevant to the proofs. The \emph{r-scaling entropy} of a measure $\lambda$ on a Euclidean space is defined to be
\[
H_r(\lambda)=\int_{\text{supp}(\lambda)}-\log\lambda(B(x,r))\,\lambda(\mathrm{d}x).
\]
We know that $H_r(\lambda)$ is lower semicontinuous as it may be expressed a the limit of an increasing sequence of continuous functions of the form
\[
\lambda\to\int\min\Big\{k,-\log \big(\int f_k(x-y)\lambda(dy)\big)\Big\}\,\lambda(dx),
\]
where $f_k$ is a decreasing sequence of continuous functions approximating $\mathbf{1}_{B(0,r)}$.
We define the \emph{entropy dimension} of a measure $\lambda$ to be
\begin{equation}
\label{eqn:entdim}
\dim_e\lambda=\liminf_{r\to0}\frac{H_r(\lambda)}{-\log r}.
\end{equation}
It is well-known that for any measure $\underline{\dim}_H\lambda\le\dim_e\lambda$, with equality when the measure is exact-dimensional.

\subsection{Lower bound for the dimension of projections}
	
For $t\in\mathbb{R}$ define
\[
S_t\colon\mathbb{R}^2\ni(x,y)\to(\delta^{t}x,y).
\]
For $\widetilde{x}=(t,\bm{v}^-,\bm{v}^+,\bm{\omega})\in \widetilde{X}$ define
\[
M(\widetilde{x})=S_t\Phi(\bar{\theta}^+_{t,\bm{v}^-,\bm{\omega}}),
\]
For $q\ge 1$ define
\[
E_{q,\theta}(\pi_s)=\int_{\widetilde{X}}\frac{1}{q\log(1/\rho)}H_{\rho^q}(\pi_s M(\widetilde{x}))\, \mathbb{Q}_\theta(\mathrm{d}\widetilde{x}).
\]
We will prove the following theorem:
\begin{theorem}\label{thm6}
For $\mathbb{Q}_\theta^-$-a.e. $(t,\bm{v}^-,\bm{\omega})\in \mathbb{T}\times\Sigma_{a,b}^-\times\Omega_{a,b}$, 
\begin{equation}\label{eqn:lb6}
\underline{\dim}_H \pi_sS_t\Phi(\bar{\theta}^+_{t,\bm{v}^-,\bm{\omega}})\ge E_{q,\theta}(\pi_s)-O(1/q)\;\;\text{ for all } s\in\mathbb{R},
\end{equation}
where the implied constant in $O(1/q)$ depends only on $\delta$, $\rho$ and $E$.
\end{theorem}

\begin{proof}
Recall the transformation
\[
\widetilde{T}(\widetilde{x})=(t,\sigma_t(\bm{v}),\kappa_{t,\bm{v}}(\bm{\omega})), \ \widetilde{x}=(t,\bm{v},\bm{\omega})\in \widetilde{X}.
\]
Then we have for $k\ge 1$,
\begin{align*}
M\circ \widetilde{T}^k(\widetilde{x})=&\ S_{R^k(t)}\Phi(\bar{\theta}^+_{R^k(t),\sigma_{k,t}\bm{v}^-,\bm{\omega}_{\bm{v}^+|_{k,t}}}).
\end{align*}
Let $\mathcal{M}$ be the set of all probability measures on the closed ball $B(0,R_E)$ in $\mathbb{R}^2$, where we denote by $R_E=|E|$. For any bounded continuous function $f$ on $\mathcal{M}$ (w.r.t the weak-topology), by Lemma \ref{erg} and Birkhoff's ergodic theorem we have for $\mathbb{Q}_\theta$-a.e. $\widetilde{x}\in\widetilde{X}$,
\begin{equation}\label{bet}
\lim_{n\to\infty}\frac{1}{n}\sum_{k=0}^{n-1}f(M\circ \widetilde{T}^k(\widetilde{x}))=\mathbb{E}_{\mathbb{Q}_\theta}(f(M)).
\end{equation}
In fact, since the space of continuous functions on $\mathcal{M}$ is separable under $\|\cdot \|_\infty$ norm, the above limit holds for all continuous functions on $\mathcal{M}$ simultaneously (see \cite[Proposition 4.1]{FJ14} for example). 

Same as in \cite[Lemma 7.1]{HS12}, for $q\ge 1$ we have
\[
\lim_{n\to\infty}\frac{1}{n}\sum_{k=0}^{n-1}f(M\circ \widetilde{T}^{qk}(\widetilde{x}))=\mathbb{E}_{\mathbb{Q}_{\theta,q}}(f(M))
\]
where $\mathbb{Q}_{\theta,q}$ is the ergodic decomposition of $\mathbb{Q}_{\theta}$ w.r.t. $\widetilde{T}^q$ hence
\[
\frac{1}{q}\sum_{l=0}^{q-1}\mathbb{Q}_{\theta,q}\circ \widetilde{T}^l=\mathbb{Q}_{\theta}
\]
Due to the lower semicontinuity of $H_{\rho^q}(\pi_s(\cdot))$, we may find a countable dense sequence of continuous functions $f_k$ on $\mathcal{M}$ that can approximate $H_{\rho^q}(\pi_s (\cdot))$ from below for all $s\in \mathbb{R}$. Thus we have for $l\in\{0,\cdots,q-1\}$, for $\mathbb{Q}_{\theta,q}$-a.e. $\widetilde{x}\in \widetilde{X}$, for all $s\in\mathbb{R}$,
\[
\liminf_{n\to\infty}\frac{1}{n}\sum_{k=0}^{n-1}\frac{1}{\log(1/\rho)}H_{\rho^q}(\pi_s (M\circ \widetilde{T}^{l+qk}(\widetilde{x})))\ge \int_{\widetilde{X}}\frac{1}{\log(1/\rho)}H_{\rho^q}(\pi_s M\circ \widetilde{T}^l(\widetilde{x}))\, \mathbb{Q}_{\theta,q}(\mathrm{d}\widetilde{x}).
\]
Since
\[
\frac{1}{q}\sum_{l=0}^{q-1} \int_{\widetilde{X}}\frac{1}{\log(1/\rho)}H_{\rho^q}(\pi_s M\circ \widetilde{T}^l(\widetilde{x}))\, \mathbb{Q}_{\theta,q}(\mathrm{d}\widetilde{x})=qE_{q,\theta}(\pi_s)
\]
for each $s\in\mathbb{R}$ there exists $l_s\in\{0,\cdots,q-1\}$ independent of $\widetilde{x}$ such that
\[
\int_{\widetilde{X}}\frac{1}{q\log(1/\rho)}H_{\rho^q}(\pi_s M\circ \widetilde{T}^{l_s}(\widetilde{x}))\, \mathbb{Q}_{\theta,q}(\mathrm{d}\widetilde{x})\ge E_{q,\theta}(\pi_s).
\]
This implies that for $\mathbb{Q}_{\theta,q}$-a.e. $\widetilde{x}\in \widetilde{X}$, for all $s\in\mathbb{R}$,
\[
\liminf_{n\to\infty}\frac{1}{n}\sum_{k=0}^{n-1}\frac{1}{q\log(1/\rho)}H_{\rho^q}(\pi_s (M\circ \widetilde{T}^{l_s+qk}(\widetilde{x})))\ge E_{q,\theta}(\pi_s).
\]
As $\mathbb{Q}_{\theta,q}$ is an ergodic decomposition of $\mathbb{Q}_{\theta}$, we finally derive that for $\mathbb{Q}_{\theta}$-a.e. $\widetilde{x}\in \widetilde{X}$, for all $s\in\mathbb{R}$,
\begin{equation}\label{eqn:lb2}
\liminf_{n\to\infty}\frac{1}{n}\sum_{k=0}^{n-1}\frac{1}{q\log(1/\rho)}H_{\rho^q}(\pi_s (M\circ \widetilde{T}^{l_s+qk}(\widetilde{x})))\ge E_{q,\theta}(\pi_s).
\end{equation}

%

We shall use the $\rho$-tree method in \cite{HS12} to obtain lower bounds for the dimensions of projections of measures. Fix an instance of $\widetilde{x}=(t,\bm{v},\bm{\omega})\in \widetilde{X}$, $s\in\mathbb{R}$ and $l_s\in\{0,\ldots,q-1\}$ such that \eqref{eqn:lb2} holds. Shortly denote by
\[
\bar{\theta}'=\bar{\theta}^+_{R^{l_s}(t),\sigma_{l_s,t}\bm{v}^-,\bm{\omega}_{\bm{v}^+|_{l_s,t}}}, \bm{w}=\sigma_{l_s,t}(\bm{v}^+), t'=R^{l_s}(t).
\]
For  $k\ge 1$ consider finite partition of $\Sigma_{a,b}^+$:
\[
\mathcal{P}_{t'}^{qk}=\{[\bm{v}^+|_{qk,t'}]=[\bm{i}^+|_{\tau_{qk}(t')}]\times[\bm{j}^+|_{qk}]: \bm{v}^+=(\bm{i}^+,\bm{j}^+)\in \Sigma_{a,b}^+\}.
\]
We may map $\mathcal{P}_{t'}^{qk}$ to $\mathbb{N}^k$ for each $k\ge 1$ in a way such that each element of $\mathcal{P}_{t'}^{qk}$ is represented by a word of length $k$, and for $u\in \mathbb{N}^k$ that represents an element $C_u\in \mathcal{P}_{t'}^{qk}$, and for $u'\in \mathbb{N}^{k+1}$ that represents an element $C_{u'}\in \mathcal{P}_{t'}^{q(k+1)}$, if $C_{u'}\subset C_{u}$, then $u'|_k=u$. All the words that represent $\{\mathcal{P}_{t'}^{qk}\}_{k\ge 1}$ then form the set of finite words of a tree in $\mathbb{N}^\mathbb{N}$ denoted by $Y_{q,t'}$. We equip $Y_{q,t'}$ with the metric $d_{\rho^q}$ so that it becomes a $\rho^q$-tree. Note also that for each $\bm{v}^+\in \Sigma_{a,b}^+$ there is a unique sequence $\{[\bm{v}^+|_{qk,t'}]\}_{k\ge1}$ in $\{\mathcal{P}_{t'}^{qk}\}_{k\ge 1}$ that converges to $\{\bm{v}^+\}$, therefore there is a one-to-one correspondence between $Y_{q,t'}$ and $\Sigma_{a,b}^+$. Hence we may identify $Y_{q,t'}$ with $\Sigma_{a,b}^+$, and for $k\ge 1$,  $[\bm{v}^+|_{qk,t'}]$ can be viewed as the cylinder of generation $k$ in the $\rho^q$-tree $Y_{q,t'}$ that contains $\bm{v}^+$.

For a measure $\lambda$ on $Y_{q,t'}$ and a finite word $u$ in $Y_{q,t'}$ with $\lambda([u])>0$ we denote by
\[
\lambda_{[u]}(B)=\frac{\lambda([u]\cap B)}{\lambda([u])},\ B\in \mathcal{B}_{Y_{q,t'}}
\]
the measure $\lambda$ conditioned on $[u]$, and
\[
\lambda^{[u]}(B)=\frac{\lambda(uB)}{\lambda([u])}, \ B\in \mathcal{B}_{Y_{q,t'}}.
\]
the zooming-in of $\lambda$ at $u$, where $uB=\{u\bm{w}:\bm{w}\in B\}$. Then by \eqref{ss11} we have
	\begin{align*}
M\circ \widetilde{T}^{l_s+qk}(\widetilde{x})=&\ S_{R^{qk}(t')}\Phi(\bar{\theta}^+_{t,\bm{v}^-\bm{v}^+|_{l_s,t}\bm{w}^+|_{t',qk},\bm{\omega}_{\bm{v}^+|_{l_s,t}\bm{w}|_{t',qk}}})\\
=&\ S_{\varphi^{qk}(t')}\Phi((\overline{\theta}')^{[\bm{w}|_{qk,t'}]}).
\end{align*}
Hence \eqref{eqn:lb2} implies for $\bar{\theta}'$-a.e. $\bm{w}\in Y_{q,t'}$, 
\begin{align*}
\liminf_{n\to\infty}\frac{1}{n}\sum_{k=0}^{n-1}\frac{1}{q\log(1/\rho)}H_{\rho^q}(\pi_s S_{\varphi^{qk}(t')}\Phi((\bar{\theta}')^{[\bm{w}|_{qk,t'}]})) \ge E_{q,\theta}(\pi_s) \numberthis\label{eqn:lb3}.
\end{align*}

By \eqref{taunt}, the canonical map $\Phi\colon (Y_{q,t'},d_{\rho^q})\to E$ is $c_E$-Lipschitz, where $c_E=R_E^2c_\delta$ for $c_\delta:=\sqrt{\delta^2+1}$. Indeed, let $\bm{v}^1=(\bm{i}^1,\bm{j}^1),\bm{v}^2=(\bm{i}^2,\bm{j}^2)\in Y_{q,t'}$ be such that $d_{\rho^q}(\bm{v}^1, \bm{v}^2)=\rho^{qk}$. This implies that their first $k$ letters must coincide, i.e., $\bm{v}^1|_{qk,t'}=\bm{v}^2|_{qk,t'}$. This means that $f_{\bm{i}^1|_{qk,t'}}=f_{\bm{i}^2|_{qk,t'}}$ and $g_{\bm{j}^1|_{qk}}=g_{\bm{j}^2|_{qk}}$ for $l=1,2$, and in particular they form the same translation on $\mathbb{R}^2$. Then, 
\begin{align*}
|\Phi(\bm{v}^1)-\Phi(\bm{v}^2)|&\le|\Phi([\bm{v}^1|_{qk,t'}])|\cdot |\Phi\circ \sigma_{qk,t'}(\bm{v}^1)-\Phi\circ \sigma_{qk,t'}(\bm{v}^2)|\\
&\le R_Ec_\delta\rho^{qk} \cdot R_E\\
&=R_E^2c_\delta\cdot d_{\rho^q}(\bm{v}^1,\bm{v}^2).
\end{align*}
For $k\ge 1$ denote by
\[
A^{qk}_{t'}=
\begin{pmatrix}
\delta^{\tau_{qk}(t')} & 0  \\
0 & \rho^{qk}  \\
\end{pmatrix}.
\]
Then using \eqref{taunt} we have for $\bm{v}^1,\bm{v}^2\in [\bm{w}|_{qk,t'}]$,
\begin{align*}
|\pi_sS_{R^{qk}(t')}\Phi\circ\sigma_{qk,t'}(\bm{v}^1)-\pi_sS_{R^{qk}(t')}\Phi\circ\sigma_{qk,t'}&(\bm{v}^2)|\\
&=|\pi_sS_{t'}(A^{qk}_{t'})^{-1}\Phi(\bm{v}^1)-\pi_sS_{t'}(A^{qk}_{t'})^{-1}\Phi(\bm{v}^2)|\\
&\le c_\delta \rho^{-qk}|\pi_sS_{t'}\Phi(\bm{v}^1)-\pi_sS_{t'}\Phi(\bm{v}^2)|.
\end{align*}
This tells us that for $\bm{v}^1,\bm{v}^2\in [\bm{w}|_{qk,t'}]$, if
\[
|\pi_sS_{t'}\Phi(\bm{v}^1)-\pi_sS_{t'}\Phi(\bm{v}^2)|\le c_\delta \rho^{q(k+1)},
\]
then
\[
|\pi_sS_{R^{qk}(t')}\Phi\circ\sigma_{qk,t'}(\bm{v}^1)-\pi_sS_{R^{qk}(t')}\Phi\circ\sigma_{qk,t'}(\bm{v}^2)|\le\rho^q.
\]
Therefore for $\bm{v}^1\in [\bm{w}|_{qk,t'}]$,
\begin{align*}
\{\bm{v}^2\in & [\bm{w}|_{qk,t'}]:|\pi_sS_{t'}\Phi(\bm{v}^1)-\pi_sS_{t'}\Phi(\bm{v}^2)|\le c_\delta \rho^{q(k+1)}\}\\
&\subset\{\bm{v}^2 \in  [\bm{w}|_{qk,t'}]:|\pi_sS_{R^{qk}(t')}\Phi\circ\sigma_{qk,t'}(\bm{v}^1)-\pi_sS_{R^{qk}(t')}\Phi\circ\sigma_{qk,t'}(\bm{v}^2)|\le\rho^q\}.
\end{align*}
This implies that
\begin{align}\nonumber
H_{\rho^q}(\pi_s S_{R^{qk}(t')}\Phi((\bar{\theta}')^{[\bm{w}|_{qk,t'}]}))=&\ H_{\rho^q}(\pi_sS_{R^{qk}(t')}(\Phi\circ\sigma_{qk,t'})(\bar{\theta}')_{[\bm{w}|_{qk,t'}]}))\\
\le&\ H_{c_\delta \rho^{q(k+1)}}(\pi_sS_{t'}\Phi((\overline{\theta}')_{[\bm{w}|_{qk,t'}]})),	\label{eqn:entbound}
\end{align}
where we have used the fact that, by definition,
\[
(\bar{\theta}')^{[\bm{w}|_{t',qk}]}=(\bar{\theta}'\circ \sigma_{qk,t'}^{-1})_{[\bm{w}|_{qk,t'}]}.
\]
At last we derive that for $\bar{\theta}'$-a.e. $\bm{w}\in Y_{q,t'}$,
\begin{align*}
\liminf_{n\to\infty}\frac{1}{n}\sum_{k=0}^{n-1}\frac{1}{q\log(1/\rho)}H_{c_\delta \rho^{q(k+1)}}(\pi_sS_{t'}\Phi((\overline{\theta}')_{[\bm{w}|_{qk,t'}]})) \ge E_{q,\theta}(\pi_s). \numberthis\label{eqn:lb4}
\end{align*}

The mapping $f_{q,t',s}:=\pi_sS_{t'}\Phi\colon(Y_{q,t'},d_{\rho^q})\to\mathbb{R}$ is $c_E$-Lipschitz. By \cite[Theorem 5.4]{HS12}, there exists a $\rho^q$-tree $(Y_{q,t',s},d_{\rho^q})$ and maps $Y_{q,t'}\xrightarrow{h_{q,t',s}}Y_{q,t',s}\xrightarrow{f'_{q,t',s}}\mathbb{R}$ such that $f_{q,t',s}=f'_{q,t',s}h_{q,t',s}$, where $h_{q,t',s}$ is a tree morphism and $f'_{q,t',s}$ is $C$-faithful (see \cite[Definition 5.1]{HS12}) for some constant $C$ depending only on $c_E$. Then applying \cite[Proposition 5.3]{HS12} to the $\rho^q$-tree $(Y_{q,t',s},d_{\rho^q})$ (for which $f'_{q,t',s}$ is $C$-faithful), and \cite[Lemma 3.5]{HS12}, there is a constant $C'$ depending only on $C$ and $c_\delta$, such that for all $k\ge1$,
\[
|H_{c_\delta \rho^{q(k+1)}}(f_{q,t',s}((\bar{\theta}')_{[\bm{w}|_{qk,t'}]})-H_{\rho^{q(k+1)}}(h_{q,t',s}((\bar{\theta}')_{[\bm{w}|_{qk,t'}]})|\le C'.
\]
Consequently, for $\bar{\theta}'$-a.e. $\bm{w}\in Y_{q,t'}$,
\begin{align*}
\frac{1}{q\log(1/\rho)}\liminf_{n\to\infty}\frac{1}{n}\sum_{k=0}^{n-1}H_{\rho^{q(k+1)}}(h_{q,t',s}((\bar{\theta}')_{[\bm{w}|_{qk,t'}]})\ge E_{q,\theta}(\pi_s)-O(1/q),\numberthis\label{eqn:lb5}
\end{align*}
where the constant in $O(1/q)$ only depends on $\rho$ and $C'$. Then, by \cite[Theorem 4.4]{HS12}, it follows that
\[
\underline{\dim}_H h_{q,t',s}(\bar{\theta}') \ge E_{q,\theta}(\pi_s)-O(1/q).
\]
Since $f'_{q,t',s}$ is $C$-faithful and $f'_{q,t',s}h_{q,t',s}=f_{q,s,t'}=\pi_sS_{t'}\Phi$, we get from \cite[Proposition 5.2]{HS12} that
\[
\underline{\dim}_H \pi_sS_{t'}\Phi(\bar{\theta}')\ge E_{q,\theta}(\pi_s)-O(1/q).
\]
Finally, as
\[
S_{t'}\Phi(\bar{\theta}')=S_{t'}\Phi(\bar{\theta}^+_{R^{l_s}(t),\sigma_{l_s,t}\bm{v}^-,\bm{\omega}_{\bm{v}^+|_{l_s,t}}})=S_{t}\Phi((\bar{\theta}^+_{t,\bm{v}^-,\bm{\omega}})^{[\bm{v}^+|_{l_s,t}]}),
\]
we deduce that for $\mathbb{Q}_\theta^-$-a.e. $(t,\bm{v}^-,\bm{\omega})\in \mathbb{T}\times\Sigma_{a,b}^-\times\Omega_{a,b}$, for all $s\in \mathbb{R}$ and $q\ge 1$,
\[
\underline{\dim}_H \pi_sS_{t}\Phi(\bar{\theta}^+_{t,\bm{v}^-,\bm{\omega}})\ge E_{q,\theta}(\pi_s)-O(1/q).
\]
\end{proof}	

\subsection{Projection theorems}\label{sec:projthm}
\begin{theorem}\label{thm:last}
The limit
\[
E_{\theta}(\pi_s)=\lim_{q\to\infty}E_{q,\theta}(\pi_s)
\]
exists for every $s\in\mathbb{R}$, and the mapping $E_{\theta}:\pi_s \to E_{\theta}(\pi_s)\in[0,2]$ is lower-semicontinuous. Moreover:
\begin{enumerate}[i)]
\item for a fixed $s\in\mathbb{R}$, for $\mathbb{Q}_\theta^-$-a.e. $(t,\bm{v}^-,\bm{\omega})\in \mathbb{T}\times\Sigma_{a,b}^-\times\Omega_{a,b}$,
\[
\dim_e \pi_sS_t\Phi(\bar{\theta}^+_{t,\bm{v}^-,\bm{\omega}})=\underline{\dim}_H \pi_sS_t\Phi(\bar{\theta}^+_{t,\bm{v}^-,\bm{\omega}})=E_{\theta}(\pi_s);
\]
\item for $\mathbb{Q}_\theta^-$-a.e. $(t,\bm{v}^-,\bm{\omega})\in \mathbb{T}\times\Sigma_{a,b}^-\times\Omega_{a,b}$,
\[
\underline{\dim}_H \pi_sS_t\Phi(\bar{\theta}^+_{t,\bm{v}^-,\bm{\omega}})\ge E_{\theta}(\pi_s)\;\;\;\text{ for all }s\in\mathbb{R}.
\]
\end{enumerate}
\end{theorem}
\begin{proof}
Fix $s\in\mathbb{R}$. We now know that for $\mathbb{Q}_\theta^-$-a.e. $(t,\bm{v}^-,\bm{\omega})\in \mathbb{T}\times\Sigma_{a,b}^-\times\Omega_{a,b}$,
\begin{align*}
\dim_e(\pi_sS_t\Phi(\bar{\theta}^+_{t,\bm{v}^-,\bm{\omega}}))&\ge\underline{\dim}_H(\pi_sS_t\Phi(\bar{\theta}^+_{t,\bm{v}^-,\bm{\omega}}))\\
&\ge E_{q,\theta}(\pi_s)-O(1/q).
\end{align*}
Since \eqref{eqn:lb6} is true for all $q\ge1$, we have that  for $\mathbb{Q}_\theta^-$-a.e. $(t,\bm{v}^-,\bm{\omega})\in \mathbb{T}\times\Sigma_{a,b}^-\times\Omega_{a,b}$,
\[
\dim_e\pi_sS_t\Phi(\bar{\theta}^+_{t,\bm{v}^-,\bm{\omega}})\ge\limsup_{q\to\infty}E_{q,\theta}(\pi_s).
\]
On the other hand, by Fatou's lemma we have
\begin{align*}
\int_{\mathbb{T}\times\Sigma_{a,b}^-\times\Omega_{a,b}} \dim_e\pi_sS_t\Phi(\bar{\theta}^+_{t,\bm{v}^-,\bm{\omega}}) \, \mathbb{Q}_\theta^-(\mathrm{d} (t,\bm{v}^-,\bm{\omega}))\le\liminf_{q\to\infty}E_{q,\theta}(\pi_s).
\end{align*}
Therefore
\begin{align*}
\limsup_{q\to\infty}E_{q,\theta}(\pi_s)
&\le\int_{\mathbb{T}\times\Sigma_{a,b}^-\times\Omega_{a,b}} \dim_e\pi_sS_t\Phi(\bar{\theta}^+_{t,\bm{v}^-,\bm{\omega}}) \, \mathbb{Q}_\theta^-(\mathrm{d} (t,\bm{v}^-,\bm{\omega}))\\
&\le\liminf_{q\to\infty}E_{q,\theta}(\pi_s).
\end{align*}
This shows that $\lim_{q\to\infty}E_{q,\theta}(\pi_s)$ exists for all $s\in\mathbb{R}$. Then $i),ii)$ follow immediately.
		
For the lower semicontinuity of $E_\theta$, fix $\pi_s$ and $\epsilon>0$. Since $E_{q,\theta}(\pi_s)\to E_\theta(\pi_s)$, we may find a $q$ large enough such that $E_{q,\theta}(\pi_s)>E_\theta(\pi_s)-\epsilon/3$ and that the error term in \eqref{eqn:lb6} satisfies $O(1/q)<\epsilon/3$. Then, since $E_{q,\theta}$ is lower semicontinuous, one can find an open neighbourhood $U(\pi_s)$ of $\pi_s$ such that for all $\pi\in U(\pi_s)$,
\[
E_{q,\theta}(\pi)>E_{q,\theta}(\pi_s)-\epsilon/3>E_\theta(\pi_s)-2\epsilon/3.
\]
Then by \eqref{eqn:lb6} we have for $\mathbb{Q}_\theta^-$-a.e. $(t,\bm{v}^-,\bm{\omega})\in \mathbb{T}\times\Sigma_{a,b}^-\times\Omega_{a,b}$, for all $\pi\in U(\pi_s)$,
\[
\underline{\dim}_H \pi S_t\Phi(\bar{\theta}^+_{t,\bm{v}^-,\bm{\omega}})\ge E_{q,\theta}(\pi)-O(1/q)>E_{q,\theta}(\pi_s)-\epsilon/3-O(1/q)>E_\theta(\pi_s)-\epsilon.
\]
By i) this implies that $E_\theta(\pi)>E_\theta(\pi_s)-\epsilon$ for all $\pi\in U(\pi_s)$, hence the lower semicontinuity of $E_\theta$.
\end{proof}
	
For each $(x,y)\in\mathbb{R}^2$ we have that
\begin{equation}\label{Sst}
\pi_sS_t(x,y)=\pi_s(\delta^{t}x,y)=\delta^{s+t}x\pm y=\pi_{s+t}(x,y).
\end{equation}
Recall that by \eqref{ed34} and Theorem \ref{ed2} both $\Phi(\bar{\mu}\times \bar{\nu})$ and $\Phi(\bar{\theta}^+_{t,\bm{v}^-,\bm{\omega}})$ are exact-dimensional with the same dimension
\[
D_{\bar{\mu},\bar{\nu},\Phi}=D_{\mu,V,\bm{I}}+D_{\nu,W,\bm{J}}=\frac{h_{\mu}-h_V-h_{\mu,V,\bm{I}}}{-\log \delta}+\frac{h_{\nu}-h_W-h_{\nu,W,\bm{J}}}{-\log \rho}.
\]
\begin{corollary}\label{com}
$\mathbb{P}_{a,b}$-a.s. conditioned on $\|\widetilde{\mu}\|\cdot\|\widetilde{\nu}\|>0$, for all $s\in\mathbb{R}$,
\[
\dim\pi_s\Phi(\bar{\mu}\times\bar{\nu})=\min\{1,D_{\bar{\mu},\bar{\nu},\Phi}\}.
\]
\end{corollary}

\begin{proof}
We shall fix an instance of $\bm{\omega}=(\bm{\omega}^a,\bm{\omega}^b)$ w.r.t. $\mathbb{P}_{a,b}$ such that the cascade measures $\bar{\mu}_{\bm{\omega}^a}\times\bar{\nu}_{\bm{\omega}^b}$ and $\bar{\theta}_{t,\bm{v}^-,\bm{\omega}}$ for $\theta$-a.e. $(t,\bm{v}^-)\in \mathbb{T}\times\Sigma_{a,b}^-$ are well-defined. The rest of the arguments are deterministic and we shall omit the notation $\bm{\omega}$. 

Note that, since $\theta$ is $T$-invariant, $\theta$ and hence $\theta^-$ projected to $\mathbb{T}$ is $R$-invariant. By unique ergodicity of irrational rotations, we get $\theta^-$ projected to $\mathbb{T}$ is equal to the Lebesgue measure $\ell$, hence we may write
\[
\theta^-(\mathrm{d}(t,\bm{v}^-))=\theta^-_t(\mathrm{d}\bm{v}^-)\ell(\mathrm{d} t),
\]
where we denote by $\theta^-_t$ the disintegration of $\theta^-$ w.r.t. $\ell$. 

By Hunt and Kaloshin's analogue of Marstrand's projection theorem for measures (\cite[Theorem 4.1]{HuntKaloshin}), we have that, for $\ell$-a.e. $t\in \mathbb{T}$, for $\theta^-_t$-a.e. $\bm{v}^-\in\Sigma_{a,b}^-$, for the measure $\bar{\theta}_{t,\bm{v}^-}$, for Lebesgue almost every $s\in\mathbb{R}$,
\[
\underline{\dim}_H\pi_{s}S_t\Phi(\bar{\theta}_{t,\bm{v}^-})=\min\{1,\underline{\dim}_HS_t\Phi(\bar{\theta}_{t,\bm{v}^-})\}.
\]
Together with Theorem \ref{thm:last} i) and Theorem \ref{ed2}, we obtain that for Lebesgue almost every $s\in\mathbb{R}$, $E_\theta(\pi_s)=\min\{1,D_{\bar{\mu},\bar{\nu},\Phi}\}$.
	
On the other hand, by \eqref{PQT3} we have for $B\in\mathcal{B}(\mathbb{R})$
\begin{align*}
\pi_s\Phi(\bar{\mu}\times\bar{\nu})(B)=&\ \int_{\mathcal{M}(X)}\int_{\mathbb{T}}\int_{\Sigma_{a,b}^-} \pi_s\Phi(\bar{\theta}_{t,\bm{v}^-})(B)\,\theta^-_t(\mathrm{d}\bm{v}^-)\ell(\mathrm{d}t)\Theta(\mathrm{d}\theta).
\end{align*}
By \eqref{Sst} and Theorem \ref{thm:last} ii) this implies that
\begin{align*}
\underline{\dim}_H \pi_s\Phi(\bar{\mu}\times\bar{\nu}) \ge&\ \mathrm{ess inf}_{\theta\sim\Theta,t\sim\ell,\bm{v}^-\sim \theta^-_t}\underline{\dim}_H \pi_s\Phi(\bar{\theta}_{t,\bm{v}^-})\\
=&\ \mathrm{ess inf}_{\theta\sim\Theta,t\sim\ell,\bm{v}^-\sim \theta^-_t}\underline{\dim}_H \pi_{s-t}S_t\Phi(\bar{\theta}_{t,\bm{v}^-})\\
\ge&\ \mathrm{ess inf}_{\theta\sim\Theta, t\sim\ell} E_\theta(\pi_{s-t})\\
=&\ \min\{1,D_{\bar{\mu},\bar{\nu},\Phi}\}.
\end{align*}
But, by the exact-dimensionality of $\Phi(\bar{\mu}\times \bar{\nu})$, $\min\{1,D_{\bar{\mu},\bar{\nu},\Phi}\}$ is also an upper bound for $\overline{\dim}_B \pi_s\Phi(\bar{\mu}\times\bar{\nu})$, hence the conclusion.
\end{proof}

\subsection{Extend to invariant measures}

So far we have proved Theorem \ref{mthm} for Mandelbrot cascades acting on ergodic measures $\mu^+$ and $\nu^+$. Now we need to extend Theorem \ref{thm1} to invariant measures.

Let $\mu^+$ be $\sigma$-invariant on $\Sigma_a^+$. Let $\zeta$ be a $\mu^+$-measurable partition generates a $\sigma$-algebra $\widehat{\zeta}$ that is equal to the $\sigma$-algebra $\mathcal{I}_a=\{B\in\mathcal{B}(\Sigma_a):\sigma^{-1}(B)=B\}$ modulo sets of zero $\mu$-measure. Then the conditional measures $\mu^\zeta_{\bm{i}^+}$ for $\mu^+$-a.e. $\bm{i}^+\in\Sigma_a^+$ form an ergodic decomposition of $\mu^+$. We have the following lemma which can be deduced from \cite{FH09}.

\begin{lemma}\label{dinv}
\[
\overline{\dim}_H \Phi_{\bm{I}}(\mu^+)=\mathrm{ess sup}_{\mu^+} \dim \Phi_{\bm{I}}(\mu^\zeta_{\bm{i}^+}) \text{ and } \underline{\dim}_H \Phi_{\bm{I}}(\mu^+)=\mathrm{ess inf}_{\mu^+} \dim \Phi_{\bm{I}}(\mu^\zeta_{\bm{i}^+})
\]
\end{lemma}

\begin{proof}
By \cite[Theorem 2.8]{FH09} we have that, on $(\Sigma_a^+,d_\delta^+)$, for $\mu^+$-a.e. $\bm{i}^+\in\Sigma_a^+$,
\[
\dim_{loc}(\Phi_{\bm{I}}(\mu^+),\Phi_{\bm{I}}(\bm{i}^+)):=\lim_{r\to \infty} \frac{\log \mu^+\circ \Phi_{\bm{I}}^{-1}(B(\Phi_{\bm{I}}(\bm{i}^+),r))}{\log r}=\frac{h_{\Phi_{\bm{I}}}(\sigma,\mu^+,\bm{i}^+)}{-\log \delta},
\]
where $h_{\Phi_{\bm{I}}}(\sigma,\mu^+,\bm{i}^+)$ is the local projection entropy of $\mu^+$ at $\bm{i}^+$ under $\Phi_{\bm{I}}$ (see \cite[Definition 2.1]{FH09}). Denote by $h_{\Phi_{\bm{I}}}(\sigma,\mu^+)=\int_{\Sigma_a} h_{\Phi_{\bm{I}}}(\sigma,\mu^+,\bm{i}^+) \mu^+(\mathrm{d}\bm{i}^+)$. If $\mu^+$ is also ergodic then $\mu^+$ is exact-dimensional with dimension $\frac{h_{\Phi_{\bm{I}}}(\sigma,\mu^+)}{-\log \delta}$. By \cite[Lemma 3.3]{FH09} applying $\eta$ to $\zeta$ we have for $\mu^+$-a.e. $\bm{i}^+$,
\[
h_{\Phi_{\bm{I}}}(\sigma,\mu^+,\bm{i}^+) =h_{\Phi_{\bm{I}}}(\sigma,\mu^\zeta_{\bm{i}^+}).
\]
This implies that
\begin{align*}
\overline{\dim}_H \Phi_{\bm{I}}(\mu^+) & = \mathrm{ess sup}_{\mu^+} \dim_{loc}(\Phi_{\bm{I}}(\mu^+),\Phi_{\bm{I}}(\bm{i}^+)) \\
& = \mathrm{ess sup}_{\mu^+} \frac{h_{\Phi_{\bm{I}}}(\sigma,\mu^+,\bm{i}^+)}{-\log \delta} \\
& = \mathrm{ess sup}_{\mu^+} \frac{h_{\Phi_{\bm{I}}}(\sigma,\mu^\zeta_{\bm{i}^+})}{-\log \delta} \\
& = \mathrm{ess sup}_{\mu^+} \dim \Phi_{\bm{I}}(\mu^\zeta_{\bm{i}^+})
\end{align*}
and
\begin{align*}
\underline{\dim}_H \Phi_{\bm{I}}(\mu^+) & = \mathrm{ess inf}_{\mu^+} \dim_{loc}(\Phi_{\bm{I}}(\mu^+),\Phi_{\bm{I}}(\bm{i}^+)) \\
& = \mathrm{ess inf}_{\mu^+} \frac{h_{\Phi_{\bm{I}}}(\sigma,\mu^+,\bm{i}^+)}{-\log \delta} \\
& = \mathrm{ess inf}_{\mu^+} \frac{h_{\Phi_{\bm{I}}}(\sigma,\mu^\zeta_{\bm{i}^+})}{-\log \delta} \\
& = \mathrm{ess inf}_{\mu^+} \dim \Phi_{\bm{I}}(\mu^\zeta_{\bm{i}^+})
\end{align*}
\end{proof}

By Lemma \ref{dinv} and Corollary \ref{com} we can easily deduce Theorem \ref{thm1}.

\section{Application to Mandelbrot percolations}\label{sec:perc}

Recall $\Sigma_{a,p}^+$, the Mandelbrot $p$-percolation on $\Sigma_a^+$, that is the closed random subset
\[
\Sigma_{a,p}^+=\{\bm{i}^+\in\Sigma_a^+: V_{i_1}V_{i_1i_2}\cdots V_{i_1\cdots i_n} >0 \text{ for all } n\ge 1\}
\]
for the random variable $V$ of the form
\begin{equation*}\label{Vp}
V=\left\{\begin{array}{ll} 1/p, & \text{with probability } p;\\ 0, & \text{with probability } 1-p.\end{array}  \right.
\end{equation*}
Recall the IFS $\bm{I}=\{\delta x+t_i\}_{i=1}^a$ and its canonical mapping $\Phi_{\bm{I}}$. For a subset $X\subset\Sigma^+_a$ denote by $X_{\bm{I}}=\Phi_{\bm{I}}(X)$ its image through $\Phi_{\bm{I}}$, in particular we have the attractor of the IFS $A_{\bm{I}}=\Phi_{\bm{I}}(\Sigma_a)$. Denote by $X_p=X\cap \Sigma^+_{a,p}$ and $X_{\bm{I},p}=\Phi_{\bm{I}}(X_p)$.

In the case when $X$ is a $\sigma$-invariant closed subset, by Theorem \ref{ed22} we have the following lower bound of the Hausdorff dimension of $X_{\bm{I},p}$: almost surely conditioned on $X_p\neq \emptyset$,
\begin{equation}\label{vpc}
\underline{\dim}_H X_{\bm{I},p}\ge \sup_{\mu^+ \in \mathcal{M}_e(X)} \frac{h_\mu-h_{\mu,p,\bm{I}}}{-\log \delta}-\frac{\log p}{\log \delta},
\end{equation}
where $\mathcal{M}_e(X)$ is the set of all $\sigma$-ergodic measures supported by $X$ and we denote by $h_{\mu,p,\bm{I}}=h_{\mu,V,\bm{I}}$ for $V$ in \eqref{Vp}. By the proof of \cite[Theorem 8.1]{FH09}, one can find an ergodic measure $\mu$ supported by a $\sigma^k$-invariant subset $X^k$ of $X$ for some large integer $k$ such that the $k$-iterated IFS $\bm{I}^k$ restricted on $X^k$ satisfies the strong separation condition, hence $h_{\mu,p,\bm{I}^k}=0$; and its dimension through the canonical mapping $\Phi_{\bm{I}}$ can be arbitrarily close to $\overline{\dim}_B X_{\bm{I}}$. Therefore, by \eqref{vpc} we obtain a lower bound for $\underline{\dim}_H X_{\bm{I},p}$:
\begin{equation}\label{vpc2}
\underline{\dim}_H X_{\bm{I},p}\ge \overline{\dim}_B X_{\bm{I}}- \frac{\log p}{\log \delta}.
\end{equation}

One would expect that the supremum in \eqref{vpc} reaches the dimension, which is indeed the case when $X=\Sigma_a$ and $p=1$: By \cite[Theorem 2.13]{FH09} we have the variational principle that
\[
\underline{\dim}_H A_{\bm{I}}= \sup_{\mu\in \mathcal{M}_e(\Sigma^+_a)} \frac{h_\mu-h_{\mu,\bm{I}}}{-\log \delta},
\]
where recall that $h_{\mu,\bm{I}}=h_{\mu,1,\bm{I}}$, or in other words, the conditional entropy in the deterministic case when there is no percolation. However, it seems hard to obtain a sharp upper bound of the Hausdorff dimension of self-similar sets with percolations when there are non-trivial overlaps.

Recall for $X\subset \Sigma_a$ and $n\ge 1$,
\[
t_{X,\bm{I},n}=\sup_{x\in \mathbb{R}}\#\{u \in X_n: \Phi_{\bm{I}}([u])\cap B(x,\delta^n)\neq \emptyset\},
\]
and
\[
\gamma_{X,\bm{I}}=\limsup_{n\to\infty} \frac{\log t_{X,\bm{I},n}}{-n\log \delta}.
\]
In terms of $\gamma_{X,\bm{I}}$ we have the following upper bound for $\overline{\dim}_B X_{\bm{I},p}$.

\begin{lemma}\label{mpub}
For $X\subset \Sigma_a^+$ we have almost surely,
\[
\overline{\dim}_B X_{\bm{I},p}\le \overline{\dim}_B X_{\bm{I}} +\gamma_{X,\bm{I}}-\frac{\log p}{\log \delta}.
\]
\end{lemma}
\begin{proof}
For $n\ge 1$ let $\{B(x_{n,i},\delta^n)\}_{i=1}^{N_{X,\bm{I},n}}$ be an optimal covering of $\Phi_{\bm{I}}(X)$, i.e., $N_{X,\bm{I},n}$ is the smallest such number. We have that
\[
\overline{\dim}_B X_{\bm{I}}=\limsup_{n\to\infty} \frac{\log N_{X,\bm{I},n}}{-n\log \delta}.
\]
For $u\in \{1,\dots,a\}^n$ we have that
\[
\mathbb{E}_{\Sigma^+_{a,p}\neq \emptyset}(\mathbf{1}_{\{[u]\cap \Sigma^+_{a,p}\neq \emptyset\}})= \frac{1}{\mathbb{P}(\Sigma_{a,p}\neq \emptyset)}\mathbb{E}(\mathbf{1}_{\{Q_u^V\neq 0\}}\mathbf{1}_{\{\Sigma^{+,(u)}_{a,p}\neq \emptyset\}})=p^n,
\]
where $\Sigma^{+,(u)}_{a,p}$ denotes the percolation set obtained from the i.i.d. sequence $\{V_{uv}\}_{v\in\Sigma_{a}^{+,*}}$, which is independent of $Q_u^V$ and has the same law as $\Sigma_{a,p}^+$. By the definition of $t_{X_{\bm{I}},n}$ we have that every $B(x_{n,i},\delta^n)$ is intersected with at most $t_{X,\bm{I},n}$ many $\{\Phi_{\bm{I}}([u]):u\in X_n\}$. Hence we have
\[
\mathbb{E}_{\Sigma^+_{a,p}\neq \emptyset}(\sum_{i=1}^{N_{X,\bm{I},n}}\sum_{u\in X_n}\mathbf{1}_{\{\Phi_{\bm{I}}([u])\cap B(x_{n,i},\delta^n)\neq \emptyset,[u]\cap \Sigma^+_{a,p}\neq \emptyset\}}) \le t_{X,\bm{I},n} \cdot p^n\cdot N_{X,\bm{I},n}.
\]
For $\epsilon>0$, by Markov's inequality we have
\begin{align*}
\mathbb{P}_{\Sigma_{a,p}\neq \emptyset}\left(\frac{\sum_{i=1}^{N_{X,\bm{I},n}}\sum_{u\in X_n}\mathbf{1}_{\{\Phi_{\bm{I}}([u])\cap B(x_{n,i},\delta^n)\neq \emptyset,[u]\cap \Sigma^+_{a,p}\neq \emptyset\}}}{ t_{X,\bm{I},n} \cdot p^n\cdot N_{X,\bm{I},n}}>(1+\epsilon)^n \right)\le (1+\epsilon)^{-n}.
\end{align*}
Then by Borel-Cantelli, $\mathbb{P}_{\Sigma^+_{a,p}\neq \emptyset}$-almost surely, for all $n$ large enough,
\begin{align*}
\sum_{i=1}^{N_{X,\bm{I},n}}\mathbf{1}_{\{\exists u\in X_n: \Phi_{\bm{I}}([u])\cap B(x_{n,i},\delta^n)\neq \emptyset,[u]\cap \Sigma^+_{a,p}\neq \emptyset\}} &\le \sum_{i=1}^{N_{X,\bm{I},n}}\sum_{u\in X_n}\mathbf{1}_{\{\Phi_{\bm{I}}([u])\cap B(x_{n,i},\delta^n)\neq \emptyset,[u]\cap \Sigma^+_{a,p}\neq \emptyset\}}\\
&\le (1+\epsilon)^n\cdot  t_{X,\bm{I},n} \cdot p^n\cdot N_{X,\bm{I},n}.
\end{align*}
Hence 
\[
\limsup_{n\to\infty} \frac{\log\sum_{i=1}^{N_{X,\bm{I},n}}\mathbf{1}_{\{\exists u\in X_n: \Phi_{\bm{I}}([u])\cap B(x_{n,i},\delta^n)\neq \emptyset,[u]\cap \Sigma^+_{a,p}\neq \emptyset\}} }{-n\log \delta} \le \overline{\dim}_B X_{\bm{I}} + \gamma_{X,\bm{I}}- \frac{\log (1+\epsilon)p}{\log \delta}.
\]
Taking a countable sequence of $\epsilon$ tending to $0$, we obtain that $\mathbb{P}_{\Sigma^+_{a,p}\neq \emptyset}$-almost surely,
\[
\overline{\dim}_B X_{\bm{I},p}\le \overline{\dim}_B X_{\bm{I}} + \gamma_{X,\bm{I}}-\frac{\log p}{\log \delta}.
\]
\end{proof}

\begin{example}\label{example1}
The upper bound in Lemma \ref{mpub} and the lower bound \eqref{vpc2} are not optimal, here is an example: consider $\bm{I}=\{f_i(x)\}_{i=1}^3$, where
\[
f_1(x)=f_2(x)=2^{-1}x,\ f_3(x)=2^{-1}x+2^{-1}.
\]
Then $f_1$ and $f_2$ are exactly overlapping and it is easy to see that $\gamma_{\Sigma_3,\bm{I}}=1$. Therefore the upper bound in Lemma \ref{mpub} is
\[
1+1+\frac{\log p}{\log 2}.
\]
Consider a $p$-Mandelbrot percolation $\Sigma^+_{a,p}$ on $\Sigma_3$, then its image $\Phi_{\bm{I}}(\Sigma^+_{a,p})$ becomes a $(p',p)$-fractal percolation on the dyadic decomposition of $[0,1]$, where $p'=1-(1-p)^2=2p-p^2$. Its Hausdorff dimension is equal to (see \cite{FalGri92} for example)
\[
1+\frac{(\log p'+\log p)/2}{\log 2}=1+\frac{\log\sqrt{2-p}}{\log 2}+\frac{\log p}{\log 2},
\]
and clearly
\[
0<\frac{\log\sqrt{2-p}}{\log 2}<\frac{1}{2}<1.
\]
\end{example}

On the other hand, when $\gamma_{X,\bm{I}}=0$, by \eqref{vpc2} we have

\begin{corollary}\label{bdp}
If $\gamma_{X,\bm{I}}=0$, then $\mathbb{P}_{\Sigma^+_{a,p}\neq \emptyset}$-almost surely
\[
\overline{\dim}_B X_{\bm{I},p}=\underline{\dim}_H X_{\bm{I},p}= \dim X_{\bm{I}}-\frac{\log p}{\log \delta}.
\]
\end{corollary}

Together with Theorem \ref{mthm} we obtain Corollary \ref{co2}.

\end{document}